\documentclass[reqno]{amsart}
\usepackage{setspace} 
\usepackage{geometry} 
\usepackage[dvipsnames]{xcolor}
\usepackage{amsfonts}
\usepackage{amsmath}
\usepackage{graphicx}
\usepackage{mathrsfs}
\usepackage{amssymb}
\usepackage{amsthm}
\usepackage{tikz}
\usepackage{tikz-cd}
\usepackage{listings}
\usepackage{color}
\usepackage{enumitem}
\usepackage{imakeidx}
\usepackage[colorinlistoftodos]{todonotes}
\usepackage{upgreek}
\usepackage[pagebackref]{hyperref} 
\hypersetup{colorlinks=true, linkcolor=., citecolor=ForestGreen, urlcolor=blue}
\usepackage[capitalise, nameinlink]{cleveref}
\crefformat{equation}{(#2#1#3)} 
\usepackage{mdframed}
\usepackage{pgfplots}
\usepackage{manfnt}
\usepackage[new]{old-arrows}
\usepackage{changepage}
\usepackage{stmaryrd}
\usepackage{appendix}
\usepackage[normalem]{ulem}
\usepackage{ytableau}
\usepackage{blkarray}
\usepackage{mathtools}
\usepackage{verbatim}
\usepackage{wasysym}

\newcommand\linkcolor{BrickRed}
\let\crefinner=\cref
\newcommand\mycref[1]{{\color{\linkcolor}\crefinner{#1}}}
\def\cref{\mycref}
\let\refinner=\ref
\newcommand\myref[1]{{\color{\linkcolor}\refinner{#1}}}
\def\ref{\myref}

\DeclareMathAlphabet{\matheuler}{U}{eus}{m}{n}

\setlist[enumerate]{align=right, label=\llap{\bfseries{(\arabic*)}}}

\newenvironment{alphabetize}{\begin{enumerate}[label=\llap{\bfseries{(\alph*)}}]}{\end{enumerate}}

\pgfplotsset{width=7cm,compat=1.9}

\makeindex[intoc]

\definecolor{codegreen}{rgb}{0,0.6,0}
\definecolor{codegray}{rgb}{0.5,0.5,0.5}
\definecolor{codepurple}{rgb}{0.58,0,0.82}
\definecolor{backcolour}{rgb}{0.95,0.95,0.92}

\lstdefinestyle{mystyle}{
	backgroundcolor=\color{backcolour},   
	commentstyle=\color{codegreen},
	keywordstyle=\color{magenta},
	numberstyle=\tiny\color{codegray},
	stringstyle=\color{codepurple},
	basicstyle=\footnotesize,
	breakatwhitespace=false,         
	breaklines=true,                 
	captionpos=b,                    
	keepspaces=true,                 
	numbers=left,                    
	numbersep=5pt,                  
	showspaces=false,                
	showstringspaces=false,
	showtabs=false,                  
	tabsize=2
}

\allowdisplaybreaks

\tikzcdset{
  cells={font=\everymath\expandafter{\the\everymath\displaystyle}},
}

\newcommand{\MAT}[7]{\begin{pmatrix}{#1}&{#2}&{#3}&{#4}\\{#5}&{#6}&{#7}&\MATHelper}
\newcommand{\MATHelper}[9]{{#1}\\{#2}&{#3}&{#4}&{#5}\\{#6}&{#7}&{#8}&{#9}\end{pmatrix}}
\newcommand{\Matrix}[7]{\begin{pmatrix}{#1}&{#2}&{#3}&{#4}&{#5}\\{#6}&{#7}&\MatrixHelper}
\newcommand{\MatrixHelper}[9]{{#1}&{#2}&{#3}\\{#4}&{#5}&{#6}&{#7}&{#8}\\{#9}&\MatrixHelperii}
\newcommand{\MatrixHelperii}[9]{{#1}&{#2}&{#3}&{#4}\\{#5}&{#6}&{#7}&{#8}&{#9}\end{pmatrix}}

\newcommand{\dual}[1]{#1^\vee}
\newcommand{\pdual}[1]{\p{#1}^\vee}

\newcommand{\units}[1]{#1^{\times}}

\newcommand{\inv}[2][1]{#2^{-#1}}
\newcommand{\pinv}[2][1]{\inv[#1]{\parens{#2}}}
\newcommand{\sqinv}[2][1]{\inv[#1]{\sq{#2}}}

\newcommand{\ctsHom}{\Hom_{\mrm{cts}}}

\DeclareMathOperator{\Hom}{Hom}

\DeclareMathOperator{\Isom}{Isom}

\DeclareMathOperator{\sHom}{\mathscr{H\mkern-7mu}\textit{om}}
\DeclareMathOperator{\sExt}{\mathscr{E\mkern-4mu}\textit{xt}}

\DeclareMathOperator{\Sym}{Sym}

\DeclareMathOperator{\coker}{coker}
\DeclareMathOperator{\Ext}{Ext}
\renewcommand{\hom}{\operatorname H}

\DeclareMathOperator{\Stab}{Stab}
\DeclareMathOperator{\im}{im}
\DeclareMathOperator{\ob}{ob}

\newcommand{\zmod}[1]{\mathbb Z/#1\mathbb Z}

\DeclareMathOperator{\Aut}{Aut}

\DeclareMathOperator{\Gal}{Gal}

\DeclareMathOperator{\Char}{char}

\DeclareMathOperator{\Frac}{Frac}
\DeclareMathOperator{\rank}{rank}

\DeclareMathOperator{\Pic}{Pic}

\DeclareMathOperator{\Sel}{Sel}

\DeclareFontFamily{U}{wncy}{}
\DeclareFontShape{U}{wncy}{m}{n}{<->wncyr10}{}
\DeclareSymbolFont{mcy}{U}{wncy}{m}{n}
\DeclareMathSymbol{\Sha}{\mathord}{mcy}{"58} 

\makeatletter
\newcommand{\bigperp}{%
  \mathop{\mathpalette\bigp@rp\relax}%
  \displaylimits
}

\newcommand{\bigp@rp}[2]{%
  \vcenter{
    \m@th\hbox{\scalebox{\ifx#1\displaystyle2.1\else1.5\fi}{$#1\perp$}}
  }%
}
\makeatother

\newcommand{\openset}{\overset{\text{open}}\subset}

\newcommand{\dm}[0]{\d m}

\renewcommand{\d}{\mathrm d} 

\DeclareMathOperator{\Bun}{Bun}

\newcommand{\GL}{\operatorname{GL}}

\newcommand{\PGL}{\operatorname{PGL}}

\newcommand{\Jac}{\operatorname{Jac}}

\newcommand{\fppf}[1]{#1_{\mathrm{fppf}}} 

\DeclareMathOperator{\spec}{Spec}
\DeclareMathOperator{\Proj}{Proj}
\DeclareMathOperator{\rProj}{\mbf{Proj}} 

\DeclareMathOperator{\Div}{Div}

\DeclareMathOperator{\Ht}{ht}

\newcommand{\abs}[1]{\left|#1\right|}

\let\templim\lim
\renewcommand{\lim}{\templim\limits}

\newcommand{\bits}{\{0,1\}}

\newcommand{\Ith}[2]{{#1}^{\left(#2\right)}}
\newcommand{\ith}[1]{\Ith{#1}i}

\newcommand{\commsquare}[8]{
	\begin{tikzcd}[ampersand replacement=\&]
	{\displaystyle #1}\ar[r, "{#2}"]\ar[d, "{#4}" left]\&{\displaystyle #3}\ar[d, "{#5}" right]\\
	{\displaystyle #6}\ar[r, "{#7}" above]\&{\displaystyle #8}
	\end{tikzcd}
}

\newcommand{\mapover}[6]{
    \begin{tikzcd}[ampersand replacement=\&]
        {\displaystyle #1}\ar[rr, "{#2}" above]\ar[dr, "{#4}"']\&\&{#3}\ar[dl, "{#5}"]\\
        \&{\displaystyle #6}
    \end{tikzcd}
}
\newcommand{\compdiag}[6]{ 
    \begin{tikzcd}[ampersand replacement=\&]
        {\displaystyle #1}\ar[r, "{#2}" below]\ar[rr, bend left, "{#6}" above]\&{\displaystyle #3}\ar[r, "{#4}" below]\&{\displaystyle #5}
    \end{tikzcd}
}

\newcommand\homses[4]{
    \begin{tikzcd}[ampersand replacement=\&]
        0\ar[r]\&{\displaystyle #1}\ar[r, "{#2}"]\ar[r]\&{\displaystyle #3}\ar[r, "{#4}"]\&
        \homsesHELPER
}
\newcommand\homsesHELPER[9]{
        {\displaystyle #1}\ar[r]\& 0\\
        0\ar[r]\&{\displaystyle #5}\ar[r, "{#6}"]\&{\displaystyle #7}\ar[r, "{#8}"]\&{\displaystyle #9}\ar[r]\& 0
        \ar[from=lllu, to=lll, "{#2}"]\ar[from=llu, to=ll, "{#3}"]\ar[from=lu, to=l, "{#4}"]
    \end{tikzcd}
}

\newcommand\frontthreesimplex[1]{
    \begin{tikzcd}[every arrow/.append style={dash}, ampersand replacement=\&, row sep = large]
        \&{\displaystyle#1}
        \frontthreesimplexHELPER
}
\newcommand\frontthreesimplexHELPER[9]{
        \ar[ddl, "{#1}"']\ar[d, "{#2}" description]\ar[ddr, "{#3}"]\\
        \&{\displaystyle#4}\ar[dl, "{#5}"]\ar[dr, "{#6}"']\\
        {\displaystyle#7}\ar[rr, "{#8}"']\&\&{\displaystyle#9}
    \end{tikzcd}
}

\renewcommand{\phi}{\varphi}

\newcommand{\mbf}{\mathbf}

\newcommand{\mbfD}{\mathbf D}

\newcommand{\msO}{\mathscr O}

\newcommand{\mf}{\mathfrak}

\newcommand{\msI}{\mathscr I}

\newcommand{\msA}{\mathscr A}
\newcommand{\msB}{\mathscr B}

\newcommand{\msD}{\mathscr D}
\newcommand{\msK}{\mathscr K}

\newcommand{\msF}{\ms F}
\newcommand{\msG}{\ms G}

\newcommand{\msL}{\ms L}
\newcommand{\msM}{\ms M}
\newcommand{\msN}{\ms N}
\newcommand{\msQ}{\ms Q}
\newcommand{\msV}{\ms V}

\newcommand{\msY}{\ms Y}

\newcommand{\ms}{\mathscr}

\newcommand{\mfQ}{\mf Q}
\newcommand{\mc}{\mathcal} 
\newcommand{\mcA}{\mc A}

\newcommand{\msT}{\ms T}

\newcommand{\mrm}{\mathrm}

\newcommand{\F}{\mathbb F}
\newcommand{\Q}{\mathbb Q}
\newcommand{\Z}{\mathbb Z}
\newcommand{\R}{\mathbb R}
\newcommand{\C}{\mathbb C}
\newcommand{\E}{\mathbb E}

\newcommand{\N}{\mathbb N}

\newcommand{\A}{\mathbb A}
\newcommand{\G}{\mathbb G}

\renewcommand{\H}{\mathbb H}

\newcommand{\eps}{\varepsilon}

\renewcommand{\tau}{\uptau}
\renewcommand{\P}{\mathbb P}
\newcommand{\msE}{\ms E}

\renewcommand{\b}{\beta}
\newcommand{\me}{\matheuler}

\newcommand{\meC}{\me C}
\newcommand{\meD}{\me D}

\newcommand{\meG}{\me G}
\newcommand{\meH}{\me H}

\newcommand{\meM}{\me M}

\newcommand{\meO}{\me O}
\newcommand{\meP}{\me P}
\newcommand{\meQ}{\me Q}

\newcommand{\meV}{\me V}

\newcommand{\floor}[1]{\left\lfloor#1\right\rfloor}
\newcommand{\ceil}[1]{\left\lceil#1\right\rceil}
\newcommand{\parens}[1]{\!\left(#1\right)}
\newcommand{\p}[1]{\!\left(#1\right)} 
\newcommand{\pfrac}[2]{\parens{\frac{#1}{#2}}}
\newcommand{\brackets}[1]{\left\{#1\right\}}
\newcommand{\bracks}[1]{\brackets{#1}}
\renewcommand{\b}{\bracks} 
\newcommand{\sqbracks}[1]{\!\left[#1\right]} 
\newcommand{\sq}{\sqbracks}
\newcommand{\angles}[1]{\left\langle#1\right\rangle}

\newcommand{\dparens}[1]{\!\left(\!\left(#1\right)\!\right)}
\newcommand\dps\dparens 

\newcommand{\mapstoo}{\longmapsto}
\newcommand{\from}{\leftarrow}

\newcommand{\xto}{\xrightarrow}

\newcommand{\too}{\longrightarrow}
\newcommand{\xtoo}{\xlongrightarrow}
\newcommand{\iso}{\xto\sim}
\newcommand{\isoo}{\xtoo\sim}
\newcommand{\xiso}[1]{\xto[#1]\sim}

\newcommand{\into}{\hookrightarrow}
\newcommand{\xinto}[1]{\overset{#1}\into}

\newcommand{\intoo}{\longhookrightarrow}

\newcommand{\onto}{\twoheadrightarrow}
\newcommand{\xonto}[1]{\overset{#1}\onto}

\newcommand{\squigto}{\rightsquigarrow}

\makeatletter
\newcommand*{\da@rightarrow}{\mathchar"0\hexnumber@\symAMSa 4B }
\newcommand*{\da@leftarrow}{\mathchar"0\hexnumber@\symAMSa 4C }
\newcommand*{\xdashrightarrow}[2][]{%
  \mathrel{%
    \mathpalette{\da@xarrow{#1}{#2}{}\da@rightarrow{\,}{}}{}%
  }%
}
\newcommand{\xdashleftarrow}[2][]{%
  \mathrel{%
    \mathpalette{\da@xarrow{#1}{#2}\da@leftarrow{}{}{\,}}{}%
  }%
}
\newcommand*{\da@xarrow}[7]{%
  \sbox0{$\ifx#7\scriptstyle\scriptscriptstyle\else\scriptstyle\fi#5#1#6\m@th$}%
  \sbox2{$\ifx#7\scriptstyle\scriptscriptstyle\else\scriptstyle\fi#5#2#6\m@th$}%
  \sbox4{$#7\dabar@\m@th$}%
  \dimen@=\wd0 %
  \ifdim\wd2 >\dimen@
    \dimen@=\wd2 %
  \fi
  \count@=2 %
  \def\da@bars{\dabar@\dabar@}%
  \@whiledim\count@\wd4<\dimen@\do{%
    \advance\count@\@ne
    \expandafter\def\expandafter\da@bars\expandafter{%
      \da@bars
      \dabar@ 
    }%
  }%
  \mathrel{#3}%
  \mathrel{%
    \mathop{\da@bars}\limits
    \ifx\\#1\\%
    \else
      _{\copy0}%
    \fi
    \ifx\\#2\\%
    \else
      ^{\copy2}%
    \fi
  }%
  \mathrel{#4}%
}
\makeatother
\newcommand\xdashto\xdashrightarrow
\newcommand\xdashfrom\xdashleftarrow

\makeatletter
\providecommand{\leftsquigarrow}{%
  \mathrel{\mathpalette\reflect@squig\relax}%
}
\newcommand{\reflect@squig}[2]{%
  \reflectbox{$\m@th#1\rightsquigarrow$}%
}
\makeatother

\newcommand{\actson}{\curvearrowright}

\newcommand*{\rom}[1]{\textup{\uppercase\expandafter{\romannumeral#1}}}
\newcommand{\tbf}{\textbf}
\newcommand{\bp}[1]{\tbf{(#1)}} 
\newcommand{\Item}[1]{\item[\tbf{(#1)}]}

\newcommand{\sm}{\setminus}
\newcommand{\bs}{\backslash}

\newcommand{\wt}{\widetilde}

\newcommand{\ul}{\underline}

\renewcommand{\ast}[1]{#1^*}

\newcommand{\Twocases}[4]{
	\begin{cases}
		\hfill\displaystyle #1\hfill&\text{if }\displaystyle #2\\
		\hfill\displaystyle #3\hfill&\text{if }\displaystyle #4
	\end{cases}
}
\newcommand{\Threecases}[6]{
	\begin{cases}
		\hfill\displaystyle #1\hfill&\text{if }\displaystyle #2\\
		\hfill\displaystyle #3\hfill&\text{if }\displaystyle #4\\
		\hfill\displaystyle #5\hfill&\text{if }\displaystyle #6
	\end{cases}
}
\newcommand{\push}[1]{#1_*}

\newcommand{\pull}[1]{#1^*}

\newcommand{\by}{\times}
\newcommand{\mapdesc}[5]{
	\begin{matrix}
		\ifblank{#1}{}{\displaystyle#1:}&\displaystyle#2&\longrightarrow&\displaystyle#3\\
		&\displaystyle#4&\longmapsto&\displaystyle #5
	\end{matrix}
}

\renewcommand{\bar}{\overline}

\DeclareMathOperator{\supp}{supp}

\DeclareMathOperator{\id}{id}

\newcommand\colonequals\coloneqq
\newcommand\equalscolon\eqqcolon

\newcommand{\tand}{\,\text{ and }\,} 

\newcommand{\twhere}{\,\text{ where }\,}
\newcommand{\twith}{\,\text{ with }\,}

\newcommand{\tas}{\,\text{ as }\,}

\newcommand{\tso}{\,\text{ so }\,}

\newcommand{\tcomma}{\text{, }\,\,}

\newcommand{\tforall}{\,\text{ for all }\,}

\newcommand{\tforany}{\,\text{ for any }\,}
\newcommand{\tforsome}{\,\text{ for some }\,}

\newcommand{\tandso}{\,\text{ and so }\,}

\renewcommand\t\text 
\newcommand\ttt\texttt

\newcommand\emphasize[1]{{\color{violet}{\texttt{#1}}}} 
\newcommand{\important}[1]{\textit{#1}}
\newcommand\noteworthy\important

\newcommand{\define}[1]{\emphasize{#1}\index{#1}}

\pgfmathdeclarerandomlist{thuses}{{Thus}{Therefore}{As such}{As we wished}{By this great fortune}{Alors}{Donc}{Following this thread}{Consequently}{Thusforth}{Rejoice}{This can only mean one thing...}{Ergo}{Voila}}

\pgfmathdeclarerandomlist{fibers}{{fiber}{fibre}}

\pgfmathdeclarerandomlist{wewins}{{we win}{it's over}{all is good}{victory has been achieved}{we've crossed the finish line}{we're happy campers}{we've done it}{that's a wrap}{we have the high ground}}

\tikzset{%
	symbol/.style={%
		,draw=none
		,every to/.append style={%
			edge node={node [sloped, allow upside down, auto=false]{$#1$}}}
	}
}

\makeatletter
\def\renewtheorem#1{%
	\expandafter\let\csname#1\endcsname\relax
	\expandafter\let\csname c@#1\endcsname\relax
	\gdef\renewtheorem@envname{#1}
	\renewtheorem@secpar
}
\def\renewtheorem@secpar{\@ifnextchar[{\renewtheorem@numberedlike}{\renewtheorem@nonumberedlike}}
\def\renewtheorem@numberedlike[#1]#2{\newtheorem{\renewtheorem@envname}[#1]{#2}}
\def\renewtheorem@nonumberedlike#1{  
	\def\renewtheorem@caption{#1}
	\edef\renewtheorem@nowithin{\noexpand\newtheorem{\renewtheorem@envname}{\renewtheorem@caption}}
	\renewtheorem@thirdpar
}
\def\renewtheorem@thirdpar{\@ifnextchar[{\renewtheorem@within}{\renewtheorem@nowithin}}
\def\renewtheorem@within[#1]{\renewtheorem@nowithin[#1]}
\makeatother

\DeclareDocumentEnvironment{MyFrame}{O{1cm}O{0.4pt}O{0.8cm}O{black}O{3}O{2ex}}
{\par\hfill\rlap{%
		\bgroup\color{#4}%
		\hskip-\dimexpr#1-#3\relax\rule{#1}{#2}%
		\hskip-\dimexpr#1/#5\relax\rule[-\dimexpr#1-\dimexpr#1/#5\relax]{#2}{#1}%
		\egroup
	}%
	\vskip-\dimexpr#1/#5+\dimexpr#1/#5-#6\relax%
}
{\par\nobreak\offinterlineskip\vskip-\dimexpr#1/#5+\dimexpr#1/#5-#6\relax\noindent%
	\hskip-#3\bgroup\color{#4}%
	\rule{#1}{#2}\hskip-\dimexpr#1-\dimexpr#1/#5-#2\relax%
	\rule[-\dimexpr#1/#5-#2\relax]{#2}{#1}\egroup\par
}

\newtheorem{innercustomgeneric}{\customgenericname}
\providecommand{\customgenericname}{}
\newcommand{\newcustomtheorem}[2]{%
	\newenvironment{#1}[1]
	{%
		\renewcommand\customgenericname{#2}%
		\renewcommand\theinnercustomgeneric{##1}%
		\innercustomgeneric
	}
	{\endinnercustomgeneric}
}

\makeatletter
\@ifclassloaded{amsart}{
}{%
    \@ifundefined{subjclass}{
    \newcommand{\subjclass}[2][1991]{%
      \let\@oldtitle\@title%
      \gdef\@title{\@oldtitle\footnotetext{#1 \emph{Mathematics subject classification.} #2}}%
    }
    }{}
    \@ifundefined{keywords}{
    \newcommand{\keywords}[1]{%
      \let\@@oldtitle\@title%
      \gdef\@title{\@@oldtitle\footnotetext{\emph{Key words and phrases.} #1.}}%
    }
    }{}
}
\makeatother

\theoremstyle{plain}
\newtheorem{thm}{Theorem}
\newtheorem{thma}{Theorem}
\newtheorem{lemma}[thm]{Lemma}
\newtheorem{cor}[thm]{Corollary}
\newtheorem{cora}[thma]{Corollary}

\renewtheorem{conj}[thm]{Conjecture}
\newtheorem{conja}[thma]{Conjecture}
\newtheorem{prop}[thm]{Proposition}

\newcustomtheorem{Prob}{Problem}
\newcustomtheorem{Exc}{Theoretical Exercise}

\theoremstyle{definition}
\newtheorem{defninner}[thm]{Definition}

\newtheorem{notn}[thm]{Notation}

\newtheorem{recinner}[thm]{Recall}
\newtheorem{set}[thm]{Setup}
\newtheorem{exinner}[thm]{Example}
\newtheorem{reminner}[thm]{Remark}

\newtheorem*{exampleinner}{Example}
\newtheorem*{assump}{Assumption}

\newtheorem*{nonexinner}{Non-example}
\newtheorem{warninner}[thm]{Warning}

\newtheorem*{ansinner}{Answer}

\newtheorem*{histinner}{History}

\theoremstyle{remark}

\newtheorem{constructinner}[thm]{Construction}

\crefname{thm}{Theorem}{Theorems}
\crefname{prop}{Proposition}{Propositions}
\crefname{lemma}{Lemma}{Lemmas}
\crefname{cor}{Corollary}{Corollaries}
\crefname{exinner}{Example}{Examples}
\crefname{reminner}{Remark}{Remarks}
\crefname{ansinner}{Answer}{Answers}
\crefname{warninner}{Warning}{Warnings}
\crefname{nonexinner}{Non-example}{Non-examples}
\crefname{recinner}{Recall}{Recall}
\crefname{defninner}{Definition}{Definitions}
\crefname{histinner}{History}{History}
\crefname{constructinner}{Construction}{Constructions}

\newcommand\exsymbol{$\triangle$}
\newcommand\remsymbol{$\circ$}
\newcommand\anssymbol{$\star$}
\newcommand\warnsymbol{$\bullet$}
\newcommand\proofsymbol{$\blacksquare$}
\newcommand\nonexsymbol{$\triangledown$}
\newcommand\recsymbol{$\odot$}
\newcommand\defnsymbol{$\diamond$}
\newcommand\histsymbol{$\ominus$}
\newcommand\constructsymbol{$\octagon$}

\newenvironment{ex}[1][]{\begin{exinner}[#1]\pushQED{\qed}\renewcommand\qedsymbol\exsymbol}{\popQED\end{exinner}}

\newenvironment{rem}[1][]{\begin{reminner}[#1]\pushQED{\qed}\renewcommand\qedsymbol\remsymbol}{\popQED\end{reminner}}

\newenvironment{warn}[1][]{\begin{warninner}[#1]\pushQED{\qed}\renewcommand\qedsymbol\warnsymbol}{\popQED\end{warninner}}

\newenvironment{rec}[1][]{\begin{recinner}[#1]\pushQED{\qed}\renewcommand\qedsymbol\recsymbol}{\popQED\end{recinner}}
\newenvironment{defn}[1][]{\begin{defninner}[#1]\pushQED{\qed}\renewcommand\qedsymbol\defnsymbol}{\popQED\end{defninner}}

\newenvironment{construct}[1][]{\begin{constructinner}[#1]\pushQED{\qed}\renewcommand\qedsymbol\constructsymbol}{\popQED\end{constructinner}}

\let\proofinner=\proof
\newcommand\myproof[1][Proof]{\proofinner[#1]\renewcommand\qedsymbol\proofsymbol}
\def\proof{\myproof}

\usepackage[T1]{fontenc} 
\usepackage[ansinew]{inputenc}

\usepackage{orcidlink}

\numberwithin{equation}{section}

\newcommand{\niven}[1]{}
\newcommand{\niventodo}[1]{\niven{(TODO) #1}}

\DeclareMathOperator\AR{AR}
\DeclareMathOperator\AS{AS}

\DeclareMathOperator\MAS{MAS}
\DeclareMathOperator\IAS{IAS}
\DeclareMathOperator\CSel{\mc{S\mkern-2mu}\textit{el}}
\DeclareMathOperator\UW{UW}
\DeclareMathOperator\WE{WE}

\newcommand\rPic{\ul\Pic}
\newcommand\cby[1]{\overset{#1}\by} 

\newcommand\PG{\P G}

\author{Niven Achenjang \orcidlink{0000-0001-9551-5821}}

\date{\today}
\title{The Average Size of 2-Selmer Groups of Elliptic Curves in Characteristic 2}

\begin{document}

\begin{abstract}
    Let $K$ be the function field of a smooth curve $B$ over a finite field $k$ of arbitrary characteristic. We prove that the average size of the $2$-Selmer groups of elliptic curves $E/K$ is at most $1+2\zeta_B(2)\zeta_B(10)$, where $\zeta_B$ is the zeta function of $B$. In particular, in the limit as $q=\#k\to\infty$ (with the genus $g(B)$ fixed), we see that the average size of 2-Selmer is bounded above by $3$, even in ``bad'' characteristics. This completes the proof that the average rank of elliptic curves, over \textit{any} fixed global field, is finite. Handling the case of characteristic $2$ requires us to develop a new theory of integral models of 2-Selmer elements, dubbed ``hyper-Weierstrass curves.'' 
\end{abstract}
\maketitle
\vspace{-3em}
\noindent\hrulefill
\tableofcontents
\vspace{-3em}
\noindent\hrulefill
\hypersetup{linkcolor=\linkcolor}

\section{\bf Introduction}\label{sect:intro}

Much recent work in arithmetic statistics is concerned with understanding the distribution of ranks of elliptic curves over a fixed global field $K$. In particular, early work of Brumer \cite{brumer} showed that the average rank of elliptic curves over $K=\F_q(t)$ is finite, at least when $\Char q\ge5$. His bound was strengthened and extended to all $q$ by de Jong \cite{dejong}, and an average rank bound over $\Q$ was later obtained by Bhargava and Shankar \cite{bhar-shan}. While Brumer used analytic techniques to obtain his rank bound, de Jong and Bhargava--Shankar pioneered the approach of producing average rank bounds by first bounding the average size of the $n$-Selmer groups of elliptic curves $E/K$, for some fixed value of $n$. Their work and subsequent investigations have led to the following conjecture, which appears for example in \cite[Section 2]{dejong}, \cite[Conjecture 1.4]{poonen-rains}, \cite[Conjecture 4]{bhargava2013average}, and \cite[Section 5.7]{ec-stats}.
\begin{conja}\label{conj:main}
    Let $K$ be a global field. When all elliptic curves $E/K$ are ordered by height, the average size of their $n$-Selmer groups is $\sum_{d\mid n}d$.
\end{conja}

The main purpose of this paper is to verify the $2$-Selmer case of \cref{conj:main} for \important{arbitrary} global function fields $K$, up to a limit as ``$q\to\infty$.''
\numberwithin{thm}{section}
\begin{set}\label{set:main}
    Let $k=\F_q$ be a finite field, let $B/k$ be a smooth $k$-curve of genus $g=g(B)$, and let $K=k(B)$ be its function field. Let
    \[\zeta_B(s)\coloneqq\prod_{v\in B}\frac1{1-q^{-s\deg v}},\]
    with $v$ ranging over \important{closed} points of $B$, be the zeta function of $B$.
\end{set}
Additionally, let $\AS_B$ (resp. $\AR_B$) denote the ``average size of $2$-Selmer groups (resp. average rank) of elliptic curves over $K$'' (see \cref{sect:conventions} for a precise definition).
\begin{thma}[= \cref{thm:main}]\label{thma:main}
    With notation as in \cref{set:main},
    \[\AS_B\le1+2\zeta_B(2)\zeta_B(10).\]
\end{thma}
\begin{cora}\label{cora:main}
    With notation as in \cref{set:main},
    \[\limsup_{n\to\infty}\AS_{B_{\F_{q^n}}}\le3.\]
\end{cora}
Using the simple fact that $2x\le2^x$ along with the exact sequence \cref{ses:Sel-fund}, one obtains the following.
\begin{cora}\label{cora:rank}
    With notation as in \cref{set:main},
    \[\AR_B\le\frac12+\zeta_B(2)\zeta_B(10)\tandso\limsup_{n\to\infty}\AR_{B_{\F_{q^n}}}\le\frac32.\]
\end{cora}
In his thesis, Shankar bounded the average rank of elliptic curves over any number field (see \cite{shankar-thesis} for a precise definition). Given this, \cref{thma:main} completes the proof that the average rank of elliptic curves, over \important{any} global field, is finite.
\begin{thma}[\cref{thma:main} + {\cite[Theorem 1.0.1]{shankar-thesis}}]\label{thma:rank-always-finite}
    Let $K$ be an arbitrary global field. When all elliptic curves $E/K$ are ordered by height, their average rank is finite.
\end{thma}
\begin{rem}\label{rem:zeta-asymp}
    We can be more precise about how the bound produced in \cref{thma:main} compares to the predicted value of $3$. Use notation as in \cref{set:main}. By the Weil conjectures, $\zeta_B(s)=\left.\prod_{i=1}^{2g}(1-\alpha_i\inv[s]q)\right/(1-\inv[s]q)(1-q^{1-s})$ for some $\alpha\in\C$ with $\abs{\alpha_i}=q^{1/2}$. Thus, for $s\ge1$,
    \[\zeta_B(s)\le\frac{(1+q^{1/2-s})^{2g}}{(1-q^{-s})(1-q^{1-s})}=1+q^{1-s}+O_{g,s}\p{q^{-s}}\tas q\to\infty.\]
    Thus, the bound in \cref{thma:main} is
    \begin{equation}\label{eqn:main-asymp}
        \AS_B\le1+2\zeta_B(2)\zeta_B(10)=3+\frac2q+O_g\p{\inv[2]q}.
    \end{equation}
    Similarly, the bound in \cref{cora:rank} is
    \[\AR_B\le\frac32+\frac1q+O_g\p{\inv[2]q}.\qedhere\]
\end{rem}
Along the road towards establishing \cref{thma:main}, we obtain a few other results which may be of independent interest. Below, $\Ht(E)$ denotes the height of an elliptic curve as defined in \cref{sect:conventions}.

\begin{thma}\label{thma:EC-count}
    With notation as in \cref{set:main},
    \[\sum_{\substack{E/K\\\Ht(E)=d}}\frac1{\#\Aut(E)}\sim\#\Pic^0(B)\cdot\frac{q^{10d+2(1-g)}}{(q-1)\zeta_B(10)}\]
    as $d\to\infty$. See \cref{thm:EC-d-asymp} for a more precise asymptotic.
\end{thma}
\begin{rem}
    When $B=\P^1_{\F_q}$, de Jong \cite{dejong} gave an \important{exact} weighted count of (isomorphism classes) of elliptic curves of height $d$ (his result is recalled in \cref{rem:comp-dJ-exact}), so the utility of \cref{thma:EC-count} is that it applies to more general bases. Prior to de Jong, Brumer \cite{brumer} computed an asymptotic count of the (unweighted) number of elliptic curves over $K=\F_q(t)$ (using a slightly different height function) when $\Char K\ge5$.
\end{rem}
\begin{thma}[= \cref{thm:E[2]-bound-char-not-2} + \cref{thm:E[2]-bound-char-2}]\label{thma:E[2]}
    Use notation as in \cref{set:main}. Then,
    \[\sum_{\substack{E/K\\\Ht(E)=d\\E[2](K)\neq0}}\frac1{\#\Aut(E)}=O\p{q^{Cd}}\twhere C=\Twocases6{\Char K\neq2}9{\Char K=2}\]
    as $d\to\infty$. 
\end{thma}

\subsubsection*{\bf Prior Work}

As previously mentioned, bounds on the average size of Selmer groups of elliptic curves go back to work of de Jong \cite{dejong}, where he verified \cref{conj:main}, up to a limit as ``$q\to\infty$'', in the case of $3$-Selmer groups over $K=\F_q(t)$. A number field case of \cref{conj:main} was first handled by Bhargava and Shankar \cite{bhar-shan} who computed the average size of $2$-Selmer groups over $K=\Q$. Since then, many new works have been produced verifying cases of \cref{conj:main} (or variations of it); a non-exhaustive list of such papers includes \cite{dejong,bhar-shan,shankar-thesis,bhar-shan-3sel,bhargava2013average,bhargava2013averagesize5selmergroup,ho-lehung-ngo,thorne,aaron-selmer,feng-landesman-rains,sun-woo,ellenberg2024homologicalstabilitygeneralizedhurwitz}. However, to the best of the author's knowledge, there is not a single paper which investigates \cref{conj:main} for an \important{arbitrary} global function field $K$ (and fixed $n$). Generally, authors will at least require that $\Char K\nmid2n$ and/or that $K=\F_q(t)$. In contrast, this paper allows for any choice of $K$ in \cref{thma:main}/\cref{set:main}; this extra permissiveness is what allows us to deduce \cref{thma:rank-always-finite}.

In the case of $2$-Selmer groups over function fields, \cref{cora:main} was previously attained by H{\`{\^o}}, L{\^e} H{\`u}ng, and Ng{\^o} \cite{ho-lehung-ngo} when $\Char K\ge5$. The methods of their paper and the present paper have a common ancestor in the work of de Jong \cite{dejong}, but new complications in small characteristic have required us to introduce new ideas, as we explain below. Finally, the difference between our main theorem \cref{thma:main} and the main theorem of \cite{ho-lehung-ngo} will be stated more explicitly at the end of this introduction.

\subsubsection*{\bf New Ideas in Small Characteristics} 

From a zoomed out perspective, the proof of \cref{thma:main} follows the usual ``parameterize and count'' strategy often employed in arithmetic statistics. However, because we wish to allow function fields of characteristic $2$ in this paper, new complications arise in several steps of the arguments previously used to count 2-Selmer elements (see \cite{bhar-shan,shankar-thesis,ho-lehung-ngo}). In fact, we cannot even use the same parameterization as employed by previous authors. Below, we enumerate a few of the extra hurdles we must overcome to obtain \cref{thma:main}; we use the notation of \cref{set:main}.
\begin{enumerate}
    \item To parameterize $2$-Selmer elements, we define and develop a theory of ``hyper-Weierstrass curves'' (see \cref{sect:hW,sect:selmer-groupoid-card}).

    In good characteristics (see \cite{bhar-shan,shankar-thesis,ho-lehung-ngo}), one parameterizes 2-Selmer elements (of elliptic curves over $K$) via binary quartic forms $f(x,z)$. One reduces to a problem of counting ``integral'' (think: existing over $B$) such forms of bounded invariants which one handles by leveraging a classically developed theory of such forms. When $\Char K=2$, these objects no longer parameterize 2-Selmer elements. Instead, one needs to study certain relative curves $H/B$ which serve as ``integral models'' of genus 1 double covers of $\P^1$; these are our hyper-Weierstrass curves (see \cref{defn:hW}). These curves naturally live in certain $\P(1,2,1)$-bundles over $B$, but they are \important{not} cut out by a global equation in the same way that usual Weierstrass curves are (see \cref{prop:high-height=>weier-eqn} for the Weierstrass case), which ultimately means one has to count them by counting sections of some \important{non-split} rank $8$ vector bundles on $B$, the vector bundles $\msV\coloneqq\push p\msO_\P(H)$ of \cref{prop:hW-norm-bund}.

    \item In order to understand the rank $8$ vector bundles $\msV$ which arise as above, we need to use the theory of rational singularities.

    As mentioned above, one is interested in counting sections of certain rank $8$ vector bundles $\msV$. When $\Char K\ge5$, the role of $\msV$ is played by an analogous rank $5$ vector bundle, called `$V(\msE,\msL)$' in \cite{ho-lehung-ngo}, which \important{splits as a sum of line bundles} in the large height case, and so which can be more easily analyzed. In the present paper, our $\msV$'s are generally non-split. To count their sections, we exploit the fact that they arise geometrically from a (possibly singular) relative curve $H\to B$ in order to ultimately bound their $h^1$'s in terms of some intersection theory on a desingularization of $H$, at least when $H$ has at worst rational singularities (this is the goal of \cref{sect:geom-lemma}, which is exploited in \cref{sect:count-hW}).

    \item As in previous work, one needs to show that elliptic curves with non-trivial 2-torsion contribute $0$ to $\AS_B$.

    In \cite[Section 6.2]{ho-lehung-ngo}, the authors show this, for $q$ sufficiently large (and $\Char K\ge5$), by separately bounding the number of such elliptic curves and the possible size of their 2-Selmer groups. In \cref{sect:count-EC-2tors}, we produce a stronger bound on sizes of 2-Selmer groups in characteristic $\ge3$ which allows us to rule out the contributions of such elliptic curves (even when $q$ is small). When $\Char K=2$, we need a more delicate argument to rule out their contribution. This is accomplished in \cref{sect:MAS-AS-char-2} by carefully counting the hyper-Weierstrass curves which could correspond to a 2-Selmer element of an elliptic curve with non-trivial 2-torsion; the argument here ultimately rests upon an application of the Schwartz--Zippel Lemma (\cref{lem:schartz-zippel}).
\end{enumerate}

\subsubsection*{\bf Paper Organization}

In \cref{sect:conventions} we lay out common conventions and notation used throughout the paper. In \cref{sect:EC-count}, we obtain an asymptotic count of elliptic curves (ordered by height) over an arbitrary global function field (\cref{thm:EC-d-asymp}), i.e. we estimate the denominator of \cref{eqn:AS-def}. In \cref{sect:hW}, we define the objects -- in this paper, dubbed `hyper-Weierstrass curves' (see \cref{defn:hW}) -- which will serve as our integral models of 2-Selmer elements. In the same section, we make explicit their relation to 2-Selmer groups of elliptic curves via the introduction of a `2-Selmer groupoid' (see \cref{defn:Selmer-groupoid}). This groupoid is used to define a modified average count $\MAS_B$ which is closely related, but not exactly equal, to $\AS_B$. In \cref{sect:selmer-groupoid-card}, we implement the ``count'' part of the ``parameterize and count'' strategy, obtaining the desired upper bound on the modified average $\MAS_B$ defined in terms of the 2-Selmer groupoid (see \cref{thm:mod-estimate}). After that, all that remains is to compare $\MAS_B$ and $\AS_B$. These quantities differ only in their contributions coming from elliptic curves with extra automorphisms or with non-trivial 2-torsion. Curves with extra automorphisms a priori contribute \important{more} to $\MAS_B$ than they do to $\AS_B$, while curves with non-trivial 2-torsion a priori contribute \important{less} to $\MAS_B$ than they do to $\AS_B$. Hence, to prove the upper bound $\AS_B\le\MAS_B$, it suffices to show that curves with non-trivial 2-torsion \important{do not contribute} to these averages. When $\Char K=2$, we show this in \cref{sect:MAS-AS-char-2} by bounding the number of 2-Selmer elements attached to such curves showing up in the count in \cref{sect:selmer-groupoid-card}. Finally, in \cref{sect:main-results}, we prove the inequality $\AS_B\le\MAS_B$ when $\Char K\neq2$, and so deduce our main result (\cref{thm:main}).
\begin{rem}
    We remark that, in order to prove \cref{thma:main} when $\Char K=2$, most of \cref{sect:main-results} is unneeded. One only needs \cref{prop:count-triv-sel} (and its consequence \cref{cor:cor-mod-estimate}) whose proof (in characteristic $2$) does not rely on any of the other results in that section. On the other hand, to prove \cref{thma:main} when $\Char K\neq2$, \cref{sect:MAS-AS-char-2} is unnecessary (except \cref{lem:IAS-MAS-comp}).
\end{rem}

\subsubsection*{\bf Brief Comparison with \cite{ho-lehung-ngo}}

We briefly explicitly state the difference between our \cref{thma:main} and the main theorem of \cite{ho-lehung-ngo}, stated below.
\begin{thma}[{\cite[Corollary 2.2.3]{ho-lehung-ngo}}]\label{thma:hlhn}
    Use notation as in \cref{set:main}, and assume that $\Char K\ge5$.
    If $q>64$, then
    \[3\zeta_B(10)^{-1}\le\AS_B\le3+\frac T{q-1}\]
    for some constant $T=T(g)$ depending only on the genus of $B$.
\end{thma}
Our statement of \cref{thma:hlhn} above differs slightly from the statement of \cite[Corollary 2.2.3]{ho-lehung-ngo}; see \cref{warn:hlhn-err}.

Our \cref{thma:main} shows that one can still bound $\AS_B$ by a quantity of the form $3+o(1)$ (with the $o(1)$ going to $0$ as $q\to\infty$) even when $\Char K$ (or $\#k$) is small. In addition, the form of our upper bound stated in \cref{eqn:main-asymp} shows that it is comparable to the upper bound obtained in \cite{ho-lehung-ngo}. However, because much of the author's motivation for studying this problem was in its application to bounding the average rank of elliptic curves (see \cref{cora:rank}), we do not obtain an analogous lower bound.
\begin{rem}
    We further remark that it is possible to slightly strengthen \cref{thma:hlhn} by replacing parts of \cite{ho-lehung-ngo} with results from our paper. The requirement $q>64$ in \cref{thma:hlhn} only exists because this is needed in their argument for showing that
    that elliptic curves with non-trivial 2-torsion do not contribute to the average size of $2$-Selmer. However, our \cref{prop:comp-IAS-AS} proves this even for small $q$, and so shows that \cref{thma:hlhn} remains valid for all $K$ of characteristic $\ge5$.
\end{rem}
\begin{rem}\label{warn:hlhn-err}
     In \cite{ho-lehung-ngo}, the authors claim that their main result only requires $q>32$ (instead of $q>64$) and that they obtain an upper bound of the form $3+T/(q-1)^2$ (instead of $3+T/(q-1)$). The parts of their paper where this `$32$' and `$T/(q-1)^2$' originate appear to contain minor errors.
    \begin{itemize}
        \item This restriction on the size of the base field originates in \cite[Section 6.2]{ho-lehung-ngo}. At one point the authors write that the degree of the discriminant divisor of a height $d$ elliptic curve is $10d$, whereas it should really be $12d$. Redoing the computations at the end of that section with this in mind shows that they need $q^4>4^{12}$ (i.e. $q>64$) in order to rule out the contribution coming from elliptic curves with non-trivial $2$-torsion.
        \item The summand $T/(q-1)^2$ originates in \cite[Case 3 in Section 6.1]{ho-lehung-ngo}. Use notation as in \cref{set:main} (so $B$ for us will be $C$ for them). Following \cite{ho-lehung-ngo}, write $\Sym^m_B(\F_q)$ for the set of effective, degree $m$ divisors on $B$. In the first series of displayed equations/inequalities appearing there, the authors seem to implicitly appeal to
        \[\sum_{\msM\in\Pic^{2d-n}(B)}\#\hom^0(B,\msM)=\#\Sym^{2d-n}_B(\F_q),\]
        whereas the correct identity is
        \[\sum_{\msM\in\Pic^{2d-n}(B)}\frac{\#\hom^0(B,\msM)-1}{q-1}=\#\Sym^{2d-n}_B(\F_q).\]
        Redoing \cite[Case 3 in Section 6]{ho-lehung-ngo} using this identity instead results in an upper bound of the form $T/(q-1)$.
        \qedhere
    \end{itemize}
\end{rem}

\subsubsection*{\bf Acknowledgements} I would like to extend immense gratitude towards Bjorn Poonen for suggesting this project and for providing much helpful guidance throughout its duration. I would also like to thank Aise Johan de Jong for answering questions about his paper \cite{dejong}. In addition, I greatly benefited from conversations with Levent Alp\"oge and Aaron Landesman, and had helpful email exchanges with Armand Brumer, Jean Gillibert, and Douglas Ulmer. Finally, over the course of completing this project, I was supported at various points by the Massachusetts Institute of Technology's Dean of Science Fellowship as well as by National Science Foundation grants DGE-2141064, DMS-1601946 and DMS-2101040.

\section{\bf Conventions}\label{sect:conventions}
In this paper, we work throughout in the fppf topology. All unadorned cohomology groups should be interpreted as fppf cohomology. For $G$ a sheaf of groups, we say \define{$G$-torsor} to mean an fppf-locally trivial right $G$-torsor sheaf.

\subsubsection*{\bf Vector bundles} Let $\msV$ be a vector bundle (by which, we mean locally free sheaf of finite rank) on a scheme $B$. We write $\msV^\vee:=\sHom(\msV,\msO_B)$ for the dual bundle. If $\msL$ is a line bundle, we also denote this by $\inv\msL:=\msL^\vee$.

If $\msV$ is a vector bundle on a scheme $B$, we write $\GL(\msV)$ to denote its group of $\msO_B$-linear automorphisms. We write $\ul\GL(\msV)$ to denote its automorphism sheaf, i.e., for $U$ a $B$-scheme, we set $\Gamma(U/B,\ul\GL(\msV))=\GL(\msV\vert_U)$.

Finally, for $\msV$ a vector bundle on a scheme $B$, its associated projective bundle is $\P(\msV):=\rProj_B\p{\Sym(\msV)}$.

\subsubsection*{\bf Duality} Let $f:X\to Y$ be a morphism of schemes. The dualizing sheaf, when it exists, of this morphism will be denoted $\omega_{X/Y}$. If $Y=\spec F$ is the spectrum of a field, then we often simply denote this by $\omega_X:=\omega_{X/F}:=\omega_{X/\spec F}$.

\subsubsection*{\bf Curves} Let $B$ be an arbitrary scheme. We say that a $B$-scheme $C\to B$ is a \define{$B$-curve} (or \define{curve over $B$} or simply a \define{curve}) if it is flat, proper, and finitely presented over $B$ with Gorenstein, connected, 1-dimensional geometric fibers. Note that, for $C\to B$ a curve, the dualizing sheaf $\omega_{C/B}$ exists and is invertible. 

If $E_1,E_2$ are elliptic curves (so, in particular, they are equipped with choices of identity points), then by an isomorphism $E_1\iso E_2$, we always mean an isomorphism of group schemes.

\subsubsection*{\bf Heights} Let $B$ be a smooth curve over a field $F$. Let $X\xto\pi B$ be a curve over $B$ such that $\push\pi\msO_X=\msO_B$ and whose relative dualizing sheaf $\omega_{X/B}$ is isomorphic to $\pull\pi\msL$ for some $\msL\in\Pic(B)$. Then, we define the \define{height} of $X/B$ to be
\[\Ht(X/B):=\deg(\msL)=\deg(\push\pi\omega_{X/B})\in\Z,\]
with the latter equality holding by the projection formula. In this situation, we call $\msL\simeq\push\pi\omega_{X/B}$ the \define{Hodge bundle} of the curve.

Let $K$ be a global function field, with corresponding curve $B$. If $C/K$ is a curve of genus at least 1, then we define its \define{height} to be
\[\Ht(C/K):=\Ht(\meC/B),\]
with $\meC/B$ the minimal proper regular model of $C$. In this situation, the \define{Hodge bundle} of $C$ is defined to be the Hodge bundle of its minimal proper regular model.

\subsubsection*{\bf Global Function Fields} Let $K$ be the function field of a smooth curve $B/\F_q$. We implicitly identify places $v$ of $K$ with closed points $v\in B$. Given such a place, we let $K_v$ denote the completion of $K$ at $v$, and we let $\msO_v$ denote the valuation ring of $K_v$, i.e. the completion of the stalk $\msO_{B,v}$. We let $\kappa(v)$ denote the residue field at $v$.

\subsubsection*{\bf Selmer Groups} Let $K$ be a global field, and let $E/K$ be an elliptic curve. For any $n\ge1$, its \define{$n$-Selmer group} is
\[\Sel_n(E)\coloneqq\ker\p{\hom^1(K,E[n])\too\prod_v\hom^1(K_v,E)},\]
where $v$ ranges over all places of $K$. Recall that $\Sel_n(E)$ fits into the short exact sequence
\begin{equation}\label{ses:Sel-fund}
    0\too E(K)/nE(K) \too \Sel_n(E) \too \Sha(E)[n] \too 0,
\end{equation}
where $\Sha(E)[n]$ denotes the $n$-torsion in the Tate-Shafarevich group of $E$.

\subsubsection*{\bf Averages} Use notation as in \cref{set:main}. Given a nonnegative integer $d$, we define
\begin{equation}\label{eqn:AS-def}
    \AS_B(d):=\frac{\displaystyle\sum_{\substack{E/K\\\Ht(E)\le d}}\frac{\#\Sel_2(E)}{\#\Aut(E)}}{\displaystyle\sum_{\substack{E/K\\\Ht(E)\le d}}\frac1{\#\Aut(E)}}
    \;\tand\!
    \AR_B(d):=\frac{\displaystyle\sum_{\substack{E/K\\\Ht(E)\le d}}\frac{\rank_\Z E(K)}{\#\Aut(E)}}{\displaystyle\sum_{\substack{E/K\\\Ht(E)\le d}}\frac1{\#\Aut(E)}}
    ,
\end{equation}
the average size of $2$-Selmer (resp. average rank) of elliptic curves over $K$ of height $\le d$. We furthermore define
\[\AS_B\coloneqq\limsup_{d\to\infty}\AS_B(d)\tand\AR_B\coloneqq\limsup_{d\to\infty}\AR_B(d).\]

\subsubsection*{\bf Asymptotics} When working within the context of \cref{set:main}, we allow our big-O constants to depend on the function field $K$. That is, when we write $f(x)=O(g(x))$ we mean that there exists some $C=C(K)>0$ such that $\abs{f(x)}\le Cg(x)$ for all large values of $x$.

\subsubsection*{\bf Groupoids} Let $\meG$ be a groupoid. We write $\abs\meG$ to denote the set of isomorphism classes of its objects. Its \define{groupoid cardinality} (or simply \define{cardinality}) is
\[\#\meG:=\sum_{x\in\abs\meG}\frac1{\#\Aut_\meG(x)}.\]
If we say that $\meG'\into\meG$ is a subgroupoid, we always mean that it is a \important{full} subgroupoid, i.e. $\Aut_{\msG'}(x)=\Aut_{\msG}(x)$ for any $x\in\meG'$.

\section{\bf An Asymptotic Count of Elliptic Curves of Bounded Height}\label{sect:EC-count}
The main result of this section (\cref{thm:EC-d-asymp}) produces an asymptotic count of the number of elliptic curves of bounded height over an arbitrary global function field. For elliptic curves over $\F_q(t)$, for arbitrary $q$, an exact count was produced already by de Jong \cite[Proposition 4.12]{dejong}. In this section, we adapt his argument to work over a general base curve $B$ instead of just $\P^1$. 

For $K$ a global function field, we introduce the following notation.
\begin{notn}
    Let $\meM_{1,1}(K)$ denote the groupoid of elliptic curves over $K$. For any $d\ge0$, we let $\meM_{1,1}^{=d}(K),\meM_{1,1}^{\le d}(K)\into\meM_{1,1}(K)$ denote, respectively, the (full) subgroupoids consisting of those elliptic curves of height $=d$ and those of height $\le d$.
\end{notn}

In order to count elliptic curves over $K$, we first briefly recall the main results of the theory of Weierstrass models. We then related the (unweighted) count of minimal Weierstrass equations of to the (weighted) count of elliptic curves, and use this to obtain as asymptotic for the latter.

\subsection{Background on Weierstrass Models}
\numberwithin{thm}{subsection}

\begin{defn}
    For an arbitrary base scheme $B$, a \define{Weierstrass curve} $(W\xto\pi B,S)$ over $B$ is a curve $W/B$ whose fibers are geometrically integral of arithmetic genus 1, equipped with a section $S\subset W$ of $\pi$ which is contained in $\pi$'s smooth locus.
\end{defn}
\begin{thm}[Summary of the theory of Weierstrass curves]\label{thm:weier-sum}
    Let $B$ be an arbitrary base scheme, let $(W\xto\pi B,S)$ be a Weierstrass curve, and let $\msL:=\push\pi\omega_{W/B}$ be its Hodge bundle. Then,
    \begin{enumerate}
        \item $\push\pi\msO_W\simeq\msO_B$ and $R^1\push\pi\msO_W\simeq\inv\msL$ both hold after arbitrary base change.
        \item For any integer $n\ge1$,
        \begin{itemize}
            \item $\push\pi\msO_X(nS)$ is a locally free sheaf of rank $n$ on $B$ whose formation commutes with arbitrary base change.
            \item $R^1\push\pi\msO_X(nS)=0$.
        \end{itemize}
        \item For $n\ge2$, there are exact sequences
        \begin{equation}\label{ses:push-nS}
            0\too\push\pi\msO_W((n-1)S)\too\push\pi\msO_W(nS)\too\msL^{-n}\too0.
        \end{equation}
        Furthermore, $\push\pi\msO_W(S)\simeq\msO_B$.
        \item The natural map $\pull\pi\push\pi\msO_W(3S)\to\msO_W(3S)$ is a surjection, and induces an embedding
        \[W\into\P(\push\pi\msO_W(3S)):=\rProj_B\p{\Sym\p{\push\pi\msO_W(3S)}}\]
        over $B$ such that $\msO_W(1):=\msO_{\P(\push\pi\msO_W(3S))}(1)\vert_W\simeq\msO_W(3S)$.
        \item There is a canonical section $\Delta\in\hom^0(B,\msL^{12})$, called the \define{discriminant} of $W$, whose zero scheme is supported exactly on the points with non-smooth fiber.
    \end{enumerate}
\end{thm}
\begin{proof}
    All of this can be found e.g. in \cite{deligne}. Technically, \cite{deligne} only claims that \bp4 holds Zariski locally on the base, but this suffices to conclude the claim as stated above.
\end{proof}
\begin{rem}\label{rem:push-proj-bund}
    In connection with \cref{thm:weier-sum}\bp4 above, we remark that for a vector bundle $\msV$ on $B$ (an arbitrary base scheme) with associated projective bundle $\P(\msV)\xto pB$, one has
    \[\push p\msO_{\P(\msV)}(n)\simeq\Sym^n(\msV)\]
    for any $n\ge0$ (see \cite[Proposition II.7.11(a)]{hart}.
\end{rem}
We next attach global equations to Weierstrass curves. It is these equations that we will be able to count most easily. The existence and shape of these equations is well-known, but we include a treatment here because of their importance to the count.

\niven{Also, in general, it is not so clear to me which things I can claim without proof. Especially when I do not know of a reference that clearly states them in the manner I want.}
\niven{For example, I think nothing in this background section is new, but I don't know a set of references where it's all written down.}

\begin{prop}\label{prop:weier-norm-bund}
    Let $(W\xto\pi B,S)$ be a Weierstrass curve with Hodge bundle $\msL:=\push\pi\omega_{W/B}$. Let $\P:=\P(\push\pi\msO_W(3S))\xto pB$, so there is a natural embedding $W\into\P$. Then, $W\into\P$ is the zero scheme of some global section of
    \[\msO_\P(W)\simeq\msO_\P(3)\otimes\pull p(\msL^6)=\p{\pull p\msL^6}(3).\]
    Hence, we may via $W\into\P$ as being cut out by some global section of
    \[\push p\msO_\P(W)\simeq\msL^6\otimes\Sym^3\p{\push\pi\msO_W(3S)}.\]
\end{prop}
\begin{proof}
    The main content of the above proposition is the computation of the line bundle $\msO_\P(W)$ on $\P$. Once we know $\msO_\P(W)\simeq\p{\pull p\msL^6}(3)$, the claimed computation of $\push p\msO_\P(W)$ follows from the projection formula and \cref{rem:push-proj-bund}. 
    
    We will find it more natural to directly compute its dual $\msO_\P(-W)$ instead. It is classical that, on fibers, $X\into\P$ is cut out by a cubic equation, so the line bundle $\msO_\P(-W)(3)$ on $\P$ is trivial on each fiber. Thus (e.g. by \cite[Proposition 25.1.11]{ravi-notes}), $\msO_\P(-W)(3)\simeq\pull p\push p\msO_\P(-W)(3)$. Hence, it will suffice to compute that
    \[\push p\msO_\P(-W)(3)\simeq\msL^{-6}.\]
    With this in mind, consider the exact sequence
    \[0\too\msO_\P(-W)(3)\too\msO_\P(3)\too\msO_W(3)\too0,\]
    and push forward along $p$. Since $\msO_W(1)\simeq\msO_W(3S)$ by \cref{thm:weier-sum}\bp4, we obtain the exact sequence
    \begin{equation}\label{ses:Sym^3(3S)}
        \begin{tikzcd}
            0\ar[r]&\push p\msO_\P(-W)(3)\ar[r]&\push p\msO_\P(3)\ar[d, equals,"\t{\cref{rem:push-proj-bund}}"]\ar[r]&\push p\msO_W(3)\ar[d, equals]\ar[r]&R^1\push p\msO_\P(-W)(3)\ar[d, equals]\\
            &&\Sym^3(\push\pi\msO_W(3S))&\push\pi\msO_W(9S)&0
            .
        \end{tikzcd}
    \end{equation}
    Above, $R^1\push p\msO_\P(-W)(3)=0$ by \cref{thm:coh-bc} since $\msO_\P(-W)(3)$ restricts to the trivial bundle on $\P^2$ in each fiber.
    Observe that the kernel $\push p\msO_\P(-W)(3)$ above is a line bundle, so it can be computed by taking determinants. By repeated use of \cref{thm:weier-sum}\bp 3 to compute $\det(\push\pi\msO_W(9S))$ and $\det(\push\pi\msO_W(3S))$, it is straightforward to compute that $\push p\msO_\P(-W)(3)\simeq\msL^{-6}$ as desired.
\end{proof}

\begin{assump}
    For the rest of this section, we work within the context of \cref{set:main}. In particular, $k$ is a finite field, and $B$ is a smooth $k$-curve of genus $g$ with function field $K=k(B)$. 
\end{assump}

\begin{rem}\label{rem:height-disc}
    Let $E/K$ be an elliptic curve. Let $\meC/B$ denote its minimal proper regular model, and let $W/B$ denote its minimal Weierstrass model. Then, $\meC$ and $W$ have isomorphic Hodge bundles. One can deduce this e.g. from \cite[Theorem 8.1]{conrad-min}. In light of \cref{thm:weier-sum}\bp5, this in particular means that $12\Ht(E)=\deg\Delta$, where $\Delta$ denotes $E$'s minimal discriminant.
\end{rem}
\begin{notn}
    We set $N(g):=\max\{-1,2g-2\}$. Note that if $\msL$ is a line bundle on $B$ of degree $>N(g)$, then $\hom^1(B,\msL)=0$ and $\deg\msL\ge0$.
\end{notn}

\begin{rem}\label{rem:high-height=>weier-eqn}
    Say $(W\xto\pi B,S)$ is a Weierstrass curve with Hodge bundle $\msL:=\push\pi\omega_{W/B}$ of degree $d>N(g)$. Then, the exact sequences (see \cref{thm:weier-sum}\bp 3)
    \[0\too\msO_B\too\push\pi\msO_W(2S)\too\msL^{-2}\too0\tand0\too\push\pi\msO_W(2S)\too\push\pi\msO_W(3S)\too\msL^{-3}\too0\]
    both split since they represent elements of
    \[\Ext^1_{\msO_B}(\msL^{-2},\msO_B)\simeq\Ext^1_{\msO_B}(\msO_B,\msL^2)\simeq\hom^1(\msL^2)=0\tand\Ext^1(\msL^{-3},\push\pi\msO_W(2S))\simeq\hom^1(\msL^3)\oplus\hom^1(\msL)=0,\]
    respectively. In particular, $\push\pi\msO_W(3S)\simeq\msO_B\oplus\msL^{-2}\oplus\msL^{-3}$.
    In this case, \cref{prop:weier-norm-bund} tells us that $W\into\P$ is given as the zero scheme of some global section of
    \begin{align*}
        \push p\msO_\P(W)
        &\simeq\push p\msO_\P(3)\otimes\msL^6\simeq\Sym^3(\push\pi\msO_W(3S))\otimes\msL^6\\
        &\simeq\msL^6\oplus\msL^4\oplus\msL^3\oplus\msL^2\oplus\msL\oplus\msO_B\oplus\msO_B\oplus\msL^{-1}\oplus\msL^{-2}\oplus\msL^{-3}
        .
    \end{align*}
    Symbolically, this is telling us that $W\into B$ is given by an equation of the form
    \[\lambda Y^2Z + a_1XYZ + a_3YZ^2 = \mu X^3 + a_2X^2Z + a_4XZ^2 + a_6Z^3\]
    and $a_i\in\hom^0(B,\msL^i)$. Finally, it is classic that we can always that $\lambda=1=\mu$ above and that $S\subset W$ is the subscheme $\{Z=0\}$.
\end{rem}
\begin{defn}\label{defn:weier-eqn}
    An equation of the form
    \begin{equation}\label{eqn:weier-eqn}
        Y^2Z + a_1XYZ + a_3YZ^2 = X^3 + a_2X^2Z + a_4XZ^2 + a_6Z^3,
    \end{equation}
    i.e. the data of a tuple $(\msL,a_1,a_2,a_3,a_4,a_6)$ with $\msL\in\Pic(B)$ and $a_i\in\hom^0(B,\msL^i)$, is called a \define{Weierstrass equation}. We call $\msL$ the \define{Hodge bundle} of the equation.
\end{defn}
The point of \cref{rem:high-height=>weier-eqn} is that, from it, one obtains the following proposition.
\begin{prop}\label{prop:high-height=>weier-eqn}
    Let $(W\xto\pi B,S)$ be a Weierstrass curve of height $>N(g)$. Then, $(W,S)$ is isomorphic to the curve cut out by some Weierstrass equation \cref{eqn:weier-eqn} whose Hodge bundle is $\push\pi\omega_{W/B}$, equipped with the subscheme $\{Z=0\}$.
\end{prop}
\begin{rem}[See the discussion after Theorem 1 of Section 3 of \cite{mum-suo-moduli}]\label{rem:interpret-weier-eqn}
    To correctly interpret equation \cref{eqn:weier-eqn}, one should regard $X,Y,Z$ are sections of various line bundles; specifically,
    \[X\in\hom^0(\P,\pull p(\msL^2)(1))\tcomma Y\in\hom^0(\P,\pull p(\msL^3)(1)),\tand Z\in\hom^0(\P,\pull p(\inv\msO_B)(1)).\]
    Above, $\P:=\P\p{\msO_B\oplus\msL^{-2}\oplus\msL^{-3}}\cong\P\p{\push\pi\msO_W(3S)}$.
    This way \cref{eqn:weier-eqn} -- or rather, the difference of its two sides -- defines a section of the line bundle $\pull p(\msL^6)(3)$ on $\P$ (as should be expected by \cref{prop:weier-norm-bund}), and the zero scheme of this section in $\P$ is $W$.
    
    Let us indicate where these sections $X,Y,Z$ come from. Note that $\sHom(\msL^{-2},\push\pi\msO_W(3S))\simeq\push\pi\msO_W(3S)\otimes\msL^2$, and let $\eta_X\in\hom^0(B,\push\pi\msO_W(3S)\otimes\msL^2)$ be the global section corresponding to the natural inclusion $\msL^{-2}\into\msO_B\oplus\msL^{-2}\oplus\msL^{-3}\simeq\push\pi\msO_W(3S)$. Note that, by definition of $\P$, it comes with a morphism $\pull p\push\pi\msO_W(3S)\to\msO_\P(1)$. Now, $X\in\hom^0(\P,\pull p(\msL^2)(1))$ is the image of $\eta_X$ under the induced map
    \[\pull p\p{\push\pi\msO_W(3S)\otimes\msL^2}\simeq\pull p(\push\pi\msO_W(3S))\otimes\pull p(\msL^2)\to\pull p(\msL^2)(1).\]
    We similarly define $Y\in\hom^0(\pull p(\msL^3)(1))$ using the inclusion $\msL^{-3}\into\push\pi\msO_W(3S)$ and define $Z\in\hom^0(\msO_\P(1))$ using $\msO_B\into\push\pi\msO_W(3S)$.
\end{rem}
Recall that we are attempting to compute (an asymptotic) for
\[\#\meM_{1,1}^{\le d}(K)=\sum_{E/K:\Ht(E)\le d}\frac1{\#\Aut(E)}=\sum_{\substack{\msL\in\Pic(B)\\\deg\msL\le d}}\sum_{\substack{E/K\\\msL_E\simeq\msL}}\frac1{\#\Aut(E)},\]
where $\msL_E$ denotes the Hodge bundle of the elliptic curve $E$. In order to capture the relationship between counts of elliptic curves and Weierstrass curves, we introduce the following notation.
\begin{notn}
    For any $\msL\in\Pic(B)$,
    \begin{itemize}
        \item let $\UW_\msL$ denote the (unweighted) number of isomorphism classes of generically smooth, \important{minimal} Weierstrass equations with Hodge bundle isomorphic to $\msL$.
        \item Let $\WE_\msL$ denote the weighted number of isomorphism classes of elliptic curves with Hodge bundle isomorphic to $\msL$, i.e.
        \[\WE_\msL:=\sum_{\substack{E/K\\\msL_E\simeq\msL}}\frac1{\#\Aut(E)}. \qedhere\]
    \end{itemize}
\end{notn}
\begin{prop}
    Fix $\msL\in\Pic(B)$ with $d:=\deg\msL>N(g)$. Let $E/K$ be an elliptic curve, and let $(W\to B,S)$ be a Weierstrass curve with generic fiber $\cong E$ and Hodge bundle $\cong\msL$. The number of Weierstrass equations \cref{eqn:weier-eqn} cutting out Weierstrass curves isomorphic to $(W\to B,S)$ is
    \[\frac{(q-1)q^{6d+3(1-g)}}{\#\Aut(W/B,S)}.\]
\end{prop}
\begin{proof}
    From the previous discussion (in particular, \cref{rem:interpret-weier-eqn}), we see that the Weierstrass equation \cref{eqn:weier-eqn} one obtains is determined up to scaling (i.e. up to choosing an isomorphism $\push\pi\omega_{W/B}\simeq\msL$) by the choice of splittings in \cref{rem:high-height=>weier-eqn}. Such splittings give rise to the coordinates $X,Y,Z$ in \cref{rem:interpret-weier-eqn}, and once these are determined, there will be a single equation they satisfy. The set of splittings for the short exact sequence $0\to\msO_B\to\push\pi\msO_W(2S)\to\inv[2]\msL\to0$ form a torsor for $\Hom(\msL^{-2},\msO_B)\simeq\hom^0(B,\msL^2)$ while splittings for $0\to\push\pi\msO_W(2S)\to\push\pi\msO_W(3S)\to\inv[3]\msL\to0$ form a torsor for 
    \[\Hom(\msL^{-3},\push\pi\msO_W(2S))\simeq\Hom(\msL^{-3},\msO_B\oplus\msL^{-2})\simeq\hom^0(B,\msL^3)\oplus\hom^0(B,\msL).\]
    Thus, including scaling, we have a total of
    \[(\#\units k)\cdot\#\hom^0(B,\msL^2)\cdot\#\hom^0(B,\msL^3)\cdot\#\hom^0(B,\msL)=(q-1)q^{6d+3(1-g)}\]
    choices of data leading to Weierstrass equations for $(W\to B,S)$. Changing the choice of splittings and scaling corresponds to modifying \cref{eqn:weier-eqn} by an automorphism of $\P(\msO_B\oplus\msL^{-2}\oplus\msL^{-3})$, and so two choices give the same equation if and only if they differ by an automorphism of the Weierstrass curve, i.e. if and only if they differ by an automorphism of $\P:=\P(\msO_B\oplus\msL^{-2}\oplus\msL^{-3})$ which carries $W\into\P$ onto itself.
\end{proof}

\begin{lemma}\label{lem:weier-iso-dvr}
    Let $R$ be a dvr, let $F=\Frac(R)$, and let $(W_1/R,S_1)$, $(W_2/R,S_2)$ be two Weierstrass curves over $R$ with smooth generic fibers. Let $\phi:W_{1,F}\iso W_{2,F}$ be an isomorphism between their generic fibers such that $\phi(S_{1,F})=S_{2,F}$. Suppose that $W_1,W_2$ have discriminants with equal valuation. Then, $\phi$ uniquely extends to an isomorphism $\Phi:W_1\iso W_2$ of $R$-schemes satisfying $\Phi(S_1)=S_2$.
\end{lemma}
\begin{proof}
    Uniqueness of $\Phi$ holds simply because $W_{1,F}$ is (schematically) dense in $W_1$. For $i=1,2$, we can write $W_i$ as the zero set of some Weierstrass equation
    \[Y^2Z+\ith a_1XYZ+\ith a_3YZ^2=X^3+\ith a_2X^2Z+\ith a_3XZ^2+\ith a_6Z^3\twith\ith a_j\in R\]
    inside $\P^2_R$. Having done so, the isomorphism $\phi$ will be of the form
    \[\phi([X:Y:Z])=\sq{u^2X+rZ:u^3Y+u^2sX+tZ:Z}\]
    for some $u\in\units F$ and $r,s,t\in F$. By using the change of variables formula in \cite[Table III.3.1]{silverman} and arguing as in \cite[Proposition VII.1.3(b)]{silverman}, since $W_1,W_2$ have discriminates with the same valuation, we must in fact have $u\in\units R$ and $r,s,t\in R$. Thus, $\phi$ does in fact extend to a $\Phi:W_1\iso W_2$ as desired.
\end{proof}
\begin{cor}\label{cor:weier-disc-gen-isom}
    Let $(W_1/B,S_1)$, $(W_2/B,S_2)$ be two Weierstrass curves with smooth generic fibers, and let $\phi:W_{1,K}\iso W_{2,K}$ be an isomorphism between their generic fibers such that $\phi(S_{1,K})=S_{2,K}$. Suppose that $W_1,W_2$ have equal discriminant divisors. Then, $\phi$ uniquely extends to an isomorphism $\Phi:W_1\iso W_2$ of $B$-schemes satisfying $\Phi(S_1)=S_2$.
\end{cor}
\begin{proof}
    Uniqueness of $\Phi$ holds simply because $W_{1,K}$ is (schematically) dense in $W_1$. Existence holds because $\phi$ automatically spreads out out to an isomorphism over some open $U\subset B$, and then further can be extended over the remaining points by \cref{lem:weier-iso-dvr}.
\end{proof}
\begin{cor}
    Let $(W/B,S)$ be a Weierstrass curve with smooth generic fiber $E:=W_K$, an elliptic curve. Then, the restriction map $\Aut(W/B,S)\to\Aut(E)$ is an isomorphism.
\end{cor}
\begin{cor}\label{cor:count-eqns-for-W}
    Fix $E$ be an elliptic curve. Let $(W/B,S)$ be a Weierstrass curve with generic fiber $\cong E$ and height $d>N(g)$. Then, the number of Weierstrass equations cutting out a Weierstrass curve isomorphic to $(W/B,S)$ is
    \[\frac{(q-1)q^{6d+3(1-g)}}{\#\Aut(E)}.\]
\end{cor}
\begin{cor}\label{cor:WE-from-UW}
    Choose $\msL\in\Pic(B)$ of degree $>N(g)$. Then,
    \[\WE_\msL=\frac{\UW_\msL}{(q-1)q^{6d+3(1-g)}}.\]
\end{cor}
\begin{proof}
    For any elliptic curve $E/K$, let $\alpha_E$ denote (the iso. class of) its minimal Weierstrass model, and let $\msL_E\in\Pic(B)$ denote its Hodge bundle. By \cref{cor:count-eqns-for-W}, we have
    \begin{align*}
        \UW_\msL
        &= \sum_{\substack{E/K\\\msL_E\simeq\msL}}\frac{(q-1)q^{6d+3(1-g)}}{\#\Aut(E)}
        = (q-1)q^{6d+3(1-g)}\WE_\msL
        .
    \end{align*}
    Rearrange to get the claimed equality.
\end{proof}

At this point, we would like to determine the number $\UW_\msL$ of generically smooth minimal Weierstrass equations over $B$ with Hodge bundle isomorphic to $\msL$. We do so by counting Weierstrass equations which are generically singular or which are non-minimal, and then subtracting these from the total number.
\begin{rem}\label{rem:count-all-weier-eqns}
    Fix $\msL\in\Pic(B)$ of degree $>N(g)$. By Riemann-Roch, the number of Weierstrass equations with Hodge bundle $\cong\msL$ is
    \[\prod_{\substack{i=0\\i\neq 5}}^6\#\hom^0(B,\msL^i)=q^{16d+5(1-g)}.\qedhere\]
\end{rem}

\subsection{Counting Generically Singular Weierstrass Curves}
To count Weierstrass curve over $B$ with singular generic fiber, one argues exactly as in \cite[Section 4.11]{dejong}. 
\begin{prop}\label{prop:count-gen-sing-weier}
    Let $\msL\in\Pic(B)$ satisfy $\deg\msL>N(g)$. Then, the number of generically singular Weierstrass curves with Hodge bundle $\cong\msL$ is
    \[\#\hom^0(B,\msL)\cdot\#\hom^0(B,\msL^2)^2\cdot\hom^0(B,\msL^3)=q^{8d+4(1-g)}.\]
\end{prop}
\begin{proof}
    Suppose $(W\xto\pi B,S)$ is a Weierstrass curve with singular generic fiber and Hodge bundle $\msL$. Then, every fiber of $W\to B$ has exactly one singular point, and so these are the image of a unique section $\tau\in W(B)$, the section extending the singular $K$-point in the generic fiber. The composition $B\xinto\tau W\into\P:=\P(\msO_B\oplus\inv[2]\msL\oplus\inv[3]\msL)$ shows that $\tau$ corresponds to a line bundle quotient $\msO_B\oplus\inv[2]\msL\oplus\inv[3]\msL\onto\msM$. However, the image of $\tau$ is disjoint from the (smooth) zero section $S\subset W$, so the composition $\msO_B\into\msO_B\oplus\inv[2]\msL\oplus\inv[3]\msL\onto\msM$ must be everywhere nonzero, i.e. $\msO_B\iso\msM$. Thus, we may view $\tau$ as a triple $[\tau_X,\tau_Y,1]$ where $\tau_X\in\Gamma(B,\msL^2)=\Hom(\inv[2]\msL,\msO_B)$ and $\tau_Y\in\Gamma(B,\msL^3)=\Hom(\inv[3]\msL,\msO_B)$. Since $\tau$ lands in the singular locus, applying the Jacobian criterion for smoothness to \cref{eqn:weier-eqn}, we conclude that counting generically singular Weierstrass equations \cref{eqn:weier-eqn} amounts to counting tuples $(a_1,a_2,a_3,a_4,a_6,\sigma_X,\sigma_Y)$ with $a_i\in\hom^0(\msL^i)$, $\sigma_X\in\hom^0(\msL^2)$ and $\sigma_Y\in\hom^0(\msL^3)$ satisfying
    \begin{align*}
        \sigma_Y^2+a_2\sigma_X\sigma_Y+a_3\sigma_X &= \sigma_X^3+a_2\sigma_X^2+a_4\sigma_X+a_6\\
        -a_1\sigma_Y &= 3\sigma_X^2+2a_2\sigma_X+a_4\\
        2\sigma_Y+a_2\sigma_X+a_3 &= 0
        .
    \end{align*}
    By the above equations, any such tuple is uniquely determined by the choice of $a_1,a_2,\sigma_X,\sigma_Y$ from whence the claim follows.
\end{proof}

\subsection{Counting Non-Minimal Weierstrass Curves}
Our main tool for counting non-minimal Weierstrass curves is the following description of their origin.
\begin{rem}\label{rem:source-nonmin-weier}
    All non-minimal Weierstrass curves of height $d\ge0$ arise in the following manner.
    
    Start with a \important{minimal} Weierstrass curve $(W'\xto{\pi'}B, S')$ of height $d'<d$ along with an effective divisor $D\in\Div(B)$ of degree $d-d'$, write $D=\sum_{i=1}^rn_i[b_i]$. Consider also the embedding $f:W'\into\P(\push\pi'\msO_{W'}(3S'))$. Choose open neighborhoods $U_i\subset B$ of $b_i$, for each $i\in\{1,\dots,r\}=:[r]$, which satisfy
    \begin{itemize}
        \item $f$ restricts to an embedding $W'_{U_i}\into\P^2_{U_i}$ with image cut out by
        \begin{equation*}
            Y^2Z + \ith a_1XYZ + \ith a_3YZ^2 = X^3 + \ith a_2X^2Z + \ith a_4XZ^2 + \ith a_6Z^3
        \end{equation*}
        ($\ith a_j\in\Gamma(U_i,\msO_B)$); and
        \item There exists some $\ith\varpi\in\Gamma(U_i,\msO_B)$ restricting to a uniformizer of $\msO_{B,b_i}$, but to a unit of $\msO_{B,b}$ for all $b\in U_i\sm\{b_i\}$; and
        \item $b_j\not\in U_i$ if $j\neq i$.
    \end{itemize}
    Let $U_0=B\sm\{b_1,\dots,b_r\}$ and $\Ith\varpi0:=1$. For each $i\in[r]$, let $\ith c_j:=\p{\ith\varpi}^{j\cdot n_i}\ith a_j\in\Gamma(U_i,\msO_B)$, and consider the curve
    \[W_i:=\b{Y^2Z + \ith c_1XYZ + \ith c_3YZ^2 = X^3 + \ith c_2X^2Z + \ith c_4XZ^2 + \ith c_6Z^3}\subset\P^2_{U_i}.\]
    Also, let $W_0:=W'\vert_{U_0}$. By construction, for distinct $i,j\in\b{0,1,\dots,r}$, there is a natural isomorphism $\alpha_{ij}\colon W_i\vert_{U_i\cap U_j}\iso W_j\vert_{U_i\cap U_j}$ which is given in coordinates as
    \[\alpha_{ij}\colon[X:Y:Z]\mapstoo\sq{\frac{\p{\Ith\varpi j}^{2n_j}}{\p{\ith\varpi}^{2n_i}}X:\frac{\p{\Ith\varpi j}^{3n_j}}{\p{\ith\varpi}^{3n_i}}Y:Z}\]
    (note above that $\ith\varpi,\Ith\varpi j\in\Gamma(U_i\cap U_j,\msO_B)^\by$).
    These isomorphisms visibly satisfy the cocycle condition, and so these $W_i$'s glue to form a global curve $W/B$. Furthermore, the $\alpha_{ij}$'s all respect the distinguished section ($[0:1:0]$) of each $W_i$ and so one obtains a corresponding section $S\subset W$. This $(W/B,S)$ is, by construction, a non-minimal Weierstrass curve of height $d$; in fact, if $\msL'$ is the Hodge bundle of $W'$, then $W$ has Hodge bundle $\msL'(D)$. Furthermore, the resulting $(W/B,S)$ is, up to isomorphism, independent of the choices made by \cref{cor:weier-disc-gen-isom}.
\end{rem}
The upshot of the above remark is that each generically smooth non-minimal Weierstrass \important{curve} (say, of height $d$) is determined by a unique choice of minimal Weierstrass \important{curve} (say, of height $e<d$) along with an effective divisor $D$ of degree $d-e$ keeping track of the non-minimality of the equation. We use this observation to obtain a recursive count as in \cite[Proposition 4.12]{dejong}.
\begin{prop}\label{prop:count-nonmin-weier}
    Fix some $\msL\in\Pic^d(B)$ with $d>N(g)$. The number of non-minimal (generically smooth) Weierstrass equations with Hodge bundle $\cong\msL$ is
    \[q^{6d+3(1-g)}\sum_{e=0}^{d-1}\sum_{\msM\in\Pic^e(B)}\p{\#\hom^0(B,\msL\otimes\dual\msM)-1}\WE_\msM.\]
\end{prop}
\begin{proof}
    As remarked above, each non-minimal (generically smooth) Weierstrass \important{curve} with Hodge bundle $\cong\msL$ is determined by a unique minimal Weierstrass curve $(W'/B,S')$, say with Hodge bundle $\msL'$, along with an effective divisor $D$ such that $\msL\cong\msL'(D)$. Recall each Weierstrass \important{curve} is cut out by $(q-1)q^{6d+3(1-g)}/\#\Aut(E)$ different Weierstrass \important{equations}, by \cref{cor:count-eqns-for-W}. Thus, the total number of (generically smooth) non-minimal Weierstrass equations with Hodge bundle $\cong\msL$ is
    \begin{equation}\label{eqn:non-min-weier-count-i}
        \sum_{e=0}^{d-1}\sum_{D\in\Div^{d-e}_+(B)}\sum_{\substack{E/K\\\msL_E\simeq\msL(-D)}}\frac{(q-1)q^{6d+3(1-g)}}{\#\Aut(E)}= (q-1)q^{6d+3(1-g)}\sum_{e=0}^{d-1}\sum_{D\in\Div_+^{d-e}(B)}\WE_{\msL(-D)},
    \end{equation}
    where $\Div_+^n(B)$ is the set of effective divisors of degree $n$ on $B$. Note that if $\msM\simeq\msL(-D)$, then $\msO(D)\simeq\msL\otimes\dual\msM$, so there are $(\#\hom^0(B,\msL\otimes\dual\msM)-1)/(q-1)$ different effective divisors $D'\sim D$. Thus, \cref{eqn:non-min-weier-count-i} equals
    \[(q-1)q^{6d+3(1-g)}\sum_{e=0}^{d-1}\sum_{\msM\in\Pic^e(B)}\frac{\#\hom^0(B,\msL\otimes\dual\msM)-1}{q-1}\WE_\msM.\qedhere\]
\end{proof}

\subsection{Counting Elliptic Curves}
\begin{prop}\label{prop:UW-recur}
    Fix any $\msL\in\Pic^d(B)$ with $d>N(g)$. Then,
    \[\UW_\msL=q^{16d+5(1-g)}-q^{8d+4(1-g)}-q^{6d+3(1-g)}\sum_{e=0}^{d-1}\sum_{\msM\in\Pic^e(B)}\p{\#\hom^0(B,\msL\otimes\dual\msM)-1}\WE_\msM,\]
    and
    \[\WE_\msL=\frac{\UW_\msL}{(q-1)q^{6d+3(1-g)}}.\]
\end{prop}
\begin{proof}
    The part after the word ``and'' is simply a restatement of \cref{cor:WE-from-UW}. To compute $\UW_\msL$, we make the simple observation that the (unweighted) number of generically smooth minimal Weierstrass equations is the total number of all Weierstrass equations minus the number of those which are generically singular minus the number of those which are generically smooth but non-minimal. With this in mind, the proposition follows combining \cref{rem:count-all-weier-eqns}, \cref{prop:count-gen-sing-weier}, and \cref{prop:count-nonmin-weier}.
\end{proof}
\begin{notn}
    We set
    \[Z_B(T):=\prod_{\t{closed }x\in B}\frac1{1-T^{\deg x}}=\sum_{n\ge0}\#\Sym^n(B)\cdot T^n,\]
    so $\zeta_B(s)=Z_B\p{q^{-s}}$.
\end{notn}
\begin{rem}\label{rem:comp-dJ-exact}
    In the case that $B=\P^1$, we have $Z_{\P^1}(T)=\sqinv{(1-T)(1-qT)}$, and \cite[Proposition 4.12]{dejong} gives an exact count
    \[\#\meM^{=d}_{1,1}(k(t))=q^{10d+1}\sq{1+\frac{1-q^{-8}-\inv[9]q+\inv[18]q}{q-1}-q^{-8d-1}+q^{-8d-3}}=\frac{q^{10d+2}}{(q-1)\zeta_{\P^1}(10)}-\frac{q^{2d+1}}{(q-1)\zeta_{\P^1}(2)},\]
    when $d\ge2$. 
\end{rem}
\begin{thm}\label{thm:EC-d-asymp}
    For any $\eps>0$, we have
    \[\#\meM_{1,1}^{=d}(K)=\#\Pic^0(B)\sq{\frac{q^{10d+2(1-g)}}{(q-1)\zeta_B(10)}-\frac{zq^{2d+(1-g)}}{(q-1)\zeta_B(2)}}+O_\eps\p{\p{q+\eps}^d}\]
    as $d\to\infty$, with implicit big-$O$ constant dependent on $\eps$. In particular,
    \[\#\meM_{1,1}^{=d}(K)\sim\#\Pic^0(B)\cdot\frac{q^{10d+2(1-g)}}{(q-1)\zeta_B(10)}.\]
\end{thm}
\begin{proof}
    We may and do assume throughout that $d\gg1$. First note that
    \begin{equation}\label{eqn:meM-from-NW}
        \#\meM_{1,1}^{=d}(K)=\sum_{\msL\in\Pic^d(B)}\WE_\msL=\frac{a_d}{(q-1)q^{6d+3(1-g)}}\twhere a_d:=\sum_{\msL\in\Pic^d(B)}\UW_\msL
        .
    \end{equation}
    (by \cref{cor:WE-from-UW}). 
    \cref{prop:UW-recur} tells us that
    \begin{align}\label{eqn:ad-init}
        a_d &= \#\Pic^0(B)\sq{q^{16d+5(1-g)}-q^{8d+4(1-g)}} - q^{6d+3(1-g)}\sum_{\msL\in\Pic^d(B)}\sum_{e=0}^{d-1}\sum_{\msM\in\Pic^e(B)}\p{\#\hom^0(B,\msL\otimes\inv\msM)-1}\WE_\msM \nonumber\\
        &= \#\Pic^0(B)\sq{q^{16d+5(1-g)}-q^{8d+4(1-g)}} - q^{6d+3(1-g)}(q-1)\sum_{e=0}^{d-1}\sum_{\msL\in\Pic^d(B)}\sum_{\msM\in\Pic^e(B)}\frac{\#\hom^0(B,\msL\otimes\inv\msM)-1}{q-1}\WE_\msM \nonumber\\
        &= \#\Pic^0(B)\sq{q^{16d+5(1-g)}-q^{8d+4(1-g)}} - q^{6d+3(1-g)}(q-1)\sum_{e=0}^{d-1}\sum_{\msN\in\Pic^{d-e}(B)}\sum_{\msM\in\Pic^e(B)}\frac{\#\hom^0(B,\msN)-1}{q-1}\WE_\msM \nonumber\\
        &= \#\Pic^0(B)\sq{q^{16d+5(1-g)}-q^{8d+4(1-g)}} - q^{6d+3(1-g)}(q-1)\sum_{e=0}^{d-1}\p{\sum_{\msN\in\Pic^{d-e}(B)}\frac{\#\hom^0(B,\msN)-1}{q-1}}\sum_{\msM\in\Pic^e(B)}\WE_\msM \nonumber\\
        &= \#\Pic^0(B)\sq{q^{16d+5(1-g)}-q^{8d+4(1-g)}} - q^{6d+3(1-g)}(q-1)\sum_{e=0}^{d-1}\#\Sym^{d-e}(B)\cdot\sum_{\msM\in\Pic^e(B)}\WE_\msM
        .
    \end{align}
    We would like to turn \cref{eqn:ad-init} into a recursive formula for $a_d$ by relating $a_e$ to $\sum_{\msM\in\Pic^e(B)}\WE_\msM$. Since $\WE_\msM$ is most mysterious when $\deg\msM\le N(g)$, we deal with these terms by observing that
    \begin{align}\label{ineq:bound-bad}
        \sum_{e=0}^{N(g)}\#\Sym^{d-e}(B)\cdot\sum_{\msM\in\Pic^e(B)}\WE_\msM
        &=\sum_{e=0}^{N(g)}\#\P^{d-e-g}(k)\cdot\#\Pic^{d-e}(B)\cdot\sum_{\msM\in\Pic^e(B)}\WE_\msM \nonumber\\
        &\le\#\Pic^0(B)\cdot\sum_{e=0}^{N(g)} q^{d+1-g}\cdot\sum_{\msM\in\Pic^e(B)}\WE_\msM \nonumber\\
        &\le \#\Pic^0(B)\cdot q^{d+1-g}\sum_{e=0}^{N(g)}\sum_{\msM\in\Pic^e(B)}\WE_\msM \nonumber\\
        &=O\p{q^d}
    \end{align}
    if $d=\deg\msL\gg1$. Thus, \cref{eqn:ad-init} can be simplified to
    \begin{align}\label{eqn:ad-recur}
        & a_d - \#\Pic^0(B)\sq{q^{16d+5(1-g)}-q^{8d+4(1-g)}} \nonumber\\ 
        =& - q^{6d+3(1-g)}(q-1)\sum_{e=0}^{d-1}\#\Sym^{d-e}(B)\cdot\sum_{\msM\in\Pic^e(B)}\WE_\msM \nonumber\\
        =& - q^{6d+3(1-g)}(q-1)\sq{O\p{q^d}+\sum_{e=N(g)+1}^{d-1}\#\Sym^{d-e}(B)\cdot\sum_{\msM\in\Pic^e(B)}\WE_\msM} \nonumber\\
        =&\, O\p{q^{7d}} - q^{6d+3(1-g)}(q-1)\sum_{e=N(g)+1}^{d-1}\#\Sym^{d-e}(B)\sum_{\msM\in\Pic^e(B)}\frac{\UW_\msM}{(q-1)q^{6e+3(1-g)}} \nonumber\\
        =&\, O\p{q^{7d}} - \sum_{e=N(g)+1}^{d-1}\#\Sym^{d-e}(B)\sum_{\msM\in\Pic^e(B)}q^{6(d-e)}\UW_\msM \nonumber\\
        =&\, O\p{q^{7d}} - \sum_{e=0}^{d-1}\#\Sym^{d-e}(B)\cdot q^{6(d-e)}a_e
        ,
    \end{align}
    where we implicitly used \cref{cor:WE-from-UW} (which required $\deg\msM>N(g)$) in the third equality, and that
    \[\sum_{e=0}^{2g-2}\#\Sym^{d-e}(B)q^{6(d-e)}\sum_{\msM\in\Pic^e(B)}\UW_\msM=O\p{q^{7d}},\]
    via reasoning as in \cref{ineq:bound-bad}, in the fifth equality. At this point, we introduce the sequence $c_d$ defined by $a_d=\#\Pic^0(B)q^{16d+5(1-g)}c_d$ and remark that (by \cref{eqn:meM-from-NW}) the theorem statement is equivalent to the claim that
    \begin{equation}\label{eqn:cd-goal}
        c_d = \zeta_B(10)^{-1} - q^{-8d-(1-g)}\zeta_B(2)^{-1} + O_\eps\p{\p{\inv[9]q+\eps}^d}
    \end{equation}
    for any $\eps>0$. To prove \cref{eqn:cd-goal}, consider the generating function $C(T):=\sum_{d\ge0}c_dT^d$. From \cref{eqn:ad-recur}, one obtains:
    \[\sum_{e=0}^d\#\Sym^{d-e}(B)\cdot q^{-10(d-e)}c_e = 1 - q^{-8d-(1-g)} + O\p{q^{-9d}}.\]
    Multiplying both sides by $T^d$ and summing over $d\ge0$, this becomes
    \[C(T)Z_B\p{\inv[10]qT}=\sum_{d\ge0}\sq{\sum_{e=0}^d\#\Sym^{d-e}(B)\inv[10(d-e)]q\cdot c_e}T^d=\frac1{1-T}-\frac{q^{g-1}}{1-T\inv[8]q}+\sum_{d\ge0}O\p{\inv[9d]q}T^d.\]
    Hence, $C(T)=Z_B\p{\inv[10]qT}^{-1}M(T)+E(T)$, where
    \[M(T):=\frac1{1-T}-\frac{q^{g-1}}{1-T\inv[8]q}\tand E(T):=Z_B\p{\inv[10]qT}^{-1}\sum_{d\ge0}O\p{q^{-9d}}T^d.\]
    Note that, by the Weil conjectures, $Z_B(T)=P(T)/[(1-T)(1-qT)]$ 
    for some polynomial $P(T)\in\C[T]$ all of whose roots $\alpha\in\C$ satisfy $\abs\alpha=1/\sqrt q$. Thus, $Z_B\p{\inv[10]qT}^{-1}$ is holomorphic on a disk of radius $q^{9.5}$, so $E(T)$ above is holomorphic on a disk of radius $q^9$. Now, set
    \begin{align*}
        F(T) &:= Z_B\p{\inv[10]qT}^{-1}M(T)-\sq{\frac{Z_B\p{\inv[10]q}^{-1}}{1-T}-\frac{q^{g-1}Z_B\p{\inv[2]q}^{-1}}{1-\inv[8]qT}}\\
        &= \frac{Z_B\p{q^{-10}T}^{-1}-Z_B\p{\inv[10]q}^{-1}}{1-T} + \frac{q^{g-1}\sq{Z_B\p{\inv[2]q}^{-1}}-Z_B\p{\inv[10]qT}^{-1}}{1-\inv[8]qT}
        .
    \end{align*}
    Above, note that the zeros of the numerators at $T=1$ and $T=q^8$, respectively, cancel out the simple zeros of the denominators there. Therefore, $F(T)$ has poles only where $Z_B\p{\inv[10]qT}^{-1}$ has poles, so $F(T)$ is holomorphic on a disk of radius $q^{9.5}$. Consequently,
    \[C(T)=Z_B\p{\inv[10]qT}^{-1}M(T)+E(T)=\frac{Z_B\p{\inv[10]q}^{-1}}{1-T}-\frac{q^{g-1}Z_B\p{\inv[2]q}^{-1}}{1-\inv[8]qT}+F(T)+E(T).\]
    Since $F(T)+E(T)$ is holomorphic on a disk of radius $q^9$, comparing Taylor coefficients shows that
    \[c_d=Z_B\p{\inv[10]q}^{-1}-q^{-8d-(1-g)}Z_B\p{\inv[2]q}^{-1}+O_\eps\p{\p{\inv[9]q+\eps}^d}\]
    for any $\eps>0$, proving the claim.
\end{proof}

\section{\bf Hyper-Weierstrass Curves}\label{sect:hW}
In \cref{sect:EC-count}, we computed $\#\meM_{1,1}^{\le d}(K)$, the denominator of \cref{eqn:AS-def}. In the current section, we turn our attention towards its numerator. In order to count 2-Selmer elements, we attach to them certain ``integral models'' whose definition and basic properties are the focus of this section.

\subsection{Definitions and Geometric Preliminaries}
The definition of the titular objects of this section is inspired by the following description of $2$-Selmer elements.
\begin{rem}\label{rem:selmer-interpretation}
    Let $K$ be as in \cref{set:main}. Let $E/K$ be an elliptic curve, and fix any $n\ge1$. Every $\alpha\in\Sel_n(E)\subset\hom^1(K,E[n])$ can be represented by a pair $(C,D)$ where $C$ is locally solvable $E$-torsor, and $D\subset C$ is an effective divisor of degree $n$. Explicitly, given such a pair, one associates to it the $E[n]$-torsor $T\subset C$ consisting of points $P\in C$ for which $nP\sim D$. Put another way, $T$ is the preimage of $\msO_C(D)\in\Pic^n(C)$ under the multiplication-by-$n$ map
    \[C\iso\rPic^1_{C/K}\too\rPic^n_{C/K}.\]
    Two such pairs $(C_1,D_1)$ and $(C_2,D_2)$ represent the same $n$-Selmer element if and only if there is an isomorphism $\phi:C_1\iso C_2$ of $E$-torsors for which $\msO_{C_1}(D_1)\simeq\pull\phi\msO_{C_2}(D_2)$. Finally, a pair $(C,D)$ represents the identity element of $\Sel_n(E)$ if and only if $D\sim nO$ for some $O\in C(K)$.
    
    This description of $n$-Selmer elements can be obtained, for example, by combining \cite[Section 1.1]{explicit-descent} with \cite[Remark after Proposition 2.3]{period-index}.
\end{rem}
For the following definition, it may be useful to recall our convention for the word `curve', stated in \cref{sect:conventions}.
\begin{defn}\label{defn:hW}
    For an arbitrary base scheme $B$, we let $\meH(B)$ denote the groupoid whose
    \begin{itemize}
        \item objects are pairs $(H\xto\pi B,D)$ of a curve $H/B$ along with a subscheme $D\subset H$ satisfying
        \begin{alphabetize}
            \item $\push\pi\msO_H\simeq\msO_B$ holds after arbitrary base change.
            \item $\omega_{H/B}\simeq\pull\pi\msL$ for some $\msL\in\Pic(B)$.
            \item $D\subset H/B$ is an effective relative Cartier divisor of degree $2$.

            By `of degree $2$', we mean that $D_b\subset H_b$ is locally principal of degree $2$ for all $b\in B$.
            \item The line bundle $\msO_H(D)$ is relatively ample over $B$.
        \end{alphabetize}
        \item (iso)morphisms $(H\xto\pi B,D)\to(H'\xto{\pi'}B,D')$ are isomorphisms $\phi:H\iso H'$ over $B$ such that
        \[\pull\phi\msO_{H'}(D')\in\msO_H(D)\otimes\pull\pi\Pic(B).\]
    \end{itemize}
    We call an element of $\meH(B)$ a \define{hyper-Weierstrass curve} (or simply an \define{hW curve}) over $B$.
\end{defn}
\begin{rem}
    Condition \bp a above implies that the fibers of $H/B$ are geometrically connected. Condition \bp b implies that that each fiber has trivial dualizing sheaf. Together, these two can be thought of as saying that $H/B$ is a family of genus 1 curves.
\end{rem}
\begin{rem}
    The definition of $\meH(B)$ was greatly inspired by the definition of the class $\mcA_{n,d}$ of curves appearing in \cite[Paragraph 5.2]{dejong}.
\end{rem}
\begin{rem}
    Classes of curves which satisfy the criteria of \cref{defn:hW} have been studied before, e.g. in \cite{liu:hyp-fr,liu:global-weier} (where they are called ``Weierstrass models'') and also \cite{Cremona_2010,sadek:genus-1} (where they are called ``Degree $2$ models of genus $1$ curves''). None of these citations considers them over an arbitrary base, and they each only consider such models whose generic fiber is smooth. Here, we allow of arbitrary bases and singular generic fibers, at least in setting up their basic geometric properties. Finally, since usual Weierstrass models of elliptic curves play a role in this paper, we opted to give these particular curves a different name.
\end{rem}
In this section (as well as the \cref{sect:selmer-groupoid-card}), we aim to develop a theory of hyper-Weierstrass curves akin to the theory of Weierstrass curves used in \cref{sect:EC-count} and summarized in \cref{thm:weier-sum}. Our first goal in such a development will be to show, analogous to \cref{thm:weier-sum}\bp4, that any hW curve can, locally on the base, be embedded in $\P(1,2,1)$ where it can be cut out by an equation of the form
\[Y^2+(a_0X^2+a_1XZ+a_2Z^2)Y=c_0X^4+c_1X^3Z+c_2X^2Z^2+c_3XZ^3+c_4Z^4.\]

\subsubsection{Fundamental Exact Sequences}\label{sect:fund-exact-seq}
\begin{set}
    Fix an arbitrary base scheme $B$.
\end{set}
\begin{lemma}\label{lem:length=deg}
    Let $k$ be a field, and let $X/k$ be a $k$-curve with trivial dualizing sheaf $\omega=\omega_{X/k}\cong\msO_X$ and with $\hom^0(X,\msO_X)=k$.
    Let $D\subset X$ be a Cartier divisor, and let $d = h^0(\msO_D)$. Assume that $d \ge 1$.  Then, $h^1(\msO_X(D))=0$, $h^0(\msO_X(D))=d$. If furthermore $d\ge2$, then $\msO_X(D)$ is globally generated.
\end{lemma}
\begin{proof}
    Consider the exact sequences
    \begin{equation}\label{ses:D}
        0\to\msO_X(-D)\to\msO_X\to\msO_D\to0\tand0\to\msO_X\to\msO_X(D)\to\msO_D(D)\to0.
    \end{equation}
    By duality, $\chi(\msO_X)=h^0(\msO_X)-h^0(\omega)=h^0(\msO_X)-h^0(\msO_X)=0$ since $\omega\cong\msO_X$. Hence, the exact sequence on the right of \cref{ses:D} gives
    \[\chi(\msO_X(D))=\chi(\msO_X)+\chi(\msO_D(D))=\chi(\msO_D(D)).\]
    Since $\msO_D$ is a skyscraper sheaf, we must have $\msO_D\simeq\msO_D(D)$. The above thus says
    \begin{equation}\label{eqn:chi(D)=d}
        \chi(\msO_X(D))=\chi(\msO_D(D))=\chi(\msO_D)=d.
    \end{equation}
    We now claim that $\hom^1(\msO_X(D))=0$. By duality, $h^0(\msO_X(D))=h^0(\omega_X\otimes\msO_X(-D))=h^0(\msO_X(-D))$. At the same time, the exact sequence on the left of \cref{ses:D} gives rise to
    \[0\too\hom^0(X,\msO_X(-D))\too\hom^0(X,\msO_X)\too\hom^0(D,\msO_D).\]
    The restriction map $k=\hom^0(X,\msO_X)\to\hom^0(D,\msO_D)$ is nonzero, so we conclude that $\hom^0(X,\msO_X(-D))=0$; hence also $\hom^1(X,\msO_X(D))=0$. Combining this with \cref{eqn:chi(D)=d}, we must have $h^0(\msO_X(D))=\chi(\msO_X(D))=d$. Finally, that $\msO_X(D)$ is globally generated when $d\ge2$ now follows from \cite[Lemma 8.4(a)]{dejong}.
\end{proof}

In developing a theory of hW curves, we will find it useful to also consider pairs $(X/B,D)$ satisfying properties \bp a -- \bp c (but not necessarily \bp d) of \cref{defn:hW}. Hence, we now name such curves.
\begin{defn}\label{defn:hawc}
    A \define{hyper almost-Weierstrass curve} (or simply \define{hawc}) is a pair $(X\xto\pi B,D)$ consisting of a curve $X/B$ along with a subscheme $D\subset X$ satisfying properties \bp a -- \bp c of \cref{defn:hW}.
\end{defn}
\begin{rem}
    We will see in \cref{cor:proj-hW} that hawcs give rise to hW curves. In \cref{sect:proof-prop}, we will apply this to show that every 2-Selmer element can be represented by an hW curve. In brief, \cref{rem:selmer-interpretation} will let us represent a 2-Selmer element by a pair $(C,D)$ with $C$ a locally solvable genus 1 curve, and $D\subset C$ a degree 2 divisor. In \cref{sect:proof-prop}, we will show that the minimal proper regular model of $C$ can be given the structure of a hawc, and so will give rise to an hW curve with $(C,D)$ as its generic fiber.
\end{rem}
\begin{lemma}\label{lem:push-E}
    Let $\pi:X\to B$ be a curve satisfying $\push\pi\msO_X\simeq\msO_B$ and $\omega_{X/B}\simeq\pull\pi\msL$ for some $\msL\in\Pic(B)$. Let $E\subset X$ be an effective relative Cartier divisor of degree $n\ge1$. Then,
    \begin{itemize}
        \item $\push\pi\msO_X(E)$ is a locally free sheaf of rank $n$ on $B$, whose formation commute with arbitrary base change; and
        \item $R^1\push\pi\msO_X(E)=0$.
    \end{itemize}
\end{lemma}
\begin{proof}
    Since $E$ is degree $n$ over the base, `Riemann-Roch of the fibers' (i.e. \cref{lem:length=deg}) shows that $h^0(E_b)=n$ and $h^1(E_b)=0$ for any $b\in B$. We now apply cohomology and base change, \cref{thm:coh-bc}, with $\msF=\msO_X(E)$ and $i=1$. Part \bp0 of \cref{thm:coh-bc} implies that $R^1\push\pi\msO_X(E)=0$, so part \bp2 implies that $\phi_b^0$ (with notation as in the theorem statement) is surjective for all $b$. Given this, we can apply \cref{thm:coh-bc} a second time, now with $i=0$ and $\msF=\msO_X(E)$. Part \bp2 shows that $\push\pi\msO_X(E)$ is a vector bundle on $B$, part \bp1 shows that its formation commutes with arbitrary base change, and part \bp0 shows that it has rank $h^0(E_b)=n$.
\end{proof}
\begin{rem}
    Let $(X\xto\pi B,D)$ be a hawc. We will most commonly apply \cref{lem:push-E} to the divisors $nD\subset X$, for $n\ge1$. In this context, \cref{lem:push-E} says, among other things, that $\push\pi\msO_X(nD)$ is a vector bundle of rank $2n$.
\end{rem}

\begin{prop}\label{prop:model-exact-seqs}
    Let $(X\xto\pi B,D)$ be a hawc with Hodge bundle $\msL:=\push\pi\omega_{X/B}$. For any integer $n\ge2$, there is an exact sequence
    \begin{equation}\label{ses:mult-by-1}
        0\too\push\pi\msO_X((n-1)D)\otimes\det(\push\pi\msO_X(D))\too\push\pi\msO_X(nD)\otimes\push\pi\msO_X(D)\too\push\pi\msO_X((n+1)D)\too0
    \end{equation}
    of vector bundles on $B$, where the right map above is the natural multiplication map. When $n=1$, there is an exact sequence
    \begin{equation}\label{ses:mult-1-1}
        0\too\Sym^2(\push\pi\msO_X(D))\too\push\pi\msO_X(2D)\too\inv\msL\otimes\det(\push\pi\msO_X(D))\too0
    \end{equation}
    of vector bundles on $B$, where the left map is the natural multiplication map.
\end{prop}
\begin{proof}
    Note that $\msO_{X_b}(D_b)$ is globally generated for all $b\in B$ by \cref{lem:length=deg}. Hence, \cref{lem:push-E,lem:glob-gen} tell us that the natural counit map is a surjection $\pull\pi\push\pi\msO_X(D)\onto\msO_X(D)$. Consider the exact sequence
    \begin{equation}\label{eqn:D-gg}
        0\too\msK\too\pull\pi\push\pi\msO_X(D)\too\msO_X(D)\too0,
    \end{equation}
    and note that $\msK$ is a kernel of a surjection between vector bundles, and so a vector bundle itself. Since $\msO_X(D)$ is a line bundle while $\pull\pi\push\pi\msO_X(D)$ is rank 2 (by \cref{lem:push-E}), $\msK$ is a line bundle, so we can take determinants to compute $\msK\simeq\msO_X(-D)\otimes\pull\pi\det(\push\pi\msO_X(D))$. 
    
    Now, fix an integer $n\ge1$. Twisting \cref{eqn:D-gg} by $\msO_X(nD)$, pushing forward the resulting sequence, and applying the projection formula\footnote{$R^k\push\pi\p{\msF\otimes\pull\pi\msG}\simeq R^k\push\pi\msF\otimes\msG$ when $\msG$ is a vector bundle, \cite[Exercise III.8.3]{hart}} to both $\msK(nD)\simeq\msO_X((n-1)D)\otimes\pull\pi\det(\push\pi\msO_X(D))$ and $\msO_X(nD)\otimes\pull\pi\push\pi\msO_X(D)$, we obtain the exact sequence
    \[\begin{tikzcd}
        0\ar[r]&\push\pi\msO_X((n-1)D)\otimes\det(\push\pi\msO_X(D))\ar[r]&\push\pi\msO_X(nD)\otimes\push\pi\msO_X(D)\ar[r]&\push\pi\msO_X((n+1)D)\\
        \ar[r]&R^1\push\pi\msO_X((n-1)D)\otimes\det(\push\pi\msO_X(D))\ar[r]&R^1\push\pi\msO_X(nD)\otimes\push\pi\msO_X(D)
        .
    \end{tikzcd}\]
    By \cref{lem:push-E}, $R^1\push\pi\msO_X(nD)=0$. If $n\ge2$, then also $R^1\push\pi\msO_X((n-1)D)=0$, so the sequence becomes
    \[0\too\push\pi\msO_X((n-1)D)\otimes\det(\push\pi\msO_X(D))\too\push\pi\msO_X(nD)\otimes\push\pi\msO_X(D)\too\push\pi\msO_X((n+1)D)\too0,\]
    as claimed. If $n=1$, then the map $\push\pi\msO_X(D)\otimes\push\pi\msO_X(D)\to\push\pi\msO_X(2D)$ factors through $\Sym^2(\push\pi\msO_X(D))$ and -- recalling that $R^1\push\pi\msO_X\simeq\inv\msL$ by duality -- we obtain the exact sequence
    \[\begin{tikzcd}
        \Sym^2(\push\pi\msO_X(D))\ar[r]&\push\pi\msO_X(2D)\ar[r]&\inv\msL\otimes\det\p{\push\pi\msO_X(D)}\ar[r]&0
    \end{tikzcd}\]
    Now, we claim that the map $\Sym^2(\push\pi\msO_X(D))\to\push\pi\msO_X(2D)$ is injective. This follows from the fact that the kernel of a surjection between vector bundles is a vector bundle. Indeed, exactness of the sequence tells us that the image of this map is the rank 3 vector bundle $\ker\p{\push\pi\msO_X(2D)\onto\inv\msL\otimes\det(\push\pi\msO_X(D))}$, and so its kernel is the rank 0 vector bundle
    \[\ker\p{\Sym^2(\push\pi\msO_X(D))\onto\ker\p{\push\pi\msO_X(2D)\to\inv\msL\otimes\det(\push\pi\msO_X(D))}}=0.\]
    Hence, the sequence above is exact on the left, finishing the proof of the claim.
\end{proof}
\begin{cor}\label{cor:det(En)}
    Let $(X\xto\pi B,D)$ be a hawc with Hodge bundle $\msL:=\push\pi\omega_{X/B}$. Let $\msD:=\det(\push\pi\msO_X(D))$. Then,
    \[\det\p{\push\pi\msO_X(nD)}\simeq\msD^{n^2}\otimes\msL^{1-n}\tforall n\ge1.\]
\end{cor}
\begin{proof}
    This is true for $n=1$ by definition. For $n=2$, this then follows from taking determinants in \cref{ses:mult-1-1}. For $n>2$, one inductively takes determinants in \cref{ses:mult-by-1}.
\end{proof}
\cref{prop:model-exact-seqs} (in particular, surjectivity of the relevant multiplication morphisms when $n\ge2$) is our main workhorse for obtaining local equations for hyper-Weierstrass curves. 

\subsubsection{Local Projective Embeddings}
We now obtain our local models for hW curves (\cref{thm:loc-model}), as mentioned above \cref{sect:fund-exact-seq}.
\begin{thm}\label{thm:loc-model}
    Let $(X\to B,D)$ be a hawc. Let
    \[H:=\rProj_B\p{\bigoplus_{n\ge0}\push\pi\msO_X(nD)}.\]
    Let $p:X\to H$ be the natural map, induced by the surjections $\pull\pi\push\pi\msO_X(nD)\onto\msO_X(nD)$ (see \cite[\href{https://stacks.math.columbia.edu/tag/01O8}{Tag 01O8}]{stacks-project}), and let $D_H\subset X$ be the (scheme-theoretic) image of $D$ under $p$.
    
    Then, each point of the base $B$ has an affine neighborhood $U=\spec R$ such that $H_U\to U$ is isomorphic over $U$ to the subscheme of $\P(1,2,1)_U$ defined by
    \begin{equation}\label{eqn:local-model}
        Y^2+(uX^2+vXZ+wZ^2)Y=aX^4+bX^3Z+cX^2Z^2+dXZ^3+eZ^4
    \end{equation}
    for some $u,v,w,a,b,c,d,e\in R$. Furthermore, we may choose the coordinates $X,Y,Z$ above so that $Z$ extends to a global section of $\msO_H(1)$, and so that $D_H$ is the divisor $\{Z=0\}$; hence, $\msO_H(1)\simeq\msO_H(D_H)$.
    
    Finally, if $D\subset X$ is relatively ample over $B$, i.e. if $(X\to B,D)\in\meH(B)$, then $X\simeq H$ as $B$-schemes, so $(X\to B,D)$ itself satisfies the above properties.
\end{thm}
\begin{proof}
    Each point of $B$ has an affine neighborhood $U=\spec R$ above which both $\push\pi\msO_X(D)$ and $\msL$ trivialize, so we may and do assume wlog that $B=U=\spec R$. Even after passing to this case, we continue to write $U$ for the base instead of $B$ in order to emphasize the fact that we're working over an affine.
    
    We want to carefully construct $x,z\in\Gamma(U,\push\pi\msO_X(D))$ and $y\in\Gamma(U,\push\pi\msO_X(2D))$ which will give the coordinates on our weighted projective space. For this, we first consider the exact sequence $0\to\msO_X\to\msO_X(D)\to\msO_D(D)\to0$ which pushes forward to
    \[0\too\msO_U\too\push\pi\msO_X(D)\too\push\pi\msO_D(D)\too\inv\msL\too0.\]
    Let $\mfQ:=\coker(\msO_U\to\push\pi\msO_X(D))=\ker(\push\pi\msO_D(D)\onto\inv\msL)$, and note that this is a vector bundle. Consider the exact sequence (recall that $U=\spec R$ is affine)
    \[\begin{tikzcd}
        0\ar[r]&\Gamma(U,\msO_U)\ar[d,equals]\ar[r]&\Gamma(U,\push\pi\msO_X(D))\ar[d,equals]\ar[r]&\Gamma(U,\msQ)\ar[r]&0.\\
        &R&R^2
    \end{tikzcd}\]
    We let $z\in\Gamma(U,\push\pi\msO_X(D))\simeq R^2$ be the image of $1\in R$ under the first map above. Since $\Gamma(U,\msQ)$ is a projective $R$-module, taking determinants shows that $\Gamma(U,\msQ)\simeq R$ is free, so we can and do fix some $x\in\Gamma(U,\msO_X(D))$ with image generating $\Gamma(U,\msQ)$. Hence, $x,z\in\Gamma(U,\push\pi\msO_X(D))$ give a basis. We went through the trouble of carefully choosing a particular $z$ to be part of our basis in order to know that the subscheme $\{z=0\}\subset X$ is equal to $D$.
    
    Now, by \cref{prop:model-exact-seqs}, the cokernel of the map $\Sym^2\Gamma(U,\push\pi\msO_X(D))\into\Gamma(U,\push\pi\msO_X(2D))$ is free, so we can find some $y\in\Gamma(U,\push\pi\msO_X(2D))$ such that $x^2,xz,z^2,y$ form a basis for $\Gamma(U,\push\pi\msO_X(2D))$. We want to use these to produce a basis for $\Gamma(U,\push\pi\msO_X(3D))$. First note that the multiplication map
    \[\Gamma(U,\push\pi\msO_X(2D))\otimes\Gamma(U,\push\pi\msO_X(D))\onto\Gamma(U,\push\pi\msO_X(3D))\]
    is surjective due to \cref{prop:model-exact-seqs}. Since the domain is $R\angles{x^2,xz,z^2,y}\otimes R\angles{x,z}$, we see by inspection that this map factors through a map
    \[R\angles{x^3,x^2z,xz^2,xy,z^3,zy}\too\Gamma(U,\push\pi\msO_X(3D))\]
    which is moreover necessarily a surjection. At the same time, \cref{lem:push-E} tells us that $\push\pi\msO_X(3D)$ is a rank 6 vector bundle on $U$, so the above is a surjection of equal rank projective $R$-modules, and hence an isomorphism.
    A similar argument shows that
    \[R\angles{x^4,x^3z,x^2z^2,xz^3,z^4,x^2y,xzy,z^2y}\isoo\Gamma(U,\push\pi\msO_X(4D)).\]
    Since we have $y^2\in\Gamma(U,\push\pi\msO_X(4D))$ as well, there must be some relation of the form
    \begin{equation}\label{eqn:local-model-ii}
        y^2+(ux^2+vxz+wz^2)y=ax^4+bx^3z+cx^2z^2+dxz^3+ez^4
    \end{equation}
    with $u,v,w,a,b,c,d,e\in R$. 
    
    Now, a straightforward induction argument using logic as above shows that the natural map
    \[B_n:=\bigoplus_{a+2b=n}\p{\Sym^a\Gamma(U,\push\pi\msO_X(D))\otimes Ry^b}\too\Gamma(U,\push\pi\msO_X(nD))=:A_n\]
    is surjective for all $n\ge0$. Thus, letting $\msB:=\bigoplus_{n\ge0}B_n$ and $\msA=\bigoplus_{n\ge0}A_n$, the natural surjection $\msB\onto\msA$ of graded $R$-algebras gives rise to a closed embedding
    \[H=\Proj\msA\into\Proj\msB\simeq\P(1,2,1)_U\]
    upon taking $\Proj$. Above, $\Proj\msB\simeq\P(1,2,1)_U$ since it is easy to check that the natural graded map
    \[R[X,Y,Z]\too\msB\t{ sending }X\mapsto x,Y\mapsto y,Z\mapsto z\]
    (in particular, $X,Z$ are degree 1, while $Y$ is degree 2) is an isomorphism, e.g. since it is visibly surjective and its graded pieces have the same rank. Combining this observation with the relation \cref{eqn:local-model-ii}, we see that we have a natural surjection
    \[\frac{R[X,Y,Z]}{\p{Y^2+(uX^2+vXZ+wZ^2)Y-(aX^4+bX^3Z+cX^2Z^2+dXZ^3+eZ^4)}}\xonto\sim\msA\]
    which is once more an isomorphism as both sides have $n$th graded piece of rank $2n$. This exactly says that $H=\Proj\msA$ is the subscheme of $\P(1,2,1)_U$ cut out by an equation of the form \cref{eqn:local-model}, as desired.
    
    Finally, in the case that $X$ is hyper-Weierstrass, we have $X\simeq H=\Proj\msA$ by \cite[\href{https://stacks.math.columbia.edu/tag/01Q1}{Tag 01Q1}]{stacks-project} (+ $X$ being proper) since $D\subset X$ is relatively ample.
\end{proof}

Among other things, \cref{thm:loc-model} describes a local model \cref{eqn:local-model} for hyper-Weierstrass curves. We now establish a converse by showing that hyper-Weierstrass curves are exactly those with such a local model.
\begin{thm}\label{thm:loc-model-conv}
    Let $H\xto\pi B$ be a $B$-scheme equipped with a closed subscheme $D\subset H$ satisfying the following property: Every point of $B$ has an affine neighborhood $U=\spec R$ above which $H_U\to U$ becomes isomorphic to a subscheme of $\P(1,2,1)$ defined by an equation of the form \cref{eqn:local-model} such that $D_U\subset H_U$ is the subscheme $\{Z=0\}$. Then, $(H\xto\pi B,D)\in\meH(B)$.
\end{thm}
\begin{proof}
    Every part of \cref{defn:hW} is local on the base, so we may and do assume that $B=\spec R$ is affine, that
    \[H=\b{Y^2+(uX^2+vXZ+wZ^2)Y=aX^4+bX^3Z+cX^2Z^2+dXZ^3+eZ^4}\subset\P(1,2,1)_R\]
    (for some $u,v,w,a,b,c,d,e\in R$), and that $D=\b{Z=0}\subset H$. Note $H$ is visibly proper and finitely presented over $R$. We first show that $H$ is flat over $R$. Note that it is covered by the open sets $\{X\neq0\}$ and $\{Z\neq0\}$. By symmetry, to show that it is flat over $R$, it suffices to show that
    \[A:=\frac{R[x,y]}{(f(x,y))}\twhere f(x,y)=y^2+(ux^2+vx+w)y-(ax^4+bx^3+cx^2+dx+e)\]
    is a flat $R$-module. This is the case simply because $A\cong R[x]\oplus R[x]y\cong R[x]^{\oplus2}$ as an $R$-module, and $R[x]$ is $R$-flat.
    
    We next show that the fibers of $\pi$ are Gorenstein curves with trivial dualizing sheaves. For any $b\in B$, we simply note that the open subset $\{X\neq 0\}\cup\{Z\neq 0\}\subset\P(1,2,1)_{\kappa(b)}$ is a regular scheme containing $H_b$, so $H_b$ is a local complete intersection, and hence a Gorenstein, 1-dimensional scheme. In particular, each $H_b$ is a `weighted hypersurface of degree $4$' in the sense of \cite{wpv}, so \cite[Theorem 3.3.4]{wpv} (see \cref{cor:ds-hs}) tells us that $\omega_H\cong\msO_H$. Furthermore, from our explicit description of $H_b\into\P(1,2,1)_{\kappa(b)}$, one can show that $\hom^0(H_b,\msO_{H_b})=\kappa(b)$, e.g. by computing \v Cech cohomology with respect to the affine open covering $\{X\neq0\}\cup\{Z\neq0\}=H_b$. Thus, \cref{lem:fibral-gs+ds} tells us that $\push\pi\msO_H=\msO_B$ holds after arbitrary base change, and that $\omega_{H/B}\in\pull\pi\Pic(B)$.
    
    What remains is to show that $D\subset X/B$ is a relatively ample effective Cartier divisor of degree 2 over $B$. Since $D=\{Z=0\}\subset H$, it is certainly an effective Cartier divisor on $H$. Furthermore, $D$ is flat over $B$ by essentially the same argument used to show that $H$ is flat over $B$, so $D$ is in fact an effective relative Cartier divisor over $B$. As $D\subset\{X\neq0\}\subset H$, we see that for any $b\in B$
    \[D_b\simeq\spec\frac{\kappa(b)[y]}{(y^2+uy-a)}\]
    is a degree 2 scheme, so $D$ is of degree $2$. Finally, $\msO_H(D)\simeq\msO_H(1):=\msO_{\P(1,2,1)}(1)\vert_H$ is indeed a relatively ample line bundle over $B$.
\end{proof}
\begin{cor}\label{cor:proj-hW}
    Let $(X\to B,D)$ be a hawc. Then,
    \[H:=\rProj_B\p{\bigoplus_{n\ge0}\push\pi\msO_X(nD)}\]
    equipped with the scheme-theoretic image of $D$ under the natural map $X\to H$ is a hyper-Weierstrass curve over $B$ such that $\msO_H(1)\simeq\msO_H(D)$.
\end{cor}
\begin{proof}
    Combine \cref{thm:loc-model,thm:loc-model-conv}.
\end{proof}

Finally, for later use in \cref{sect:proof-prop}, we establish a few facts about the map $p$ of \cref{thm:loc-model} when the base $B$ is (the spectrum of) a field.

\begin{prop}\label{prop:contract}
    Let $F$ be a field, and let $(C,D)$ be a hawc over $F$. Let $S:=\bigoplus_{n\ge0}\hom^0(C,\msO_C(nD))$ and $X:=\Proj S$. Consider the natural morphism $p:C\to X$ induced by the identity map $S=\bigoplus_{n\ge0}\hom^0(C,\msO_C(nD))$ via \cite[\href{https://stacks.math.columbia.edu/tag/01N8}{Tag 01N8}]{stacks-project} with $d=1$; in applying this citation, we use \cref{lem:length=deg} to know that $\msO_C(D)$ is generated by global sections. Then,
    \begin{alphabetize}
        \item The locus $U_1:=\bigoplus_{f\in S_1}D_+(f)$ referenced in the citation is all of $X$. Consequently, the citation gives an isomorphism $\alpha:\pull p\msO_X(1)\iso\msO_C(D)$.
        \item The induced maps $\hom^0(X,\msO_X(n))\to\hom^0(C,\msO_C(nD))$ are isomorphisms for all $n\ge1$.
        \item The accompanying map $\msO_X\to\push p\msO_C$ is an isomorphism of sheaves.
        \item The induced map $\hom^0(C,\omega_C)\to\hom^0(X,\omega_X)$, dual to $\hom^1(X,\msO_X)\to\hom^1(C,\msO_C)$, is an isomorphism.
    \end{alphabetize}
\end{prop}
\begin{proof}
    First, let $D_X\subset X$ be the scheme-theoretic image of $D\subset C$ under $p:C\to X$. Note that \cref{cor:proj-hW} tells us that $(XFk,D_X)\in\meH(\spec F)$ is an hW curve with $\msO_X(D_X)\simeq\msO_X(1)$. In particular, by applying \cref{lem:length=deg} twice,
    \[h^0(X,\msO_X(n))=h^0(X,\msO_X(nD_X))=2n=h^0(C,\msO_C(nD))\tforall n\ge1.\]
    Similarly, we have $h^0(X,\msO_X)=1=h^0(C,\msO_C)$ by assumption on $C$ and since $X$ is hyper-Weierstrass over $k$. Hence, $h^0(X,\msO_X(n))=h^0(C,\msO_C(nD))$ for all $n\ge0$.
    
    \bp a We first show that $X$ is covered by distinguished affines coming from elements in degree 1, i.e. that $X=U_1$. \cref{thm:loc-model} gives an embedding $X\into\P(1,2,1)_k\simeq\Proj k[X,Y,Z]$, with $X,Z$ in degree 1 and $Y$ in degree 2, so that $X$ is cut out by an equation of the form \cref{eqn:local-model}. Consequently,  $X=D_+(X)\cup D_+(Z)\subset U_1\subset X$, so $X=U_1$ as claimed. \cite[\href{https://stacks.math.columbia.edu/tag/01N8}{Tag 01N8}]{stacks-project} then tells us that $\pull p\msO_X(1)\simeq\msO_C(D)$. 
    
    \bp b By taking powers, $\pull p\msO_X(n)\simeq\msO_C(nD)$ for all $n\in\Z$. We tensor the map $\msO_X\to\push p\msO_C$ with $\msO_X(n)$, apply the projection formula, and then apply this isomorphism $\pull p\msO_X(n)\simeq\msO_C(nD)$ in order to obtain
    \begin{equation}\label{O(n)->O(nD)}
        \msO_X(n)\too\push p\msO_C\otimes\msO_X(n)\simeq\push p\p{\msO_C\otimes\pull p\msO_X(n)}\simeq\push p\msO_C(nD).
    \end{equation}
    Taking global section, we obtain a map
    \[\Gamma(X,\msO_X(n))\too\Gamma(C,\msO_C(nD))\]
    for all $n\ge0$, which is furthermore surjective as \cite[\href{https://stacks.math.columbia.edu/tag/01N8}{Tag 01N8}]{stacks-project} shows it fits in a commutative diagram
    \[\compdiag{S_n}{}{\Gamma(X,\msO_X(n))}{}{\Gamma(C,\msO_C(nD)).}{\id}\]
    We proved earlier that $\dim_F\Gamma(X,\msO_X(n))=\dim_F\Gamma(C,\msO_C(nD))$, so we in fact have isomorphisms $\Gamma(X,\msO_X(n))\iso\Gamma(C,\msO_C(nD))$ for all $n\ge0$.
    
    \bp c To show that the induced map $\msO_X\to\push p\msO_C$ is an isomorphism. we simply observe that \bp b tells us that \cref{O(n)->O(nD)} induces the following isomorphism of coherent sheaves on $X=\Proj S$ (see \cite[\href{https://stacks.math.columbia.edu/tag/0AG5}{Tag 0AG5}]{stacks-project} for the outer isomorphisms)
    \[\msO_X\simeq\p{\bigoplus_{n\ge0}\Gamma(X,\msO_X(n))}^\sim\iso\p{\bigoplus_{n\ge0}\Gamma(X,\push p\msO_C\otimes\msO_X(n))}^\sim\simeq\push p\msO_C.\]
    
    \bp d Finally, we will show that $p$ induces an isomorphism $\hom^0(C,\omega_C)\iso\hom^0(X,\omega_X)$. The Leray spectral sequence $\hom^p(X,R^q\push p\msO_C)\implies\hom^{p+q}(C,\msO_C)$ gives an embedding $\hom^1(X,\push p\msO_C)\into\hom^1(C,\msO_C)$. By \bp c, this is $\hom^1(X,\msO_X)\into\hom^1(C,\msO_C)$. The dual of this is a surjection $\hom^0(C,\omega_C)\onto\hom^0(X,\omega_X)$. $\hom^0(C,\omega_C)=\hom^0(C,\msO_C)=k$ by assumption on $C$ and similarly $\hom^0(X,\omega_X)=\hom^0(X,\msO_X)=k$ since $X$ is hyper-Weierstrass over $k$, so we conclude that $\hom^0(C,\omega_C)\xonto\sim\hom^0(X,\omega_X)$ is in fact an isomorphism.
\end{proof}
\begin{lemma}\label{lem:contract}
    Let $F$ be a field, and let $(C,D)$ be a hawc over $F$. Set $S\coloneqq\bigoplus_{n\ge0}\hom^0(C,\msO_C(nD))$ and $H\coloneqq\Proj S$. Consider the natural map $p\colon C\to H$.

    Let $\{C_i\}_{i\in I}$ denote the irreducible components of $C$ which $D$ does \important{not} meet, and set $U\coloneqq C\sm\bigcup_{i\in I}C_i\openset C$. Then, $p$ restricts to an open immersion $U\into H$ with dense image. In other words, $p\colon C\to H$ is a contraction of the components of $C$ not meeting $D$.
\end{lemma}
\begin{proof}
    By \cref{prop:contract}\bp c, the induced map $\msO_H\to\push p\msO_C$ is an isomorphism; in particular, $p$ has connected fibers. Therefore, by Zariski's Main Theorem \cite[Theorem 2.3/2]{neron} (see also \cite[Proposition 3.8]{akhil-zmt}), to prove the claim, it suffices to show that $p\vert_U\colon U\to H$ is quasi-finite. Let $C_0\into C$ denote an irreducible component which \important{does} meet $D$; we only need to show that $C_0$ is not contracted by $p$. Fix a point $x\in C_0\cap D$. \cref{lem:length=deg} tells us that $\msO_C(D)$ is globally generated, so there exists a section $\sigma\in\hom^0(C,\msO_C(D))$ which is nonvanishing at $x$. On the other hand, the section $\sigma'\coloneqq1\in\hom^0(C,\msO_C(D))$ does vanish at $x$. Thus, the sets $\inv p(D_+(\sigma))\cap C_0$ and $\inv p(D_+(\sigma'))\cap C_0$ -- i.e. the loci $C_{0,\sigma}$ and $C_{0,\sigma'}$ on $C_0$ where $\sigma,\sigma'$ are nonvanishing -- are nonempty and distinct, forcing $p\vert_{C_0}$ to not be constant.
\end{proof}

\subsection{Connection to 2-Selmer}
Throughout this section, We work in the context of \cref{set:main}. In particular, $k$ is a finite field, and (unless otherwise stated) $B$ is a smooth $k$-curve of genus $g$ with function field $K=k(B)$.

\subsubsection{Selmer Groupoid}
We want to reduce counting Selmer elements to counting hyper-Weierstrass curves. In either case, when counting these objects, we do so in a weighted fashion, e.g. for $E$ an elliptic curve, we will count some $\alpha\in\Sel_2(E)$ with weight $1/\#\Aut(E)$. Thus, these Selmer elements are best thought of as belong not to some set, but instead to some groupoid. With that in mind, we take a moment to set up this language before formally relating 2-Selmer elements to hyper-Weierstrass curves.

The following definition is inspired by \cref{rem:selmer-interpretation}.
\begin{defn}\label{defn:Selmer-groupoid}
    Fix an integer $n\ge1$. The \define{$n$-Selmer groupoid} (over $K$) is the groupoid whose
    \begin{itemize}
        \item objects are tuples $(C,E,\rho,D)$ where
        \begin{itemize}
            \item $E/K$ is an elliptic curve.
            \item $C/K$ is a locally solvable genus 1 curve.
            \item $\rho:C\by E\to C$ is a group action making $C$ into an $E$-torsor. 
            
            We will write $c\cdot x:=\rho(c,x)$ when $c\in C(S)$ and $x\in E(S)$ for any $K$-scheme $S$.
            \item $D\subset C$ is a degree $n$ effective divisor, defined over $K$.
        \end{itemize}
        \item (iso)morphisms $(C,E,\rho,D)\to(C',E',\rho',D')$ are pairs $\p{\phi:C\iso C',\psi:E\iso E'}$ where
        \begin{itemize}
            \item $\psi$ is an isomorphism of $K$-group schemes.
            \item $\phi(c\cdot x)=\phi(c)\cdot\psi(x)$ for all $c\in C,x\in E$.\footnote{by which we really mean $c\in C(S)$ and $x\in E(S)$ for $S$ any $K$-scheme}
            \item $\pull\phi\msO_{C'}(D')\simeq\msO_C(D).$
        \end{itemize}
    \end{itemize}
    We denote this groupoid by $\CSel_n=\CSel_{n,K}$. Given, any $(C,E,\rho,D)\in\CSel_n$, we define its \define{height} to be the height of $E$, i.e. $\Ht(C,E,\rho,D):=\Ht(E)$. Furthermore, we say $(C,E,\rho,D)\in\CSel_n$ is \define{trivial} if $D\sim nP$ for some $P\in C(K)$, i.e. if $(C,D)$ represents the identity element of $\Sel_n(E)$.
\end{defn}
\begin{ex}\label{ex:triv-sel}
    Say $(C,E,\rho,D)\in\CSel_n$ is trivial, and choose $P\in C(K)$ such that $D\sim nP$. Let $O\in E(K)$ denote the identity element. Then, $(C,E,\rho,D)\simeq(E,E,\rho_E,nO)$, where $\rho_E:E\by E\to E$ is $E$'s multiplication map. Indeed, one can $\phi:C\iso E$ to be the ``subtract $P$'' map defined by
    \[P\cdot\phi(c)=c\tforany c\in C,\]
    and can take $\psi=\id_E$.
\end{ex}
\begin{rem}\label{rem:height-torsor}
    Let $E/K$ be an elliptic curve, and let $C/K$ be a locally solvable $E$-torsor. Then, $\Ht(E)=\Ht(C)$. One can see this, for example, in \cite[Section 5.12]{dejong}, which proves this when $B=\P^1$, but whose argument works for any $B$.
\end{rem}
\begin{lemma}\label{lem:sel-aut}
    Fix some $n\ge1$ as well as some $(C,E,\rho,D)\in\CSel_n$. Let $\alpha=[(C,D)]\in\Sel_n(E)$ be the corresponding Selmer element. Then, there is an exact sequence
    \[0\too E[n](K)\too\Aut_{\CSel_n}(C,E,\rho,D)\too\Stab_{\Aut(E)}(\alpha)\too0,\]
    where $\Aut(E)\actson\Sel_n(E)$ in the natural way.
\end{lemma}
\begin{proof}
    Consider the map $f:\Aut_{\CSel_n}(C,E,\rho,D)\to\Aut(E)$, $(\phi,\psi)\mapsto\psi$. We will show that is has kernel $E[n](K)$ and image $\Stab_{\Aut(E)}(\alpha)$. 

    First say $(\phi,\psi)\in\Aut(C,E,\rho,D)$ is an automorphism with $\psi=\id_E$. Then, $\phi(c\cdot x)=\phi(c)\cdot x$ for any $c\in C,x\in E$, so $\phi$ is an isomorphism of $E$-torsors. Thus, there is some $x_0\in E(K)$ so that $\phi(c)=c\cdot x_0$ for all $c\in C$. We claim that $x_0$ must be $n$-torsion. The action $\rho:C\by E\to C$ induces an isomorphism $f:E\iso\rPic^0_{C/K}$ so that $E$'s action on $C$ correspond to $\rPic^0_{C/K}$'s natural action on $\rPic^1_{C/K}\simeq C$ (coming from adding a degree 0 line bundle). Thus, $\phi$ acts on $\Pic^n_{C/K}\ni\msO_C(D)$ via translation by $nx_0$. This action is trivial if and only if $x_0\in E[n](K)$.
    
    Fix some $\psi\in\Aut(E)$. When is $\psi$ in the image of $f$? Well, consider some $E$-torsor structure $C_1=(C,\rho_1)$ on $C$, by which we mean an action $\rho_1:C\by E\to C$ making $C$ into an $E$-torsor. Let $C_2=(C,\rho_2)$ be another $E$-torsor structure on $C$. By definition, given an automorphism $\phi:C\iso C$, the pair $(\phi,\psi)\in\Aut(C,E,\rho,D)$ if and only if $\phi:C_1\to C_2$ is an $E$-torsor map preserving $\msO_C(D)$. Thus, $\psi\in\im(f)$ if and only if there exists such a $\phi$ if and only if $(C_1,D)$ and $(C_2,D)$ represent the same element of $\hom^1(K,E[n])$. By construction, $(C_2,D)$ represents the element $\pull\psi[(C_1,D)]$, so we get the claimed description of $\im(f)$.
\end{proof}
\begin{rem}\label{rem:no-extra-aut}
    When $n=2$, $\hom^1(K,E[2])$ is $2$-torsion, so $\{\pm1\}\subset\Stab_{\Aut(E)}{(C,E,\rho,D)}$ always. Thus, \cref{lem:sel-aut} implies that we always have
    \[\{\pm1\}\subset\im\p{\Aut_{\CSel_2}(C,E,\rho,D)\to\Aut(E)}\qedhere.\]
\end{rem}

With $\CSel_n$ introduced, recall the groupoid $\meH(B)$ of hyper-Weierstrass curves over $B$ (\cref{defn:hW}). We are going to show that for every 2-Selmer element $(C,E,\rho,D)\in\CSel_2$, there is some ``nice" hW curve $(H/B,D_H)\in\meH(B)$ whose generic fiber is $(C,D)$. This will allow us to relate counting 2-Selmer elements to the problem of counting ``nice'' hW curves. We begin by making explicit what we mean by ``nice''.
\begin{defn}\label{defn:minimal}
    Let $(H\xto fB,D)\in\meH(B)$ be an hW curve. We say that it is \define{minimal} if it's normal, its generic fiber is smooth, and it has at worst \important{rational singularities}, i.e. for some (equivalently, any) proper birational map $p:\meC\to H$ with $\meC$ regular, the sheaf $R^1\push p\msO_\meC$ vanishes, see \cite{ArithGeo-Lipman}.
\end{defn}
\begin{rem}\label{rem:min=rat-sing}
    A Weierstrass model of an elliptic curve is minimal (in the usual sense) if and only if it has at worst rational singularities \cite[Corollary 8.4]{conrad-min}.
\end{rem}
\begin{warn}
    Even in good characteristics, the question of how many minimal hW models a given elliptic curve has is a subtle one, see e.g. \cite[Theorem 4.2]{sadek:genus-1}.
\end{warn}

\begin{notn}\label{notn:Hs}
    \hfill\begin{itemize}
        \item Let $\meH_M(B)\into\meH(B)$ denote the full subgroupoid consisting of minimal hW curves.
        \item Let $\meH_{M,NT}(B)\into\meH_M(B)$ denote the full subgroupoid consisting of minimal hW curves $(H\xto\pi B,D)$ such that $D_K$ is not twice a point (on the generic fiber).

        These curves will correspond to non-trivial Selmer elements.
        \item Let $\meH_{LS}(B)\into\meH_M  (B)$ denote the full subgroupoid consisting of minimal hW curves $(H\to B,D)$ whose generic fiber $H_K$ is locally solvable.
        \item Let $\meH_{LS,NT}(B)\into\meH_{LS}(B)$ denote the full subgroupoid $\meH_{LS,NT}(B)=\meH_{LS}(B)\cap\meH_{M,NT}(B)$. 
    \end{itemize}
\end{notn}
\begin{notn}
    Given $d\in\Z$, we write $\CSel_n^{\le d},\meH^{\le d}(B),\meH^{\le d}_M(B)$, etc. to denote the corresponding full subgroupoid consisting of objects of height $\le d$. We similarly use a $^{=d}$ superscript to denote the full subgroupoid of objects of height $=d$.
\end{notn}

\begin{prop}[To be proven in \cref{sect:proof-prop}]\label{prop:sel-conn}
    There is an essentially surjective, faithful functor $F:\meH_{LS}(B)\to\CSel_2$ such that for every $\alpha\in\CSel_2$, there exists some (minimal) $\beta\in\meH_{LS}(B)$ satisfying $F(\beta)\simeq\alpha$ and $\Ht(\beta)=\Ht(\alpha)$. Furthermore, if $\alpha$ is non-trivial, then we may choose $\beta$ lying in $\meH_{LS,NT}(B)$.
\end{prop}
Accepting this proposition for now, let us explain its utility by giving an overview of the ultimate proof of \cref{thma:main}. Recall the quantity $\AS_B(d)$ defined in \cref{eqn:AS-def}, and that our goal is to produce an upper bound for $\AS_B=\limsup_{d\to\infty}\AS_B(d)$. In place of $\AS_B(d)$, we will find it more convenient to study the following \define{modified average size of $2$-Selmer}:
\begin{equation}\label{eqn:MAS-def}
    \MAS_B(d):=\frac{\#\CSel_2^{\le d}}{\#\meM_{1,1}^{\le d}(K)}
    \tand
    \MAS_B\coloneqq\limsup_{d\to\infty}\MAS_B(d)
    .
\end{equation}
In \cref{sect:MAS-AS-char-2,sect:main-results} (see \cref{cor:AS-MAS-comp-char-2}, \cref{lem:IAS-MAS-comp}, and \cref{prop:comp-IAS-AS}), we will show that
\begin{equation}\label{ineq:AS-MAS-overview}
    \AS_B\le\MAS_B
\end{equation}
By \cref{lem:sel-aut} and \cref{rem:no-extra-aut}, the difference of the two sides of \cref{ineq:AS-MAS-overview} is accounted for by elliptic curves with nontrivial 2-torsion or with extra automorphisms. Therefore, we will prove \cref{ineq:AS-MAS-overview} by analyzing the contributions of such curves. Accepting \cref{ineq:AS-MAS-overview} for now, we will be interested in bounding $\MAS_B(d)$. As \cref{prop:sel-conn} suggests, we will find it helpful to separately bound the contributions coming from trivial and non-trivial 2-Selmer elements.
\begin{notn}\label{notn:Sel2NT}
    Let $\CSel_{2,T}$ (resp. $\CSel_{2,NT}$) denote the full subgroupoid of $\CSel_2$ consisting of trivial (resp. non-trivial) objects.
\end{notn}
Note that $\#\CSel_2^{\le d}=\#\CSel_{2,T}^{\le d}+\#\CSel_{2,NT}^{\le d}$. Let us separately analyze each summand. 
\begin{itemize}
    \item We begin with $\#\CSel_{2,NT}^{\le d}$. This is the summand which makes use of \cref{prop:sel-conn}.
    \begin{cor}[of \cref{prop:sel-conn}]\label{cor:sel-conn}
        $\#\CSel_{2,NT}^{\le d}\le\#\meH_{LS,NT}^{\le d}(B)\le\#\meH_{M,NT}^{\le d}(B)$
    \end{cor}
    \begin{proof}
        The first inequality follows directly from \cref{prop:sel-conn}. The second inequality holds simply because $\meH^{\le d}_{LS,NT}(B)\into\meH^{\le d}_{M,NT}(B)$ is a full subgroupoid.
    \end{proof}
    As this corollary suggests, we will bound $\#\CSel_{2,NT}^{\le d}$ by bounding $\#\meH^{\le d}_{M,NT}$, that is, by counting hW curves. This count will be carried out in \cref{sect:selmer-groupoid-card}, culminating in \cref{cor:mod-nt-estimate}, which says that
    \begin{equation}\label{ineq:NT-Sel-bound}
        \limsup_{d\to\infty}\frac{\#\meH_{M,NT}^{\le d}(B)}{\#\meM^{\le d}_{1,1}(K)}\le2\zeta_B(2)\zeta_B(10).
    \end{equation}
    
    \item For the trivial Selmer elements, we use a separate argument. Note that (e.g. by \cref{ex:triv-sel})
    \[\#\CSel_{2,T}^{\le d}=\sum_{\substack{E/K\\\Ht(E)\le d}}\frac1{\#\Aut_{\CSel_2}(E,E,\rho_E,2O)},\]
    where, for an elliptic curve $E/K$, $\rho_E:E\by E\to E$ is the multiplication map, and $O\in E(K)$ is the identity element. By \cref{lem:sel-aut}, there is a short exact sequence
    \[0\too E[2](K)\too\Aut_{\CSel_2}(E,E,\rho_E,2O)\too\Aut(E)\too0,\]
    so $\#\Aut_{\CSel_2}(E,E,\rho_E,2O)=\#E[2](K)\cdot\#\Aut(E)$. Consequently,
    \begin{equation}\label{eqn:CSelT-card}
        \#\CSel_{2,T}^{\le d}=\sum_{\substack{E/K\\\Ht(E)\le d}}\frac1{\#E[2](K)\cdot\#\Aut(E)}.
    \end{equation}
    In \cref{sect:main-results} (see \cref{prop:count-triv-sel}), we will show that
    \begin{equation}\label{eq:triv-is-1}
        \#\CSel_{2,T}^{\le d} = \sum_{\substack{E/K\\\Ht(E)\le d}}\frac1{\#E[2](K)\cdot\#\Aut(E)} \sim \sum_{\substack{E/K\\\Ht(E)\le d}}\frac1{\#\Aut(E)} = \#\meM_{1,1}^{\le d}(K).
    \end{equation}
\end{itemize}
Once we have established \cref{ineq:AS-MAS-overview} in \cref{sect:main-results}, \cref{ineq:NT-Sel-bound} in \cref{cor:mod-nt-estimate}, and \cref{eq:triv-is-1} in \cref{sect:main-results}, \cref{thma:main} (= \cref{thm:main}) will immediately follow. 

\subsubsection{Proof of \cref{prop:sel-conn}}\label{sect:proof-prop}
We want to construct an essentially surjective, faithful functor
\[F:\meH_{LS}(B)\to\CSel_2\]
along with a choice of nice preimage for any object in $\CSel_2$. 

\begin{construct}\label{construct:hW-Sel-func}
    The desired functor $F$ is defined on objects by
    \[F(H\to B,D):=(H_K,\rPic^0_{H_K},\rho_{H_K},D_K),\]
    with the $_K$ subscript denoting the generic fiber, and $\rho_{H_K}:H_K\by\Pic^0_{H_K}\to H_K$ being the natural action (coming from identifying $H_K\iso\Pic^1_{H_K}$). That this is functorial, i.e. defined on morphisms, comes from the fact that $\rPic^0_{H_K}$ is the Albanese variety of $H_K$. Hence, any morphism $\phi:(H/B,D)\to(H'/B,D')$ in $\meH_{LS}(B)$ will induce a $\psi:\rPic^0_{H_K}\to\rPic^0_{H'_K}$ so that $(\phi,\psi):F(H/B,D)\to F(H'/B,D)$.
\end{construct}
With $F$ defined, to prove \cref{prop:sel-conn}, we still need to construct nice preimages and show faithfulness. We begin with faithfulness.

\begin{prop}\label{prop:F-faithful}
    For $(H\to B,D)\in\meH_{LS}(B)$ with $(C,E,\rho,D):=F(X/B,D)$, the induced map
    \[\push F:\Aut_{\meH(B)}(H\to B,D)\too\Aut_{\CSel_2}(C,E,\rho,D)\]
    is injective. That is, $F:\meH_{LS}(B)\to\CSel_2$ is faithful.
\end{prop}
\begin{proof}
    Fix any hW automorphism $\phi:H\iso H$ such that $\push F(\phi)=(\phi_K,\psi)$ is the identity. Because $H\to B$ is flat with reduced generic fiber $H_K=C$, \cite[Proposition 4.3.8]{liu} tells us that $H$ is reduced. Thus, $H_K$ is schematically dense in $H$; hence, $\phi_K=\id_{H_K}\implies\phi=\id_H$.
\end{proof}

\subsubsection*{\bf essential surjectivity of $F$} The proof that $F$ is essentially surjective will occupy us for the next several pages. The rough idea is to first start with a Selmer element $(C,D)$, consider the minimal proper regular model $\meC/B$ of $C$, and then to extend $D$ to a divisor $\meD$ on $\meC$. We will do this in such a way that the pair $(\meC,\meD)$ becomes a hawc. Then, using \cref{cor:proj-hW}, we can construct from this a particular hW model $H/B$ of $C$. This $H$ will be our choice of nice preimage. The bulk of the remainder of this section will be spent verifying $H$ has all the properties claimed in the statement of \cref{prop:sel-conn}.

\begin{set}
    Fix any $(C,E,\rho,D)\in\CSel_2$. Let $\pi:\meC\to B$ denote the minimal proper regular model of $C$, and let $\meD\subset\meC$ denote the scheme-theoretic closure of $D\subset C=\meC_K\subset\meC$. Note that $\meD$ is a Cartier divisor because $\meC$ is regular.
\end{set}
\begin{lemma}\label{lem:C-hawc}
    The pair $(\meC,\meD)$ is a hawc over $B$. That is, $\meC/B$ is a curve satisfying \bp a $\push\pi\msO_\meC\simeq\msO_B$, \bp b $\omega_{\meC/B}\in\pull\pi\Pic(B)$, and \bp c $\meD\subset\meC$ is an effective relative Cartier divisor of degree $2$. In fact, $\omega_{\meC/B}\simeq\pull\pi\msL$, where $\msL=\push\pi\omega_{\meC/B}\in\Pic(B)$.
\end{lemma}
\begin{proof}
    It is clear that $\meC$ is a curve over $B$.
    \begin{itemize}[align=left]
        \Item{a,b} Because $C$ has a $K_v$-point for every place $v$ of $K$, \cite[Lemma 9.1]{dejong} shows that \bp a,\bp b hold fiberwise, i.e. that
        \[\hom^0(\meC_b,\msO_{\meC_b})=\kappa(b)\tand\omega_{\meC_b}\simeq\msO_{\meC_b},\]
        for every closed $b\in B$. By \cref{lem:fibral-gs+ds}, we conclude that $\msL:=\push\pi\omega_{\meC/B}$ is a line bundle, and that
        \[\push\pi\msO_\meC=\msO_B\tand\omega_{\meC/B}\simeq\pull\pi\msL\]
        (both holding after arbitrary base change).
        \Item{c} Given the definition of $\meD$, to prove that it is an effective \important{relative} Cartier divisor of degree 2, it suffices to show that it is flat over $B$. Thus, for any scheme point $d\in\meD$, we need to show that the ring map $\msO_{B,\pi(d)}\to\msO_{\meD,d}$ is flat. Because $B$ is a Dedekind scheme, this holds if and only if $\msO_{\meD,d}$ is $\msO_{B,\pi(d)}$-torsion-free. Note that, by definition, $\msO_\meD=\msO_\meC/\ker\p{\msO_\meC\to\push i\msO_D}$, where $i:D\into\meC$ is the natural inclusion. Hence, $\msO_{\meD,d}$ is contained in the $K$-vector space $\msO_D$, and so is certainly $\msO_{B,\pi(d)}$-torsion-free (note $K=\Frac\msO_{B,\pi(d)}$).
        \qedhere
    \end{itemize}
\end{proof}
Now, let
\[H:=\rProj_B\p{\bigoplus_{n\ge0}\push\pi\msO_\meC(n\meD)}\xtoo fB,\]
and let $D_H\subset H$ be the scheme-theoretic image of $\meD$ under the natural map $p:\meC\to H$. Then, \cref{lem:C-hawc} and \cref{cor:proj-hW} together tells us that $(H,D_H)$ is an hW curve over $B$. 
\begin{rem}
    Note that $H_K=\Proj\p{\bigoplus_{n\ge0}\hom^0(C,\msO_C(nD))}=C$ because $D\subset C$ is ample.
\end{rem}
\begin{rem}\label{rem:H-iso-dense-open}
    Let $\{F_i\}_{i\in I}$ be the (finite) set of fibral components $F_i\subset\meC/B$ \important{not} meeting $\meD$, and let $U:=\meC\sm\bigcup_{i\in I}F_i\openset\meC$. Then, $\meD\subset U$ by definition, and
    \[U\xtoo pp(U)\subset H\]
    is an open immersion with dense image. Indeed, \cref{lem:contract} proves this holds on each fiber over $B$. Thus, the fibral open immersion criterion \cite[Corollaire 17.9.5]{egaiv.4} says the same is true of $p$ globally. In particular, $D_H\subset p(U)\subset X$ can alternatively be described as the pullback of $\meD\subset U$ along the isomorphism $\pinv{p\vert_U}:p(U)\iso U$. 
\end{rem}
\begin{rem}\label{rem:H-normal}
    We remark that $H$ is normal. Indeed, $H$ is Gorenstein because \cref{thm:loc-model} shows that it is locally a hypersurface in $\P(1,2,1)$. Further, \cref{rem:H-iso-dense-open} above shows that $H$ is isomorphic to $\meC$ away from a codimension 2 subset (the images of the finitely many fibral components $F_i$ \important{not} meeting $\meD$), so $H$ is regular in codimension 1. Thus, $H$ must be normal by Serre's criterion.
\end{rem}
At this point, it is clear that the $(H,D_H)$ just constructed is an hW curve whose generic fiber is $(C,D)$. To finish the proof of \cref{prop:sel-conn}, we still need to prove the following:
\begin{itemize}
    \item $\Ht(H)=\Ht(C,E,\rho,D):=\Ht(E)$. By \cref{rem:height-torsor}, it is equivalent to prove that $\Ht(H)=\Ht(C):=\Ht(\meC)$. We show this in \cref{prop:contract-family}.
    \item $(H,D_H)$ is minimal in the sense of \cref{defn:minimal}. We show this in \cref{cor:construct-min}. Given this, it follows from definitions that if $(C,E,\rho,D)$ is non-trivial, then $(H,D_H)\in\meH_{LS,NT}(B)$. 
\end{itemize}
\begin{prop}\label{prop:contract-family}
    The above constructed $(H\xto fB,D_H)$ satisfies both
    \begin{enumerate}
        \item $\push\pi\omega_{\meC/B}\simeq\push f\omega_{H/B}$; and
        \item $\push\pi\msO_\meC(\meD)\simeq\push f\msO_H(D_H)$.
    \end{enumerate}
    In particular, by \bp1 above, the height of $X$ equals the height of $\meC$.
\end{prop}
\begin{proof}
    \bp1 The Grothendieck spectral sequence $R^p\push f(R^q\push p\msO_\meC)\implies R^{p+q}\push\pi\msO_\meC$ gives us a morphism $R^1\push f(\push p\msO_\meC)\to R^1\push\pi\msO_\meC$. Dualizing, and recalling also the map $\msO_H\to\push p\msO_\meC$, below we define $\phi:\push\pi\omega_{\meC/B}\to\push f\omega_{H/B}$ as the composition
    \[\push\pi\omega_{\meC/B}\simeq\pdual{R^1\push\pi\msO_\meC}\to\pdual{R^1\push f(\push p\msO_\meC)}\to\pdual{R^1\push f\msO_H}\simeq\push f\omega_{H/B}.\]
    For each $b\in B$, one has $\omega_{\meC/B}\vert_{\meC_b}=\omega_{\meC_b}$ (and similarly for $\omega_{H/B}$) by \cite[\href{https://stacks.math.columbia.edu/tag/0E6R}{Tag 0E6R}]{stacks-project}, so one obtains a commutative diagram
    \[\commsquare{\push\pi\omega_{\meC/B}\otimes\kappa(b)}{\phi_b}{\push f\omega_{X/B}\otimes\kappa(b)}{}{}{\hom^0(\meC_b,\omega_{\meC_b})}\sim{\hom^0(X_b,\omega_{X_b})}\]
    whose bottom horizontal map is an isomorphism by \cref{prop:contract}\bp d. Furthermore, both vertical maps above are isomorphisms as well, e.g. by \cref{lem:fibral-gs+ds}. We also remark that $\push\pi\omega_{\meC/B},\push f\omega_{H/B}$ are both line bundles, e.g. by \cref{lem:fibral-gs+ds}. Hence, $\phi$ is a map of line bundles inducing isomorphisms on the fibers, and so itself an isomorphism.
    
    \bp2 The argument that $\push\pi\msO_\meC(\meD)\simeq\push f\msO_H(D_H)$ is even simpler. It follows from \cref{rem:H-iso-dense-open} that $\meD=\pull pD_H$. Hence, $p$ induces a natural map $\push f\msO_H(D_H)\to\push\pi\msO_\meC(\meD)$. Since both sides of this map are vector bundles whose formations commute with arbitrary base change (e.g. by \cref{lem:push-E}), this map is an isomorphism if and only if it is an isomorphism on fibers, and on fibers, this map is the isomorphism of \cref{prop:contract}\bp b. 
\end{proof}
\begin{cor}\label{cor:construct-min}
    The above constructed $(H\xto fB,D_H)$ is minimal.
\end{cor}
\begin{proof}
    Note that $H$ is normal by \cref{rem:H-normal} and has smooth generic fiber by construction. Hence, it suffices to show that $R^1\push p\msO_\meC$ vanishes. Because $H$ is normal, \cite[(3.3)]{ArithGeo-Lipman} provides a short exact sequence
    \[0\too\push p\omega_{\meC/B}\too\omega_{H/B}\too\sExt^2_{\msO_H}\p{R^1\push p\msO_{\meC}, \omega_{H/B}}\too0.\]
    As a consequence of \cref{prop:contract-family}\bp1, $\omega_{\meC/B}\simeq\pull\pi\msL$ and $\omega_{H/B}\simeq\pull f\msL$ for the same $\msL\in\Pic(B)$. Hence, $\push p\omega_{\meC/B}\iso\omega_{H/B}$, so $\sExt^2(R^1\push p\msO_\meC,\omega_{H/B})=0$. By \cite[(1.5)]{ArithGeo-Lipman}, this means that $R^1\push p\msO_\meC=0$.
\end{proof}
This completes the proof of \cref{prop:sel-conn}.

\subsection{A Geometric Lemma for Minimal hW Curves}\label{sect:geom-lemma}
For later use in \cref{sect:count-hW}, we now prove a technical lemma (\cref{cor:deg(E)-constraint}) involving minimal hW curves. The reader is encouraged to skip this section for now, only returning to it when its results are needed.

\begin{set}\label{set:min-hW}
    We continue to work within the context of \cref{set:main}. Let $(H\xto fB,D)$ be a minimal hW curve, and let $p:\meC\to H$ be a minimal resolution of singularities (so $\meC$ regular, and $\meC_K\iso H_K$). Let $\meD:=\pull pD$, and let $\pi=f\circ p$. Thus, we have a commutative triangle
    \[\mapover\meC pH\pi f{B.}\]
\end{set}
\begin{rem}\label{rem:min-desing-min}
    We remark that $\meC$ is the minimal proper regular model of its generic fiber $\meC_K=H_K$. Indeed, \cite[Proposition (5.1)]{ArithGeo-Lipman} shows that $\omega_{\meC/B}\simeq\pull p\omega_{H/B}$, so $\omega_{\meC/B}\simeq\pull\pi(\push p\omega_{H/B})$ is fibral and hence minimality of $\meC$ follows from \cite[Corollary 3.26]{liu}.
\end{rem}
\begin{lemma}\label{lem:min-push-match}
    $\push p\omega_{\meC/B}\simeq\omega_{H/B}$ and $\push p\msO_\meC(\meD)\simeq\msO_H(D)$. Consequently, $\push\pi\omega_{\meC/B}\simeq\push f\omega_{H/B}$ and $\push\pi\msO_\meC(\meD)\simeq\push f\msO_H(D)$.
\end{lemma}
\begin{proof}
    Because $H$ has rational singularities, the first of these follows from the short exact sequence $0\to\push p\omega_{\meC/B}\to\omega_{H/B}\to\sExt^2(R^1\push p\msO_\meC,\omega_{H/B})\to0$, \cite[(3.3)]{ArithGeo-Lipman}. For the second, we use the projection formula to compute
    \[\push p\msO_\meC(\meD)\simeq\push p\p{\msO_\meC\otimes\pull p\msO_H(D)}\simeq\push p\msO_\meC\otimes\msO_H(D)\simeq\msO_H(D).\qedhere\]
\end{proof}

In the below lemmas, we define the \define{degree} of a vector bundle $\msV$ on a possibly singular curve $Y/k$ to be $\deg\msV:=\deg(\pull\nu\msV)$, where $\nu:\wt Y\to Y$ is its normalization. Note that Riemann-Roch tells us that $\deg\msV=\chi(\msV)-\rank(\msV)\chi(\msO_Y)$. 
\begin{lemma}\label{lem:curve-push-deg}
    Let $Y/k$ be an irreducible curve equipped with a finite map $f\colon Y\to B$. Choose any $\msM\in\Pic(Y)$. Then, 
    \[\deg\p{\push f\msM}=\deg\p{\push f\msO_Y}+\deg\msM.\]
\end{lemma}
\begin{proof}
    Riemann-Roch on $B$, \cite[Exercise III.4.1]{hart}, and then Riemann-Roch on $Y$ yields
    \[\deg(\push f\msM)-\deg(\push f\msO_Y)=\chi(\push f\msM)-\chi(\push f\msO_Y)=\chi(\msM)-\chi(\msO_Y)=\deg\msM.\qedhere\]
\end{proof}
\begin{lemma}\label{lem:curve-dual-deg}
    Let $Y/k$ be an irreducible curve equipped with a finite map $f\colon Y\to B$. Let $\nu:\wt Y\to Y$ be its normalization. Then,
    \[\deg\omega_{Y/B}\ge\deg\omega_{\wt Y/B}.\]
\end{lemma}
\begin{proof}
    \cite[Remark (26)(vii)]{kleiman-duality} applied to the composition $\wt Y\to Y\to B$, and then to the composition $\wt Y\to Y\to\spec K$ tells us that
    \[\omega_{\wt Y/B}\otimes\pinv{\pull\nu\omega_{Y/B}}\simeq\omega_{\wt Y/Y}\simeq\omega_{\wt Y/k}\otimes\pinv{\pull\nu\omega_{Y/k}}.\]
    Taking degrees, we see that
    \[\deg\omega_{\wt Y/B}-\deg\omega_{Y/B}=\deg\omega_{\wt Y/Y}=\deg\omega_{\wt Y/k}-\deg\omega_{Y/k}=(2p_a(\wt Y)-2)-(2p_a(Y)-2)=2(p_a(\wt Y)-p_a(Y)),\]
    with the penultimate equality holding by Riemann-Roch for possibly singular curves (e.g \cite[Exercise IV.1.9]{hart}). The claim now holds as $p_a(\wt Y)\ge p_a(Y)$.
\end{proof}
\begin{prop}\label{prop:reg-deg(E)}
    Use notation as in \cref{set:min-hW}. Suppose that $\meD\subset\meC$ is the closure of its generic fiber $D_K$. Let $\msE:=\push\pi\msO_\meC(\meD)$, $\msL:=\push\pi\omega_{\meC/B}$, and let $d:=\deg\msL$. Then, one of the following holds:
    \begin{enumerate}
        \item $D_K=2P$ for some $P\in\meC(K)$. In this case ,$\det(\msE)\simeq\inv[2]\msL$, so $\deg\msE=-2d$.
        \item $\deg\msE\ge-d$.
        \item $\Char K=2$, $D_K$ is a closed point with residue field inseparable over $K$, and $\deg\msE\ge1-(d+g)$.
    \end{enumerate}
\end{prop}
\begin{proof}
    Keep in mind that $\msL\simeq\pdual{R^1\push\pi\msO_\meC}$ by duality. To prove that one of \bp1,\bp2,\bp3 above holds, we break into cases depending on the form of the divisor $D_K\subset\meC_K=:C$.
    \begin{itemize}[align=left]
        \item Case 1: $D_K=P+Q$ for some (possibly equal) $P,Q\in C(K)$.
        
        Extend $P,Q$ to sections $\meP,\meQ\in\meC(B)$, respectively (so $\meD=\meP+\meQ$). The exact sequence $0\to\msO_\meC\to\msO_\meC(\meP)\to\meO_\meP(\meP)\to0$ pushes forward to
        \[0\too\msO_B\too\push\pi\msO_\meC(\meP)\too\push\pi\meO_\meP(\meP)\too R^1\push\pi\msO_\meC\too R^1\push\pi\msO_\meC(\meP)=0,\]
        with last equality holding by \cref{lem:C-hawc,lem:push-E}. An easy cohomology and base change argument shows that every object above is a line bundle, so we quickly conclude that
        \begin{equation}\label{isos:easy-pushes}
            \push\pi\msO_\meP(\meP)\iso R^1\push\pi\msO_\meC\cong\inv\msL,\tandso \msO_B\iso\push\pi\msO_\meC(\meP).
        \end{equation}
        By symmetry, the same is true with $\meQ$ in place of $\meP$. Note that $\pi$ restricts to an isomorphism $\meP\to B$, so $\push\pi$ preserves tensor products of sheaves supported on $\meP$. With this in mind, the exact sequence $0\to\msO_\meC(\meQ)\to\meO_\meC(\meD)\to\meO_\meP(\meP+\meQ)\to0$ pushes forward to
        \[0\too\msO_B\too\msE\too\push\pi\msO_\meP(\meP)\otimes\push\pi\msO_\meP(\meQ)\too0.\]
        If $\meQ=\meP$ (i.e. if $Q=P$, i.e. if $D_K=2P$), then $\det\msE\simeq\inv[2]\msL$ (recall \cref{isos:easy-pushes}), which is \bp a of the proposition. If $\meQ\neq\meP$, then $n:=\deg\push\pi\meO_\meP(\meQ)=\deg\meO_\meP(\meQ)=\meP\cdot\meQ\ge0$ (since it is the intersection number of distinct irreducible curves), so $\deg\msE=n-d\ge-d$, which is \bp b of the proposition.
        
        \item Case 2: $D_K$ is a closed point with residue field $L$ quadratic over $K$.
        
        The sequence $0\to\msO_\meC\to\msO_\meC(\meD)\to\msO_\meD(\meD)\to0$ pushes forward to
        \[0\too\msO_B\too\msE\too\push\pi\msO_\meD(\meD)\too\inv\msL\too0.\]
        Note $\push\pi\msO_\meD(\meD)$ is a vector bundle, as its the pushforward of a line bundle along a finite map of curves, so we can compute $\det\msE$ by taking determinants above: $\det\msE\simeq\det(\push\pi\msO_\meD(\meD))\otimes\msL$. \cref{lem:curve-push-deg} then gives
        \begin{equation}\label{eqn:deg-E}
            \deg\msE=\deg\push\pi\msO_\meD+\deg\msO_\meD(\meD)+\deg\msL = \deg\push\pi\msO_\meD + \meD\cdot\meD + d.
        \end{equation}
        We are now interested in computing $\det\push\pi\msO_\meD$. For this, we turn to the exact sequence $0\to\msO_\meC(-\meD)\to\meO_\meC\to\msO_\meD\to0$, which pushes forward to
        \begin{equation}\label{les:push-OD}
            0\too\push\pi\msO_\meC(-\meD)\too\msO_B\too\push\pi\msO_\meD\too R^1\push\pi\msO_\meC(-\meD)\too\inv\msL\too0.
        \end{equation}
        Note that, on each fiber, $\hom^0(\meC_b,\msO_{\meC_b}(-\meD_b))$ is the subset of $\hom^0(\meC_b,\msO_{\meC_b})$ vanishing along $\meD_b$, but $\hom^0(\meC_b,\msO_{\meC_b})=\kappa(b)$ by \cref{lem:C-hawc}, so $\hom^0(\meC_b,\msO_{\meC_b}(-\meD_b))=0$. Hence, \cref{thm:coh-bc} implies that $\push\pi\msO_\meC(-\meD)=0$. By \cref{lem:C-hawc}, $\omega_{\meC/B}\simeq\pull\pi\msL$, duality and the projection formula tell us that
        \[R^1\push\pi\msO_\meC(-\meD)\simeq\sq{\push\pi\p{\msO_\meC(\meD)\otimes\omega_{\meC/B}}}^\vee\simeq\dual\msE\otimes\inv\msL.\]
        Hence, \cref{les:push-OD} becomes $0\to\msO_B\to\push\pi\msO_\meD\to\dual\msE\otimes\inv\msL\to\inv\msL\to0$. Taking determinants (and using that $\rank\msE=2$), we have
        \begin{equation}\label{eqn:det-push-OD}
            \det\push\pi\msO_\meD\simeq\det(\msE)^{-1}\otimes\inv\msL.
        \end{equation}
        Combining \cref{eqn:deg-E} and \cref{eqn:det-push-OD},
        \[\deg\msE=\frac12\meD\cdot\meD=\frac12\deg\msO_\meD(\meD).\]
        Thus, it suffices to show that $\deg\msO_\meD(\meD)$ is either $\ge-2d$ or $\ge2-2(g+d)$. Recalling that $\omega_{\meC/B}\simeq\pull\pi\msL$, we apply adjunction \cite[Corollary (19)]{kleiman-duality} to $\meD\into\meC$, which tells us that
        \[\omega_{\meD/B}\simeq\omega_{\meC/B}(\meD)\vert_\meD\simeq\p{\pull\pi\msL}\vert_\meD\otimes\msO_\meD(\meD).\]
        Taking degrees, we see that
        \begin{equation}\label{eqn:deg-dual-DB}
            \deg\omega_{\meD/B}=2\deg\msL+\meD\cdot\meD=2d+\meD\cdot\meD.
        \end{equation}
        Now, let $\wt\meD$ be the normalization of $\meD$. If $\meD\to B$ is generically separable, then $\deg\omega_{\wt\meD/B}\ge0$ because it is the degree of the ramification divisor of $\wt\meD\to B$ (e.g. by \cite[Proposition IV.2.3]{hart}), so \cref{lem:curve-dual-deg} and \cref{eqn:deg-dual-DB} tell us that
        \[2d+\meD\cdot\meD=\deg\omega_{\meD/B}\ge\deg\omega_{\wt\meD/B}\ge0,\tso\meD\cdot\meD\ge-2d,\]
        which is \bp b of the proposition. Finally, if $\meD\to B$ is generically inseparable, then $\wt\meD\xto fB$ is Frobenius, so $g(\wt\meD)=g(B)$, which means (by \cite[Remark (26)(vii)]{kleiman-duality}) that
        \[\deg\omega_{\wt\meD/B}=\deg\omega_{\wt D/k}-\deg\pull f\omega_{B/k}=\deg\omega_{\wt D/k}-2\deg\omega_{B/k}=2-2g.\]
        Hence, \cref{lem:curve-dual-deg} and \cref{eqn:deg-dual-DB} tell us that
        \[2d+\meD\cdot\meD=\deg{\omega_{\meD/B}}\ge\deg\omega_{\wt\meD/B}=2-2g\implies\meD\cdot\meD\ge2-2(g+d),\]
        which is \bp c of the proposition.
        \qedhere
    \end{itemize}
\end{proof}
\begin{cor}\label{cor:deg(E)-constraint}
    Use notation as in \cref{set:min-hW}. Let $\msE:=\push f\msO_H(D)$, let $\msL:=\push\pi\omega_{\meC/B}$, and let $d:=\deg\msL$. Assume that $D_K$ is not twice a point. Then, $\deg\msE\ge-(d+g)$. Furthermore, if $\Char K\neq2$, then $\deg\msE\ge-d$.
\end{cor}
\begin{proof}
    By \cref{lem:min-push-match}, $\msL\simeq\push\pi\omega_{\meC/B}$ and $\msE\simeq\push\pi\msO_\meC(\meD)$. Write $\meD=\meD'+\meV$, where $\meD'$ is the closure of $D_K$ in $\meC$ and $\meV$ is an effective vertical divisor. The exact sequence $0\to\msO_\meC(\meD')\to\msO_\meC(\meD)\to\msO_\meV(\meD)\to0$ pushes forward to
    \begin{equation}\label{eqn:ses-horz-all}
        0\too\push\pi\msO_\meC(\meD')\too\msE\too\push\pi\msO_\meV(\meD)\too0=R^1\push\pi\msO_\meC(\meD'),
    \end{equation}
    where the last equality holds by \cref{lem:push-E} (whose hypotheses are satisfied by combining \cref{rem:min-desing-min} and \cref{lem:C-hawc}).
    Furthermore, $\push\pi\msO_\meC(\meD')$ is a rank 2 vector bundle by \cref{lem:push-E} while $\push\pi\msO_\meV(\meD)$ is a skyscraper sheaf supported on the (finite) image of $\meV$ in $B$. Thus, taking Euler characteristics in \cref{eqn:ses-horz-all} and applying Riemann-Roch shows that
    \[\deg\msE=\deg\push\pi\msO_\meC(\meD')+h^0(\push\pi\msO_\meV(\meD))\ge\deg\push\pi\msO_\meC(\meD').\]
    The claim follows from applying \cref{prop:reg-deg(E)} to $\meD'$, recalling that $(\meD')_K=D_K$ is not twice a point.
\end{proof}

\section{\bf An Upper Bound on the Cardinality of the $2$-Selmer Groupoid}\label{sect:selmer-groupoid-card}
Recall, in the context of \cref{set:main}, the function
\[\MAS_B(d):=\frac{\#\CSel_2^{\le d}}{\#\meM^{\le d}_{1,1}(K)}\]
introduced in \cref{eqn:MAS-def}. The main result of this section (\cref{thm:mod-estimate}) is that
\[\MAS_B=\limsup_{d\to\infty}\MAS_B(d)\le1+2\zeta_B(2)\zeta_B(10).\]
In the sections after this one, we will show that $\AS_B\le\MAS_B$, and so deduce \cref{thma:main}.

As in \cref{sect:EC-count}, we begin by studying ``the form of the equation needed to cut out an hW curve over $B$.''
\subsection{Global Equations for hW Curves}
\begin{set}
    Fix an arbitrary base scheme $B$.
\end{set}
Recall (\cref{thm:loc-model}) that an hW curve $(H\xto\pi B,D)\in\meH(B)$ can locally (on $B$) be embedded into $\P(1,2,1)$. In this section, we globalize this result by embedding $H$ into a $\P(1,2,1)$-bundle $\P$ over $B$ and then studying the line bundle $\msO_\P(H)$ (see \cref{prop:hW-norm-bund,prop:hW-norm-filt}). The proof of \cref{thm:loc-model} suggests that $H$ should embed into a $\P(1,2,1)$-bundle whose homogeneous coordinate ring is generated, as a graded $\msO_B$-algebra, by $\push\pi\msO_H(D)$ (in degree 1) and $\push\pi\msO_H(2D)$ (in degree 2). Inspired by this, we make the following definition.

\begin{defn}\label{defn:121-datum}
    Let $\mbf D=(\msE_1,\msE_2,\mu)$ be a tuple consisting of
    \begin{itemize}
        \item a rank 2 vector bundle $\msE_1$ on $B$,
        \item a rank 4 vector bundle $\msE_2$ on $B$, and
        \item a monomorphism $\mu:\Sym^2(\msE_1)\into\msE_2$ whose cokernel is a line bundle.
    \end{itemize}
    We call such a tuple a \define{(1,2,1)-datum} (over $B$) as it will allow us to define a $\P(1,2,1)$-bundle over $B$ (see \cref{lem:D-bund}). We say $\mbfD$ is \define{isomorphic to another (1,2,1)-datum} $(\msV_1,\msV_2,\nu)$ is there exists a line bundle $\msM\in\Pic(B)$ and isomorphisms $\phi:\msE_1\otimes\msM\iso\msV_1$ and $\psi:\msE_2\otimes\msM^2\to\msV_2$ such that
    \[\commsquare{\Sym^2(\msE_1\otimes\msM)}{\mu_\msM}{\msE_2\otimes\msM^2}{\Sym^2(\phi)}\psi{\Sym^2(\msV_1)}\nu{\msV_2}\]
    commutes, where $\mu_\msM$ is the natural composition
    $\Sym^2(\msE_1\otimes\msM)\simeq\Sym^2(\msE_1)\otimes\msM^2\xto{\mu\otimes\id}\msE_2\otimes\msM^2$.
\end{defn}
\begin{construct}\label{construct:msB}
    Let $\mbf D=(\msE_1,\msE_2,\mu)$ be a (1,2,1)-datum over $B$. Form the sheaf of graded $\msO_B$-algebras $\msT(\msE_1,\msE_2):=\Sym(\msE_1\oplus\msE_2)$ graded by declaring $\msE_1$,$\msE_2$ to be in degrees 1,2, respectively, i.e.
    \[\msT(\msE_1,\msE_2)_n=\bigoplus_{a+2b=n}\Sym^a(\msE_1)\otimes\Sym^b(\msE_2)\]
    for any $n\ge0$. Let $\msI(\msE_1,\msE_2,\mu)\subset\msT(\msE_1,\msE_2)$ be the (graded) ideal sheaf generated by sections of the form $\alpha\beta-\mu(\alpha\beta)$ with $\alpha,\beta$ both local sections of $\msE_1$. Finally, set
    \[\msB(\mbfD):=\msB(\msE_1,\msE_2,\mu):=\frac{\msT(\msE_1,\msE_2)}{\msI(\msE_1,\msE_2,\mu)}\tand\P(\mbfD):=\rProj_B\msB(\mbfD).\qedhere\]
\end{construct}
\begin{ex}
    Say $(H\xto\pi B,D)$ is an hW curve. Then, by \cref{lem:push-E} and \cref{prop:model-exact-seqs}, the triple
    \[\mbfD(H/B,D):=\p{\push\pi\msO_H(D),\push\pi\msO_H(2D),\mu},\]
    where $\mu:\Sym^2(\push\pi\msO_H(D))\to\push\pi\msO_H(2D)$ is the natural multiplication map, is a (1,2,1)-datum, called the curve's \define{associated (1,2,1)-datum}. In this case, we write $\P(H/B,D):=\P(\mbfD(H/B,D))$,
    and we similarly define $\msT(H/B,D),\msI(H/B,D)$, and $\msB(H/B,D)$.
\end{ex}
\begin{defn}\label{defn:(121)-Hodge}
    Inspired by the above example, along with \cref{prop:model-exact-seqs}, given any (1,2,1)-datum $\mbfD=(\msE_1,\msE_2,\mu)$, we define its \define{Hodge bundle} to be $\msL:=\det(\msE_1)\otimes\coker(\mu)^{-1}$.
\end{defn}
\begin{lemma}\label{lem:D-bund}
    Let $\mbfD:=(\msE_1,\msE_2,\mu)$ be a (1,2,1)-datum. Then,
    \[\P(\mbfD)\xtoo pB\]
    is a Zariski-locally trivial $\P(1,2,1)$-bundle over $B$.
\end{lemma}
\begin{proof}
    We may assume without loss of generality that $\msE_1\simeq\msO_B^{\oplus2}$, $\msE_2\simeq\msO_B^{\oplus4}$, and $\coker(\mu)\simeq\msO_B$ since these all hold Zariski locally on $B$. Let $X,Z$ be a global basis for $\msE_1$, and let $Y\in\Gamma(B,\msE_2)$ restrict to a global basis for $\coker(\mu)$. Then,
    \[\msT(\msE_1,\msE_2)\simeq\msO_B\sq{X,Y,Z,\mu(X^2),\mu(XZ),\mu(Z^2)}\]
    is a polynomial algebra with $X,Z$ in degree 1, and $Y,\mu(X^2),\mu(XZ),\mu(Z^2)$ all in degree 2. Furthermore, the ideal $\msI(\mbfD)$ is generated by
    \[X^2-\mu(X^2),XZ-\mu(XZ),Z^2-\mu(Z^2),\]
    so $\msB(\mbfD)\simeq\msO_B[X,Y,Z]$ and $\P(\mbfD)\simeq\P(1,2,1)_B$.
\end{proof}
\begin{rem}\label{rem:B-rank}
    Let $\mbfD$ be a (1,2,1)-datum. As a consequence of (the proof of) \cref{lem:D-bund}, we see that the rank of $\msB(\mbfD)_n$ is equal to the number of (monic) degree $n$ monomials in $\Z[X,Y,Z]$ where $X,Z$ have degree $1$ and $Y$ has degree $2$. We will see below (\cref{lem:push-P(121)}) that $\msB(\mbfD)_n\simeq\push p\msO_{\P(\mbfD)}(n)$, so this also computes the rank of $\push p\msO_{\P(\mbfD)}(n)$.
\end{rem}
\begin{lemma}\label{lem:push-P(121)}
    Let $\mbfD=(\msE_1,\msE_2,\mu)$ be a (1,2,1)-datum over $B$, and consider $\P:=\P(\mbfD)\xto pB$. For any $n\ge0$, the natural map
    \[\msB(\mbfD)_n\too\push p\msO_\P(n)\]
    is an isomorphism.
\end{lemma}
\begin{proof}
    We will apply cohomology and base change, \cref{thm:coh-bc}. By \cref{lem:H^1(O(n))=0}, $\hom^1(\P_b,\msO_{\P_b}(n))=0$ for all $b\in B$, so \cref{thm:coh-bc}\bp{0,2} applied to $\msF=\msO_\P(n)$ with $i=1$ shows that $R^1\push p\msO_\P(n)=0$ and that the comparison map
    \[\phi^0_b:\push p\msO_\P(n)\otimes\kappa(b)\too\hom^0(\P_b,\msO_{\P^b}(n))\]
    is surjective for all $b\in B$. Thus, a second application of \cref{thm:coh-bc}\bp{0,2}, now to $\msF=\msO_\P(n)$ with $i=0$, shows that $\push p\msO_\P(n)$ is a locally free sheaf on $B$. Hence, one can check that the natural map $\msB(\mbfD)_n\to\push p\msO_\P(n)$ is an isomorphism by checking this on fibers, where it becomes the classical fact that $k[X,Y,Z]_n\iso\hom^0(\P(1,2,1),\msO(n))$, see e.g. \cite[Theorem 1.4.1(i) and Notations 1.1]{wpv}.
\end{proof}
\begin{lemma}\label{lem:glob-embed}
    Let $(H\xto\pi B,D)\in\meH(B)$ be a hyper-Weierstrass curve. Then, there is a natural embedding
    \[H\into\P(H/B,D)=:\P,\]
    for which $\msO_H(n):=\msO_\P(n)\vert_H\simeq\msO_H(nD)$ for all $n\ge0$.
\end{lemma}
\begin{proof}
    Since $D\subset H$ is relatively ample, $H\simeq\rProj_B\bigoplus_{n\ge0}\push\pi\msO_H(nD)$, and the claimed embedding comes from the natural morphism
    \[\msB(H/B,D)=\frac{\ast\Sym\p{\push\pi\msO_H(D)\oplus\push\pi\msO_H(2D)}}{\msI(H/B,D)}\too\bigoplus_{n\ge0}\push\pi\msO_H(nD)\]
    (induced by the multiplication maps $\push\pi\msO_H(D)^{\oplus a}\otimes\push\pi\msO_H(2D)^{\oplus b}\too\push\pi\msO_H((a+2b)D)$). This morphism is surjective (and so induces a closed embedding upon taking $\rProj_B$) because this was verified locally in the proof of \cref{thm:loc-model}.
\end{proof}
\cref{lem:glob-embed} provides us with an embedding of an hW curve $H$ into some $\P(1,2,1)$-bundle $\P$. We now want to understand ``the shape of the equation cutting out $H$,'' i.e. to understand the line bundle $\msO_\P(H)$ supporting a section cutting out $H$, as well as its pushforward to $B$.

\begin{lemma}\label{lem:filt-key}
    Let $\mbfD=(\msE_1,\msE_2,\mu)$ be a (1,2,1)-datum, and let $\msY:=\coker\p{\mu:\Sym^2(\msE_1)\into\msE_2}$. Then, there is a short exact sequence
    \begin{equation}\label{ses:filt-key}
        0\too\Sym^4(\msE_1)\xtoo\nu\msB(\mbfD)_4\too\msE_2\otimes\msY\too0
        .
    \end{equation}
    Above, $\nu$ is the composition $\Sym^4(\msE_1)\into\msT(\msE_1,\msE_2)_4\onto\msB(\mbfD)_4$.
\end{lemma}
\begin{proof}
    We construct \cref{ses:filt-key} locally, and then glue by observing that the locally constructed maps are independent of any choices. That being said, let $U\openset B$ be small enough that $\msE_1\vert_U\cong\msO_U^{\oplus2}$ and $\msE_2\vert_U\cong\msO_U^{\oplus 4}$ (so then also $\msY\vert_U\cong\msO_U$). Let $X,Z\in\Gamma(U,\msE_1)$ be a basis for $\msE_1\vert_U$, and choose $Y\in\Gamma(U,\msE_2)$ so that $\mu(X^2),\mu(XZ),\mu(Z^2),Y$ form a basis for $\msE_2\vert_U$. Let $\bar Y\in\Gamma(U,\msY)$ be the image of $Y$. Then, it is not difficult to see that the images of
    \[\begin{matrix}
        X^4 & X^3Z & X^2Z^2 & XZ^3 & Z^4 &
        X^2\otimes Y & XZ\otimes Y & Z^2\otimes Y &
        Y\otimes Y
    \end{matrix}\]
    under the quotient map $\msT(\msE_1,\msE_2)_4\onto\msB(\mbfD)_4$ form a basis over $U$. Define a map $\msB(\mbfD)_4\vert_U\to\msE_2\vert_U\otimes\msY\vert_U$ by sending
    \[\begin{matrix}
        X^2\otimes Y&\mapstoo&\mu(X^2)\otimes\bar Y, && XZ\otimes Y &\mapstoo& \mu(XZ)\otimes\bar Y,\\
        Z^2\otimes Y&\mapstoo&\mu(Z^2)\otimes\bar Y, && Y\otimes Y &\mapstoo& \bar Y\otimes\bar Y,
    \end{matrix}\]
    and sending all other basis elements to $0$. By construction, the kernel of this map is (isomorphic to) $\Sym^4(\msE_1)\vert_U$, i.e. we have an exact sequence
    \[0\too\Sym^4(\msE_1)\vert_U\too\msB(\mbfD)_4\vert_U\too\msE_2\vert_U\otimes\msY\vert_U\too0\]
    over $U$. Finally, one can check that the above maps are independent of the choice of $Y\in\Gamma(U,\msE_2)$ making $\mu(X^2),\mu(XZ),\mu(Z^2),Y$ a basis for $\msE_2\vert_U$ and are independent of the choice of basis $X,Z\in\Gamma(U,\msE_1)$ for $\msE_1\vert_U$. Therefore, the above short exact sequence globalizes to give the claimed sequence \cref{ses:filt-key}.
\end{proof}

\begin{prop}\label{prop:hW-norm-bund}
    Let $(H\xto\pi B,D)\in\meH(B)$ be an hW curve, and consider the natural embedding $H\into\P(H/B,D)=:\P$, constructed in \cref{lem:glob-embed}. Then, $H\into\P$ is a Cartier divisor, and so is the zero scheme of some global section of the line bundle $\msO_\P(H)$. Furthermore, we compute this line bundle to be
    \[\msO_\P(H)\simeq\msO_\P(4)\otimes\pull p(\msD^{-2}\otimes\msL^2)=\pull p\p{\inv[2]\msD\otimes\msL^2}(4),\]
    where $\msD:=\det(\push\pi\msO_H(D))$, $\msL:=\push\pi\omega_{H/B}$, and $p:\P\to B$ is the structure morphism.
    That is, we can view $H\into\P$ as being cut out by some global section of
    \[\push p\msO_\P(H)\simeq\msB(H/B,D)_4\otimes\inv[2]\msD\otimes\msL^2.\]
\end{prop}
\begin{proof}
    Once we know $\msO_\P(H)\simeq\pull p\p{\inv[2]\msD\otimes\msL^2}(4)$, the claimed computation of $\push p\msO_\P(H)$ follows from the projection formula and \cref{lem:push-P(121)}.
    
    We will find it more natural to instead directly compute the dual $\msO_\P(-H)$. First note that $H\into\P$ is Cartier by \cref{thm:loc-model}, which shows that it is locally cut out by a single equation. That same theorem also shows that the fibers of $H\into\P$ (over $B$) are cut out by weighted degree 4 equations, so the line bundle $\msO_\P(-H)(4)$ on $\P$ is trivial on each fiber. Thus, e.g. by \cite[Proposition 25.1.11]{ravi-notes}, $\msO_\P(-H)(4)\simeq\pull p\push p\msO_\P(-H)(4)$. Hence, it will suffice to compute that
    \[\push p\msO_\P(-H)(4)\simeq\msD^2\otimes\msL^{-2}.\]
    With this in mind, consider the exact sequence
    \[0\too\msO_\P(-H)(4)\too\msO_\P(4)\too\msO_H(4)\too0,\]
    and push forward along $p$. We know that $\msO_H(4)\simeq\msO_H(4D)$ by \cref{lem:glob-embed}, that $\push p\msO_\P(4)\simeq\msB(H/B,D)_4$ by \cref{lem:push-P(121)}, and that $R^1\push p\msO_\P(-H)(4)=0$ by \cref{thm:coh-bc} combined with \cref{lem:H^1(O(n))=0}. Hence, we obtain
    \begin{equation}\label{ses:O(-H)(4) comp}
        0\too\push p\msO_\P(-H)(4)\too\msB(H/B,D)_4\too\push\pi\msO_H(4D)\too0.
    \end{equation}
    Because $\rank\msB(H/B,D)_4=9$ (by \cref{rem:B-rank}) and $\rank\push\pi\msO_H(4D)=8$ (by \cref{lem:push-E}), the kernel $\push p\msO_\P(-H)(4)$ above must be a line bundle, and so it can be computed by taking determinants. \cref{cor:det(En)} tells us that
    \[\det(\push\pi\msO_H(4D))\simeq\msD^{16}\otimes\inv[3]\msL\tand\det(\push\pi\msO_H(2D))\simeq\msD^4\otimes\inv\msL.\]
    Taking determinants in the exact sequence \cref{ses:filt-key} with $\msE_1=\push\pi\msO_H(D)$ and $\msE_2=\push\pi\msO_H(2D)$ (and note that $\msY\simeq\inv\msL\otimes\msD$ by \cref{prop:model-exact-seqs}), one computes that $\det\msB(H/B,D)_4\simeq\msD^{18}\otimes\inv[5]\msL$. Finally, taking determinants in \cref{ses:O(-H)(4) comp} shows that $\push p\msO_\P(-H)(4)\simeq\msD^2\otimes\inv[2]\msL$, proving the claim.
\end{proof}

The last thing we want to take care of here is improving our understanding of the rank 9 vector bundle
\[\push p\msO_\P(H)\simeq\msB(H/B,D)_4\otimes\inv[2]\msD\otimes\msL^2\]
appearing in \cref{prop:hW-norm-bund}. We will do this by endowing it with a filtration, all of whose graded pieces are line bundles.
\begin{defn}\label{defn:normalized}
    Let $\mbfD=(\msE_1,\msE_2,\mu)$ be a (1,2,1)-datum. We say that $\mbfD$ is \define{normalized} if either
    \begin{enumerate}
        \item $\msE_1$ has Harder-Narasimhan filtration of the form
        \[0\too\msO_B\too\msE_1\too\msD\too0,\]
        necessarily with $u:=\deg\msD<0$. In this case,  we call $u$ the \define{unstable degree} of $\mbfD$.
        \item $\msE_1$ is semistable. In this case, we say $\mbfD$ has \define{unstable degree} $u=0$.
        \qedhere
    \end{enumerate}
\end{defn}
The above definition was inspired by \cite[Section 6.1]{ho-lehung-ngo}, though our ``unstable degree'' is the negation of the one appearing there. This is to allow for easier application of \cref{cor:deg(E)-constraint} when we do the actual counting.
\begin{lemma}\label{lem:hW-normalized}
    Every hW curve is isomorphic to one whose associated (1,2,1)-datum is normalized. 
\end{lemma}
\begin{proof}
    Let $(H\xto\pi B,D)$ be an hW curve, and let $\msE_1:=\push\pi\msO_H(D)$. If $\msE_1$ is semistable, then $\mbfD(H/B,D)$ is already normalized. Hence, assume that $\msE_1$ is unstable. Let $\msM\into\msE$ be a destabilizing line subbundle, so $\msO_B$ is destabilizing line subbundle of $\msF:=\msE\otimes\inv\msM$. Let 
    \[S:=\rProj_B\p{\bigoplus_{n\ge0}\p{\push\pi\msO_H(nD)\otimes\inv[n]\msM}}\xto\rho B,\]
    and let $f:S\iso H$ be the natural isomorphism \cite[\href{https://stacks.math.columbia.edu/tag/02NB}{Tag 02NB}]{stacks-project}. Let $\P:=\P(H/B,D)\xto pB$ and consider its line bundle $\pull p(\inv\msM)(1)$. By the projection formula and \cref{lem:push-P(121)}, $\push p\pull p(\inv\msM)(1)\cong\msE\otimes\inv\msM=\msF$; thus, $H^0(\P,\pull p(\inv\msM)(1))=H^0(B,\msF)$ is nonzero (recall $\msO_B\into\msF$). Embed $S\xto fH\into\P$, and let $E\subset S$ be the zero scheme of some nonzero section of $\pull p(\inv\msM)(1)$. One can use \cref{thm:loc-model-conv} to show that $(S/B,E)$ is an hW curve over $B$. By construction, this hW is isomorphic to $H$ and its associated (1,2,1)-datum is normalized.
\end{proof}
\begin{prop}\label{prop:hW-norm-filt}
    Let $\mbfD=(\msE_1,\msE_2,\mu)$ be a (1,2,1)-datum. Let $\msD:=\det(\msE_1)$, and let $\msY:=\coker(\mu)$. Then, there is a filtration $0=\msF_0\subset\msF_5\subset\msF_8\subset\msF_9=\msB(\mbfD)_4$ such that $\msF_i$ is a rank $i$ vector bundle on $B$, where
    \[\msF_5=\Sym^4(\msE_1)\tcomma\frac{\msF_8}{\msF_5}\cong\Sym^2(\msE_1)\otimes\msY,\tand\frac{\msF_9}{\msF_8}\cong\msY^2.\]
    Furthermore, if $\mbfD$ is normalized with $\msE_1$ unstable, then this filtration extends to a filtration
    \[0=\msF_0\subset\msF_1\subset\dots\subset\msF_8\subset\msF_9=\msB(\mbfD)_4\]
    by vector bundles on $B$ with graded pieces
    \[\frac{\msF_{i+1}}{\msF_i}\cong\Threecases{\msD^i}{0\le i\le4}{\msD^{i-5}\otimes\msY}{5\le i\le 7}{\msY^2}{i=8.}\]
\end{prop}
\begin{proof}[Proof Sketch]
    This follows from the existence of the exact sequences
    \[\begin{array}[b]{cccccccccl}
        0 &\too& \Sym^4(\msE_1) &\too& \msB(\mbfD)_4 &\too& \msE_2\otimes\msY &\too& 0 &\t{by \cref{lem:filt-key};}\\
        0 &\too& \Sym^2(\msE_1) &\too& \msE_2 &\too& \msY &\too& 0 &\t{by definition of (1,2,1)-datum; and}\\
        0 &\too& \msO_B &\too& \msE_1 &\too& \msD &\too& 0 &\t {if $\msE_1$ is unstable and $\mbfD$ is normalized}
        .
    \end{array}\qedhere\]
\end{proof}
\begin{rem}\label{rem:hw-glob-eqn}
    Let $(H/B,D)$ be an hW curve with Hodge bundle $\msL$. Suppose that $\push\pi\msO_H(D)$ is unstable and that the (1,2,1)-datum $\mbfD(H/B,D)$ is normalized. Let $\msD:=\det(\push\pi\msO_H(D))$. If the filtration of \cref{prop:hW-norm-filt} applied to $\mbfD(H/B,D)$ splits, then $\P:=\P(H/B,D)$ has global coordinates $X,Y,Z$ with $Y$ defined using the splitting of $0\to\Sym^2(\msE_1)\to\msE_2\to\inv\msL\otimes\msD\to0$, and $X,Z$ defined using the splitting of $0\to\msO_B\to\msE_1\to\msD\to0$ (analogously to \cref{rem:interpret-weier-eqn}). Defined appropriately, these global coordinates $X,Y,Z$ are sections
    \[X\in\hom^0\p{\P,\pull p\p{\inv\msO_B}(1)}\tcomma Y\in\hom^0\p{\P,\pull p\p{\msL\otimes\inv\msD}(2)},\tand Z\in\hom^0\p{\P,\pull p\p{\inv\msD}(1)}.\]
    With this in mind, in this case, the vector bundle $\push p\msO_\P(H)$ naturally splits as a sum of line bundles (compare \cref{prop:hW-norm-bund,prop:hW-norm-filt}), and $H\into\P$ can be described as the zero scheme of an equation
    \[\lambda Y^2 + (a_0X^2+a_1XZ+a_2Z^2)Y = c_0X^4 + c_1X^3Z + c_2X^2Z^2 + c_3XZ^3 + c_4Z^4\]
    with $\lambda\in\hom^0(B,\msO_B)$, $a_i\in\hom^0(B,\msD^{i-1}\otimes\msL)$ and $c_j\in\hom^0(B,\msD^{j-2}\otimes\msL^2)$ (note that both sides above are sections of $\pull p\p{\inv[2]\msD\otimes\msL^2}(4)$). Furthermore, by comparing with the local equations of \cref{thm:loc-model}, we see that $\lambda$ above must be nonzero, so after scaling, $H\into\P$ is cut out by an equation of the form
    \begin{equation}\label{eqn:hW-glob}
        Y^2 + (a_0X^2+a_1XZ+a_2Z^2)Y = c_0X^4 + c_1X^3Z + c_2X^2Z^2 + c_3XZ^3 + c_4Z^4,
    \end{equation}
    akin to the Weierstrass equations of \cref{defn:weier-eqn}.
\end{rem}

\subsection{Properly Embedded hW Curves}
In order to count hW curves, we will partition them according to their (1,2,1)-data. To that end, we begin by fixing such a choice of datum and studying the hW curves which embed into the corresponding $\P(1,2,1)$-bundle.
\begin{set}
    We continue to let $B$ denote an arbitrary base scheme. We also fix any choice of 
    (1,2,1)-datum $\mbfD:=(\msE_1,\msE_2,\mu)$ over $B$. 
    Finally, we write $\P:=\P(\mbfD)$ and let $p:\P\to B$ denote its structure map.
\end{set}
\begin{defn}
    We say that an hW curve $(H/B,D)$ equipped with an embedding $H\into\P$ is \define{properly embedded} if $\msO_H(1):=\msO_\P(1)\vert_H\simeq\msO_H(D)$.
\end{defn} 
\begin{lemma}\label{lem:basic-0}
    Every hW curve is isomorphic to one which properly embeds into a $\P(\mbfD)$ with $\mbfD$ normalized.
\end{lemma}
\begin{proof}
    This follows immediately from \cref{lem:glob-embed,lem:hW-normalized}.
\end{proof}
\begin{lemma}\label{lem:basic-1}
    Let $(H\xto\pi B,D)$ be an hW curve properly embedded in $\P$. Then, the natural map
    \[\push p\msO_\P(n)\too\push\pi\msO_H(n)\simeq\push\pi\msO_H(nD)\]
    is surjective for all $n\in\Z$. Furthermore, it is an isomorphism for $n=0,1,2,3$.
\end{lemma}
\begin{proof}
    Consider the exact sequence $0\to\msO_\P(-H)(n)\to\msO_\P(n)\too\msO_H(n)\to0$. By \cref{lem:H^1(O(n))=0} and the isomorphisms $\msO_\P(-H)(n)_b\simeq\msO_{\P(1,2,1)_{\kappa(b)}}(n-4)$, we have $\hom^1\p{\P(1,2,1)_b,\msO_\P(-H)(n)_b}=0$ for all $b\in B$. Thus, \cref{thm:coh-bc} tells us that $R^1\push p\msO_\P(-H)(n)=0$. Given this, our short exact sequence induces a surjection
    \[\push p\msO_\P(n)\onto\push\pi\msO_H(n)\simeq\push\pi\msO_H(nD).\]
    When $n\in\{0,1,2,3\}$, $\push p\msO_\P(n)$ and $\push\pi\msO_H(n)$ are vector bundles of rank the same rank (by \cref{rem:B-rank} and \cref{lem:push-E}), so this must be an isomorphism.
\end{proof}
\begin{cor}\label{cor:basic}
    An hW curve $(H/B,D)$ properly embeds into some $\P(\mbfD)$ for a unique, up to isomorphism, (1,2,1)-datum $\mbfD$, necessarily $\mbfD\cong\mbfD(H/B,D)$.
\end{cor}
\begin{lemma}\label{lem:basic-2}
    Let $(H\xto\pi B,D)$ and $(S\xto\rho B,E)$ be two hW curves properly embedded in $\P$. Let $f:H\iso S$ be a hyper-Weierstrass isomorphism. Then, $\pull f\msO_S(nE)\simeq\msO_H(nD)$ for all $n$.
\end{lemma}
\begin{proof}
    Since pullbacks commute with tensor products, it suffices to prove the claim when $n=1$. By definition, there exists some $\msM\in\Pic(B)$ such that $\pull f\msO_S(E)\simeq\msO_H(D)\otimes\pull\pi\msM$. Pushing forwards along $\pi$, we see that $\push\rho\msO_S(E)\simeq\push\pi\msO_H(D)\otimes\msM$. At the same time, \cref{lem:basic-1} shows that $\push\pi\msO_H(D)\simeq\push p\msO_\P(1)\simeq\push\rho\msO_S(E)$. Taken together, these two statements imply that $\msM\simeq\msO_B$, from which the claim follows.
\end{proof}
\begin{notn}
    Let $G(\mbfD)$ denote the (abstract) group of pairs $(\phi,\psi)$ of automorphisms of $\msE_1,\msE_2$ which are compatible with multiplication, i.e.
    \[G(\mbfD):=\b{(\phi,\psi)\in\GL(\msE_1)\by\GL(\msE_2)\,\left|\commsquare{\Sym^2(\msE_1)}\mu{\msE_2}{\Sym^2(\phi)}\psi{\Sym^2(\msE_1)}\mu{\msE_2}\t{ commutes.}\,\right.}\]
    We let $\PG:=\PG(\mbfD)$ denote the quotient of $G(\mbfD)$ by the scalar subgroup $\units k\into G(\mbfD),\lambda\mapsto(\lambda,\lambda^2)$, where $k:=\Gamma(B,\msO_B)$.
\end{notn}

\begin{rem}
    If you imagine you have an hW curve $H\into\P(1,2,1)$ with degree 1 coordinates $X,Z$ and degree 2 coordinate $Y$, then $\phi:\msE_1\iso\msE_1$ as above corresponds to some linear change of variables $(X,Z)\rightsquigarrow(\alpha X+\beta Z,\gamma X+\delta Z)$ and the extension to $\psi:\msE_2\iso\msE_2$ corresponds to also choosing $Y\rightsquigarrow\lambda Y+rX^2+sXZ+tZ^2$.
\end{rem}

\begin{rem}\label{rem:PG-action}
    We remark that $G(\mbfD)$ acts on $\P$. Indeed, elements of $G(\mbfD)$ induce (graded) automorphisms of the sheaf $\msB(\mbfD)$ of graded $\msO_B$-algebras (recall \cref{construct:msB}), and so induce automorphisms of $\P\simeq\rProj_B\msB(\mbfD)$. Furthermore, this action descends to one of $\PG$ on $\P$.
\end{rem}

\begin{prop}\label{prop:prop-emb-iso}
    Let $(H\xto\pi B,D)$ and $(S\xto\rho B,E)$ be two hyper-Weierstrass curves properly embedded in $\P$. Then, there is a natural isomorphism
    $$\Hom_{\meH(B)}\p{(H/B,D),(S/B,E)}\iso\b{g\in\PG:g(H)=S}$$
    (Above, `$g(H)=S$' means equality as subschemes of $\P$).
\end{prop}
\begin{proof}
    We simply construct maps in both directions.
    
    ($\to$) Let $f:H\iso S$ be an hW isomorphism. Since $H,X$ are both properly embedded in $\P$, \cref{lem:basic-2} tells us that $\pull f\msO_S(nE)\simeq\msO_H(nD)$ for any $n\in\Z$, so $f$ induces isomorphisms
    $$\alpha_n(f):\push\rho\msO_S(nE)\iso\push\pi\msO_H(nD).$$
    At the same time, \cref{lem:basic-1} tells us that the proper embeddings $H,S\into\P$ induce isomorphisms $\msE_n=\push p\msO_\P(n)\simeq\push\pi\msO_H(nD)$ and $\msE_n=\push p\msO_\P(n)\simeq\push\rho\msO_S(nE)$ when $n=1,2$. Composing these with $\alpha_n(f)$ then shows that $f$ induces automorphisms
    $$\phi(f):\msE_1\iso\msE_1\tand\psi(f):\msE_2\iso\msE_2.$$
    The map in one direction is $f\mapsto(\phi(f),\psi(f))$.
    
    ($\from$) Fix some $g\in\PG$ carrying $H\into\P$ onto $G\into\P$. Then, by assumption, $g$ give an isomorphism $f_g:H\iso G$ over $B$. To see that is an hW isomorphism, we note that the action of $\PG$ on $\P$ preserves $\msO_\P(1)$, so
    $$\pull f\msO_S(E)\simeq\pull f\msO_S(1)\simeq\msO_H(1)\simeq\msO_H(D).$$
    The assignment $g\mapsto f_g$ gives the inverse map.
\end{proof}

Let us now introduce/recall some notation. Let $0=\msF_0\subset\dots\subset\msF_9=\msB(\mbfD)_4$ denote the filtration of \cref{prop:hW-norm-filt}. Let $\msY:=\coker(\mu)$ and $\msD:=\det(\msE_1)$ as usual, and let $\msL:=\msD\otimes\inv\msY$ be the Hodge bundle of $\mbfD$. Let $\msG_i:=\msF_i\otimes\msL^2\otimes\inv[2]\msD\simeq\msF_i\otimes\inv[2]\msY$ for all $i$. Note in particular that 
\[\msG_9=\msB(\mbfD)_4\otimes\inv[2]\msD\otimes\msL^2\tand\msG_9/\msG_8\simeq\msO_B.\]
We next count the (weighted) number of hW curves properly embedded in $\P=\P(\mbfD)$, see \cref{prop:weighted-hW}.
\begin{assump}
    Assume from now on that $k:=\Gamma(B,\msO_B)$ is a field, and let $q:=\#k$. This is not strictly necessary for what comes below, but it does simplify some statements. 
\end{assump}
\begin{construct}\label{construct:F}
    Let $\meG$ denote the groupoid whose objects are global sections $s\in\hom^0(B,\msG_9)$ \important{with nonzero image} in $\hom^0(B,\msG_9/\msG_8)$, and whose $\Hom$-sets are the transporters
    \[\Hom_\meG(s_1,s_2):=\b{g\in G(\mbfD):g\cdot s_1=s_2}\]
    where $G(\mbfD)\actson\msG_9$ via its action on $\P$.

    Assume that $h^0(\msE_1)>0$, and implicitly fix a choice of nonzero section $\sigma_0\in H^0(B,\msE_1)$.
    Recall (\cref{lem:push-P(121)}) that $\msE_1\simeq\push p\msO_\P(1)$. Let $L\subset\P$ denote the hyperplane cut out by $\sigma_0\in\hom^0(\P,\msO_\P(1))$. There is a functor $F:\meG\to\meH(B)$ given on objects by $F(s):=(Z(s),D_s)$,
    with $Z(s)\into\P$ the zero scheme of $s$ and $D_s:=Z(s)\cap L$. That $F(s)$ is an hW curve over $B$ can be deduced from \cref{thm:loc-model-conv}.
\end{construct}
\begin{notn}\label{notn:sum-H-into-P}
    Whenever we write
    \[\sum_{H\into\P}(*),\]
    we mean that the sum ranges over isomorphism classes of hW curves properly embedded in $\P$.
\end{notn}
\begin{prop}\label{prop:weighted-hW}
    The weighted number of hW curves properly embedded in $\P=\P(\mbfD)$ is
    \[\sum_{H\into\P}\frac1{\#\Aut_{\meH(B)}(H)}=\#\meG\le\frac{\#\hom^0(B,\msG_8)}{\#\PG},\]
    with equality if the left hand side is nonzero. As indicated in \cref{notn:sum-H-into-P}, the sum above ranges over isomorphism classes of hW curves properly embedded in $\P$.
\end{prop}
\begin{proof}
    Suppose the left hand side is nonzero, i.e. that there exists some hW curve properly embedded in $\P$. Note that this forces $h^0(\msE_1)>0$. We first claim that the image of the functor $F$ of \cref{construct:F} consists exactly of the hW curves which can be properly embedded in $\P$. By definition, any curve in the image is properly embedded in $\P$. Conversely, if $H\into\P$ is properly embedded, then \cref{prop:hW-norm-bund} shows that $H$ is the zero set of some global section $s$ of $\msG_9$. Furthermore, the local models in \cref{thm:loc-model} show that the ``$Y^2$ coefficieint'' of the equation cutting out $H$ is always nonzero, i.e. that $s$ has nonzero image in $\hom^0(B,\msG_9/\msG_8)$. As a consequence of \cref{prop:prop-emb-iso}, given $s,s'\in\meG$, the induced map $\Hom_\meG(s,s')\to\Hom_{\meH(B)}(F(s),F(s'))$ is bijective. Thus, $\meG$ is equivalent to the groupoid of hW curves properly embedded in $\P$, proving the first equality in the claim.
    Since $\meG$ is the groupoid associated to action of $G:=G(\mbfD)$ on the set $X:=\ob\meG$, one easily computes $\#\meG=\#X/\#G$.
    Finally $\msG_9/\msG_8\cong\msO_B$ and $\#G=(q-1)\cdot\#\PG$, from which the rest of the claim follows.
\end{proof}

\begin{assump}
    From here on out, assume we are working within the context of \cref{set:main}. In particular, $k=\F_q$ is a finite field and $B/k$ is a smooth $k$-curve of genus $g=g(B)$.
\end{assump}

\begin{prop}\label{prop:PG-card}
    Let $d:=\deg\msL$, and let $\msV:=\sHom(\msY,\Sym^2(\msE_1))$. Then,
    \[\#\GL(\msE_1)q^{3d+3(1-g)}\le\#\PG(\mbfD)\le\#\GL(\msE_1)\cdot\#\hom^0(B,\msV)\]
\end{prop}
\begin{proof}
    We first compute $\#G(\mbfD)$, and then we divide by $(q-1)=\#\units k$. To do this, we upgrade $G(\mbfD)$ by considering the (Zariski) sheaf $\ul G$ on $B$ defined by
    \[\ul G(U):=\b{(\phi,\psi)\in\GL(\msE_1\vert_U)\by\GL(\msE_2\vert_U)\,\left|\,\commsquare{\Sym^2(\msE_1\vert_U)}\mu{\msE_2\vert_U}{\Sym^2(\phi)}\psi{\Sym^2(\msE_1\vert_U)}\mu{\msE_2\vert_U}\t{ commutes}\right.}\]
    with the obvious restriction maps. In particular, $\ul G(B)=G(\mbfD)$. We next note that there is a map $\ul G\to\ul\GL(\msE_1)\by\G_m$ given, on sections over some $U\openset B$, by $(\phi,\psi)\mapsto(\phi,\lambda)$ where $\lambda\in\G_m(U)$ is uniquely chosen so that
    \[\homses{\Sym^2(\msE_1\vert_U)}{}{\msE_2\vert_U}{}{\msY\vert_U}{\Sym^2(\phi)}\psi\lambda{\Sym^2(\msE_1\vert_U)}{}{\msE_2\vert_U}{}{\msY\vert_U}\]
    commutes. Finally this map fits into a sequence
    \[0\too\underbrace{\sHom(\msY,\Sym^2(\msE_1))}_\msV\too\ul G\too\ul\GL(\msE_1)\by\G_m\too0\]
    which is furthermore exact, as can be checked over an open cover trivializing $\msE_1,\msE_2$. Recall we are interested in computing the order of $G(\mbfD)=\hom^0(B,\ul G)$. The utility of phrasing things as above is that \cite[Proposition 3.3.2.2 + Corollaire 3.3.2.3]{giraud} now gives us an exact sequence
    \begin{equation}\label{es:gir}
        0\too\hom^0(B,\msV)\too\hom^0(B,\ul G)\xtoo F\GL(\msE_1)\by\units k\xtoo d\hom^1(B,\msV)
    \end{equation}
    of pointed sets whose differential $d$ induces an injection (of sets) $\im(F)\bs(\GL(\msE_1)\by\units k)\intoo\hom^1(B,\msV)$.
    Because $\im(F)=\ker(d)$ acts freely on $\GL(\msE_1)\by\units k$, we can take an alternating product of cardinalities in \cref{es:gir} to conclude that
    \[\frac{\#\hom^0(B,\msV)}{\#\hom^0(B,\ul G)}\cdot\#(\GL(\msE_1)\by\units k)=\#\im(d).\]
    The trivial inequalities $1\le\#\im(d)\le\#\hom^1(B,\msV)$ thus give
    \begin{equation*}
       (q-1)\#\GL(\msE_1)q^{\chi(\msV)}=(q-1)\#\GL(\msE_1)\cdot\frac{\#\hom^0(B,\msV)}{\#\hom^1(B,\msV)} \le \#\hom^0(B,\ul G) \le (q-1)\#\GL(\msE_1)\cdot\#\hom^0(B,\msV)
       .
    \end{equation*}
    One easily computes $\deg\msV=3\deg\msL=3d$, so the claimed lower bound follows from Riemann-Roch.
\end{proof}

\subsection{Counting Minimal hW Curves}\label{sect:count-hW}
We continue to work in the context of \cref{set:main}.
\begin{rec}[see \cref{notn:Hs}]
    Recall that $\meH_M(B)\into\meH(B)$ denotes the full subgroupoid consisting of minimal hW curves, and that $\meH_{M,NT}(B)\into\meH_M(B)$ denotes the full subgroupoid consisting of those minimal hW curves $(H\xto\pi B,D)$ for which $D_K$ is not twice a point.
\end{rec}
Recall  that every hW curve is isomorphic to one which can be properly embedded in the projective bundle associated to some unique (up to isomorphism), normalized (1,2,1)-datum (\cref{lem:basic-0} and \cref{cor:basic}). In order to bound the number of minimal hW curves, we will partition them according to their normalized (1,2,1)-datum, and then count the number of curves w/ given (1,2,1)-datum using a combination of \cref{prop:weighted-hW,prop:PG-card}. In order to compute the quantities appearing in these propositions, we will make use of the filtration constructed in \cref{prop:hW-norm-filt}. That being said, let us first name the objects which will appear in our analysis.
\begin{notn}\label{notn:myriad}
    Given a normalized (1,2,1)-datum $\mbfD=(\msE_1,\msE_2,\mu)$, define the following myriad of objects.
    \begin{itemize}
        \item Let $\msD=\msD(\mbfD):=\det(\msE_1)$. If $\msE_1$ is unstable, it has Harder-Narasimhan filtration
        \begin{equation}\label{ses:normalized}
            0\too\msO_B\too\msE_1\too\msD\too0.
        \end{equation}
        \item Let $u=u(\mbfD)$ be the unstable degree of $\msE_1$. This is $0$ if $\msE_1$ is semistable, but is otherwise $\deg\msD<0$, see \cref{defn:normalized}.
        \item Let $\msL=\msL(\mbfD):=\det(\msE_1)\otimes\coker(\mu)^{-1}$ be the Hodge bundle of the datum. 
        \item Let $d=d(\mbfD):=\deg\msL$.
        \item Let $0=\msF_0\subset\dots\subset\msF_9=\msB(\mbfD)_4$ denote the filtration of \cref{prop:hW-norm-filt}. Only $\msF_0,\msF_5,\msF_8,\msF_9$ are defined if $\msE_1$ is semistable.
        \item Let $\msG_i:=\msF_i\otimes\msL^2\otimes\inv[2]\msD$ for all $i$. By \cref{prop:hW-norm-filt}, we always have an exact sequence
        \begin{equation}\label{ses:ss-filt}
            0\too \underbrace{\Sym^4(\msE_1)\otimes\inv[2]\msD\otimes\msL^2}_{\msG_5}\too \msG_8\too\Sym^2(\msE_1)\otimes\inv\msD\otimes\msL\too0,
        \end{equation}
        and if $\msE_1$ is unstable (i.e. if $u<0$), we further have
        \begin{equation}\label{eqn:filt-grad}
            \frac{\msG_{i+1}}{\msG_i}\cong\Threecases{\msD^{i-2}\otimes\msL^2}{0\le i\le4}{\msD^{i-6}\otimes\msL}{5\le i\le 7}{\msO_B}{i=8.}
        \end{equation}
    \end{itemize}
\end{notn}
Motivated by \cref{prop:weighted-hW}, our first task will be to find an upper bound for $\#\hom^0(B,\msG_8)$. Equivalently, in light of Riemann-Roch, we first bound $\#\hom^1(B,\msG_8)$ for the $(1,2,1)$-data relevant to our count. For later use, we also bound $\#\hom^1(B,\msG_8/\msG_5)$.
\begin{rem}\label{rem:u-constraint}
    By \cref{cor:deg(E)-constraint}, all normalized (1,2,1)-data associated to an hW curve in $\meH_{M,NT}(B)$ satisfy $-(d+g)\le u$. Recall also that $u\le0$ always, by definition, and that every hW curve is isomorphic to one whose associated (1,2,1)-datum is normalized, by \cref{lem:hW-normalized}.
\end{rem}
\begin{lemma}\label{lem:G8-bound-us}
    Let $\mbfD$ be a normalized (1,2,1)-datum with $-(d+g)\le u<0$. Furthermore, assume $d>3g$. Then, $h^1(\msG_8)\le7g-2$ and $h^1(\msG_8/\msG_5)\le3g-1$.
\end{lemma}
\begin{proof}
    The existence of the filtration $\msG_i$ of \cref{notn:myriad} shows that $h^1(\msG_8)\le\sum_{i=0}^7h^1(\msG_{i+1}/\msG_i)$ and $h^1(\msG_8/\msG_5)\le\sum_{i=5}^7h^1(\msG_{i+1}/\msG_i)$. With our bounds on $u$, for $i\neq4,7$, $\deg(\msG_{i+1}/\msG_i)>2g-2$ (see \cref{eqn:filt-grad}), so $h^1(\msG_{i+1}/\msG_i)=0$ unless $i=4,7$ (i.e. excluding the graded pieces $\msD^2\otimes\msL^2$ and $\msD\otimes\msL$). Recalling that $u+d\ge-g$ by assumption, for these pieces, one has
    \[h^1(\msD^2\otimes\msL^2)=h^0(\omega_B\otimes\inv[2]\msD\otimes\inv[2]\msL)\le\deg(\omega_B\otimes\inv[2]\msD\otimes\inv[2]\msL)+1=2g-2-2(u+d)+1\le4g-1,\]
    and similarly $h^1(\msD\otimes\msL)\le3g-1$.
\end{proof}
We still need to bound $h^1(\msG_8)$ when $u=0$, i.e. when $\msE_1$ is semistable.
\begin{lemma}\label{lem:ss-sym}
     Let $\msE$ be a rank $r\ge1$ semistable vector bundle on $B$, and fix an integer $k\ge1$. Let $\msM$ be a line bundle on $B$ with $\deg\msM\ge2grk-1$. Then, 
     \[H^1\p{B,\Sym^{rk}(\msE)\otimes\pinv[k]{\det\msE}\otimes\msM}=0.\]
\end{lemma}
\begin{proof}
    First note that the vector bundle $\Sym^{rk}(\msE)\otimes\pinv[k]{\det\msE}$ is unchanged under the substitution $\msE\squigto\msE\otimes\msN$ for any line bundle $\msN$ on $B$. Thus, we may twist $\msE$ in order to assume that 
    \[(2g-1)r<\deg(\msE)\le 2gr,\]
    in particular, that it has slope $\mu(\msE)>2g-1$.  Since $\msE$ is semistable of high slope, \cite[Proposition 10.27]{mukai} tells us that it is globally generated. Fix a surjection $\msO^{\oplus N}_B\onto\msE$. From this, one obtains a surjection $\msE^{\oplus N(rk-1)}=\msE\otimes\p{\msO_B^{\oplus N}}^{\otimes(rk-1)}\onto\msE\otimes\msE^{\otimes(rk-1)}\onto\Sym^{rk}(\msE)$. Tensoring with $\pinv[k]{\det\msE}\otimes\msM$ then gives the surjection
    \begin{equation}\label{surj:ss-sym}
        F:\sq{\msE\otimes\pinv[k]{\det\msE}\otimes\msM}^{\oplus N(rk-1)}\onto\Sym^{rk}(\msE)\otimes\pinv[k]{\det\msE}\otimes\msM
        .
    \end{equation}
    Because $\hom^2(B,\ker F)=0$, \cref{surj:ss-sym} induces a surjection on $\hom^1$'s, so it suffices to show that $\hom^1(B,\msE\otimes\pinv[k]{\det\msE}\otimes\msM)=0$. Because $\msE\otimes\pinv[k]{\det\msE}\otimes\msM$ is semistable with slope
    \[\mu(\msE)-k\deg(\msE)+\deg(\msM)>(2g-1)-2grk+(2grk-1)=2g-2,\] 
    we win by \cite[Proposition 10.26]{mukai}.
\end{proof}
\begin{cor}\label{cor:G8-bound-ss}
    Let $\mbfD$ be a normalized (1,2,1)-datum with $u=0$. Furthermore, assume $d\ge4g$. Then, $h^1(\msG_8)=0$ and $h^1(\msG_8/\msG_5)=0$.
\end{cor}
\begin{proof}
    That $u=0$ means that $\msE_1$ is semistable. Thus, this follows from \cref{ses:ss-filt} along with \cref{lem:ss-sym}.
\end{proof}
Given some $(1,2,1)$-datum $\mbfD$, \cref{prop:weighted-hW,prop:PG-card} tell us that
\[\sum_{H\into\P(\mbfD)}\frac1{\#\Aut_{\meH(B)}(H)}\le\frac{\#\hom^0(B,\msG_8)}{\#\PG}\le\frac{\#\hom^0(B,\msG_8)}{\#\GL(\msE_1)q^{3d+3(1-g)}}.\]
Recall (\cref{defn:121-datum}) that $\msE_1$ above is not an isomorphism invariant of $\mbfD$, but its associated $\PGL_2$-torsor is. Thus, we would like a bound given only in terms of this $\PGL_2$-torsor.

\begin{lemma}\label{lem:Aut(P)}
    Let $\msE$ be a rank 2 vector bundle on $B$, with associated $\PGL_2$-torsor $P=\ul\Isom(\msO^{\oplus2},\msE)\cby{\GL_2}\PGL_2$. Then,
    \[\#\Aut(P)=\frac{\#\GL(\msE)}{q-1}.\]
\end{lemma}
\begin{proof}
    Taking inner twists in $1\to\G_m\to\GL_2\to\PGL_2\to1$ by a cocycle defining $\msE$ gives the exact sequence
    \[0\too\G_m\too\ul\GL(\msE)\too\ul\Aut(P)\too0.\]
    To prove the claim, it suffices to show that this sequence remains exact after taking global sections. Consider the following commutative diagram with top row exact:
    \[\begin{tikzcd}
        \GL(\msE)\ar[r, "(1)"]&\Aut(P)\ar[r]&\hom^1(B,\G_m)\ar[r, "(2)"]\ar[d, equals]&\hom^1(B,\ul\GL(\msE))\ar[d]&T\ar[l, symbol=\in]\ar[d, maps to]\\
        &&\hom^1(B,\G_m)\ar[r, "(3)"]&\hom^1(B,\GL_2)&T\cby{\ul\GL(\msE)}\ul\Isom(\msO^{\oplus2},\msE)\ar[l, symbol=\in]
    \end{tikzcd}\]
    Surjectivity of (1) is equivalent, by exactness of the top row, to injectivity of (2). Commutativity tells us that (2) is injective if (3) is. Finally, (3) is injective because it can be identified with the map sending a line bundle $\msL$ to the rank 2 vector bundle $\msL\oplus\msL$.
\end{proof}
\begin{cor}\label{cor:case-2-ineq}
    Let $\mbfD$ be a normalized (1,2,1)-datum. Then,
    \[\sum_{H\into\P(\mbfD)}\frac1{\#\Aut_{\meH(B)}(H)}\le\frac{\#\hom^0(B,\msG_8)}{(q-1)\cdot\#\Aut(P)q^{3d+3(1-g)}}.\]
\end{cor}
\begin{proof}
    This follows from \cref{prop:weighted-hW,prop:PG-card,lem:Aut(P)}.
\end{proof}
In the end, we will need to understand the sum of the above expressions as $\mbfD$ varies over isomorphism classes of (1,2,1)-data.
\begin{notn}
    \hfill\begin{itemize}
        \item Let $P$ be a $\PGL_2$-torsor on $B$. We let $\msV(P)$ denote the rank 3 vector bundle (associated the to the $\GL_3$-torsor) obtained by pushing $P$ along the $\PGL_2$-representation $\Sym^2(\texttt{taut})\otimes\inv{\texttt{det}}:\PGL_2\to\GL_3$.
        \item Furthermore, extending \cref{notn:myriad}, given a (1,2,1)-datum $\mbfD=(\msE_1,\msE_2,\mu)$, we let $P=P(\mbfD)$ denote the $\PGL_2$-torsor associated to $\msE_1$.
        
        Note that, in this context, $\msV(P)\cong\Sym^2(\msE_1)\otimes\pinv{\det\msE_1}=\Sym^2(\msE_1)\otimes\inv\msD$.
    \end{itemize}
\end{notn}
\begin{lemma}\label{lem:121-iso-bij}
    There is a bijection between isomorphism classes of $(1,2,1)$-data and triples $(P,\msL,\eps)$, where $P\in\hom^1(B,\PGL_2)$, $\msL\in\Pic(B)$, and $\eps\in\Ext^1(\inv\msL,\msV(P))\cong\hom^1(B,\msV(P)\otimes\msL)$.
\end{lemma}
\begin{proof}
    Let $\mbfD=(\msE_1,\msE_2,\mu)$ be a (1,2,1)-datum. Then, $\msV(P(\mbfD))\cong\Sym^2(\msE_1)\otimes\inv\msD$, so the extension $0\to\Sym^2(\msE_1)\xto\mu\msE_2\to\inv{\msL(\mbfD)}\otimes\msD\to0$, after tensoring with $\inv\msD$, gives rise to a class $\eps(\mbfD)\in\Ext^1(\inv\msL,\msV(P))$. In one direction, the bijection is given by $\mbfD\mapsto(P(\mbfD),\msL(\mbfD),\eps(\mbfD))$. This triple is easily checked to be an isomorphism invariant.

    Conversely, suppose we're given $(P,\msL,\eps)$. Because $H^2(B,\G_m)=0$ by \cite[Example III.2.22 Case (g)]{milne-et}, we can choose some rank 2 vector bundle $\msE$ lifting $P$. Having made such a choice, $\eps$ defines an extension $0\to\Sym^2(\msE)\otimes\pinv{\det\msE}\xto{\mu'}\msE'\to\inv\msL\to0$. Observe that $(\msE,\msE'\otimes\det\msE,\mu'\otimes1)$ is a (1,2,1)-datum and that its isomorphism class is independent of the choices made. This gives the other direction of the bijection.
\end{proof}
\begin{notn}\label{notn:G8-iso-inv}
    Given $P,\msL,\eps$ as in \cref{lem:121-iso-bij}, let $\msG_8=\msG_8(P,\msL,\eps)$ denote the (isomorphism class of the) rank $8$ vector bundle $\msG_8(\mbfD)$ associated to any (1,2,1)-datum $\mbfD$ associated to the triple $(P,\msL,\eps)$ via \cref{lem:121-iso-bij}. We similarly define $\msG_i=\msG_i(P,\msL,\eps)$ for all other $i\in\{0,5,8,9\}$.
\end{notn}
\begin{notn}\label{notn:M}
    Let $\Bun_{\PGL_2}(k)$ denote the groupoid of $\PGL_2$-torsors over $B$, and set
    \[M:=\abs{\Bun_{\PGL_2}(k)}=\hom^1(B,\PGL_2),\]
    the set of isomorphism classes of $\PGL_2$-torsors over $B$. Endow $M$ with the discrete measure $m$ where each $[P]\in\hom^1(B,\PGL_2)$ is weighted by $1/\#\Aut(P)$.
\end{notn}
\begin{lemma}\label{lem:BunG-card}
    $\#\Bun_{\PGL_2}(k)=\int_M\dm=2q^{3(g-1)}\zeta_B(2)$
\end{lemma}
\begin{proof}
    Note that the first equality is by definition. Siegel's formula \cite[Theorem 4.8 + Proposition 4.13]{tama} tells us that the Tamagawa number $\tau(\PGL_2)$ of $\PGL_2$ is related to the groupoid cardinality of $\Bun_{\PGL_2}(k)$ via
    \[\frac{\tau(\PGL_2)}{\#\Bun_{\PGL_2}(k)}=q^{(1-g)\dim\PGL_2}\prod_{\t{closed }x\in B}\frac{\#\PGL_2(\kappa(x))}{(\#\kappa(x))^{\dim\PGL_2}}=q^{3(1-g)}\prod_{\t{closed }x\in B}\p{1-q^{-2\deg x}}=q^{3(1-g)}\zeta_B(2)^{-1}.\]
    It is well-know that $\tau(\PGL_2)=2$; this can be deduced e.g. from the main result of \cite{gaitsgory-lurie} (see also \cite[Theorem 6.1]{tama}). Thus, we conclude that $\#\Bun_{\PGL_2}(k)=2q^{3(g-1)}\zeta_B(2)$.
\end{proof}
\begin{thm}\label{thm:mas-count-hard}
    Use notation as in \cref{set:main}. Then,
    \[\limsup_{d\to\infty}\frac{\#\meH_{M,NT}^{=d}(B)}{\#\meM_{1,1}^{= d}(K)}\le2\zeta_B(2)\zeta_B(10).\]
\end{thm}
\begin{proof}
    We begin with a bit of notation. For any $u\in\Z_{<0}$, let $M^{\ge u}\subset M$ denote the subset consisting of isomorphism classes of $\PGL_2$-torsors over $B$ which lift to a rank 2 vector bundle $\msV$ on $B$ which is either semistable or has Harder-Narasimhan filtration of the form $0\to\msO_B\to\msV\to\det\msV\to0$, with $\deg\msV\ge u$.

    Below, when we write $\sum_\mbfD$, we mean that the sum is over isomorphism classes of \important{normalized} (1,2,1)-data.
    \begin{align}\label{ineq:M-NT-init-count}
        \#\meH_{M,NT}^{=d}(B)
        &=\sum_{\substack{\alpha\in\abs{\meH_{M,NT}(B)}\\\Ht(\alpha)=d}}\frac1{\#\Aut_{\meH(B)}(\alpha)}
        \nonumber\\
        &=\sum_{\substack{\mbfD\\d(\mbfD)=d}}\sum_{\substack{H\in\abs{\meH_{M,NT}(B)}\\\mbfD(H)\cong\mbfD}}\frac1{\#\Aut_{\meH(B)}(H)} &&\t{by \cref{lem:hW-normalized}}
        \nonumber\\
        &=\sum_{\substack{\mbfD\\d(\mbfD)=d}}\sum_{\substack{H\into\P(\mbfD)\\H\in\abs{\meH_{M,NT}(B)}}}\frac1{\#\Aut_{\meH(B)}(H)} &&\t{by \cref{cor:basic}}
        \nonumber\\
        &=\sum_{\substack{\mbfD\\d(\mbfD)=d\\-(d+g)\le u(\mbfD)\le0}}\sum_{\substack{H\into\P(\mbfD)\\H\in\abs{\meH_{M,NT}(B)}}}\frac1{\#\Aut_{\meH(B)}(H)} &&\t{by \cref{rem:u-constraint}}
        \nonumber\\
        &\le\sum_{\substack{\mbfD\\d(\mbfD)=d\\-(d+g)\le u(\mbfD)\le0}}\frac{\#\hom^0(B,\msG_8)}{(q-1)\cdot\#\Aut(P)q^{3d+3(1-g)}}&&\t{by \cref{cor:case-2-ineq}}
        \nonumber\\
        &=\sum_{P\in M^{\ge-(d+g)}}\;\sum_{\msL\in\Pic^d(B)}\;\sum_{\eps\in\hom^1(B,\msV(P)\otimes\msL)}\frac{\#\hom^0(B,\msG_8)}{(q-1)\cdot \#\Aut(P)q^{3d+3(1-g)}}&&\t{by \cref{lem:121-iso-bij}}
        \nonumber\\
        &=\int_{M^{\ge-(d+g)}}\sum_{\msL\in\Pic^d(B)}\;\sum_{\eps\in\hom^1(B,\msV(P)\otimes\msL)}\frac{\#\hom^0(B,\msG_8)}{(q-1)q^{3d+3(1-g)}}\dm
        \nonumber\\
        &=\frac1{q-1}\int_M\chi_d(P)\sum_{\msL\in\Pic^d(B)}\;\sum_{\eps\in\hom^1(B,\msV(P)\otimes\msL)}\frac{\#\hom^0(B,\msG_8)}{q^{3d+3(1-g)}}\dm
        ,
    \end{align}
    where $\chi_d:M\to\bits$ is the characteristic function of $M^{\ge-(d+g)}$.
    By \cref{thm:EC-d-asymp}, $\#\meM_{1,1}^{=d}(K)\sim\#\Pic^0(B)\cdot q^{10d+2(1-g)}/\sq{(q-1)\zeta_B(10)}$. Thus,
    \begin{align}\label{ineq:wrong-order}
        \limsup_{d\to\infty}\frac{\#\meH_{M,NT}^{=d}(B)}{\#\meM_{1,1}^{=d}(K)} 
        &= \limsup_{d\to\infty}(q-1)\zeta_B(10)\frac{\#\meH_{M,NT}^{=d}(B)}{\#\Pic^0(B)\cdot q^{10d+2(1-g)}} \nonumber\\
        &\le\frac{\zeta_B(10)}{\#\Pic^0(B)}\lim_{d\to\infty}\int_M{\chi_d(P)\sum_{\msL\in\Pic^d(B)}\;\sum_{\eps\in\hom^1(B,\msV(P)\otimes\msL)}\frac{\#\hom^0(B,\msG_8)}{q^{13d+5(1-g)}}}\dm \nonumber\\
        &=\frac{q^{3(1-g)}\zeta_B(10)}{\#\Pic^0(B)}\lim_{d\to\infty}\int_M\chi_d(P)\underbrace{\sum_{\msL\in\Pic^d(B)}\;\sum_{\eps\in\hom^1(B,\msV(P)\otimes\msL)}q^{h^1(\msG_8)}}_{I_d(P)}\dm
        ,
    \end{align}
    with last equality holding by Riemann-Roch (note that $\deg\msG_8=13d$).
    We would like to commute the limit and integral in \cref{ineq:wrong-order}, so we will bound $I_d(P)$ and then apply dominated convergence. Observe \cref{ses:ss-filt} $\msV(P)\otimes\msL\cong\msG_8/\msG_5$, so \cref{lem:G8-bound-us} and \cref{cor:G8-bound-ss} tell us that $h^1(\msG_8)\le7g-2$ and $h^1(\msV(P)\otimes\msL)\le3g-1$ whenever $d\gg_g1$. Putting these together, whenever $d\gg_g1$ and $P\in M^{\ge-(d+g)}$, we have (with $I_d(P)$ defined as indicated in \cref{ineq:wrong-order})
    \begin{equation}\label{ineq:Id-DCT}
        I_d(P)\le\#\Pic^0(B)\cdot q^{3g-1}\cdot q^{7g-2}=\#\Pic^0(B)q^{10g-3}.
    \end{equation}
    Observe that $\int_M\#\Pic^0(B)q^{10g-3}\dm<\infty$ and $\int_M\lim_{d\to\infty}I_d(P)\dm=\int_M\#\Pic^0(B)\dm<\infty$ (with equality by Serre vanishing) by \cref{lem:BunG-card}. Thus, \cref{ineq:Id-DCT} allows us to apply the Dominated Convergence Theorem (DCT) below:
    \begin{align}\label{ineq:M-NT-frac-final}
        \limsup_{d\to\infty}\frac{\#\meH^{=d}_{M,NT}(B)}{\#\meM^{=d}_{1,1}(K)}
        &\le\frac{q^{3(1-g)}\zeta_B(10)}{\#\Pic^0(B)}\lim_{d\to\infty}\int_M\chi_d(P)I_d(P)\dm &&\t{by \cref{ineq:wrong-order}}
        \nonumber\\
        &=\frac{q^{3(1-g)}\zeta_B(10)}{\#\Pic^0(B)}\int_M\lim_{d\to\infty}{\chi_d(P)I_d(P)}\dm &&\t{by DCT}
        \nonumber\\
        &=\frac{q^{3(1-g)}\zeta_B(10)}{\#\Pic^0(B)}\int_M\#\Pic^0(B)\dm &&\t{by Serre vanishing}
        \nonumber\\
        &=\frac{q^{3(1-g)}\zeta_B(10)}{\#\Pic^0(B)}\cdot\#\Pic^0(B)\cdot2q^{3(g-1)}\zeta_B(2) &&\t{by \cref{lem:BunG-card}}
        \nonumber\\
        &=2\zeta_B(2)\zeta_B(10)
        .
    \end{align}
    This was the claimed inequality.
\end{proof}
\begin{cor}\label{cor:mod-nt-estimate}
    Use notation as in \cref{set:main}. Then,
    \[\limsup_{d\to\infty}\frac{\#\meH^{\le d}_{M,NT}(B)}{\#\meM^{\le d}_{1,1}(K)}\le2\zeta_B(2)\zeta_B(10).\]
\end{cor}
\begin{proof}
    This is a consequence of \cref{thm:EC-d-asymp,thm:mas-count-hard}.
\end{proof}
\begin{thm}[= \cref{cor:cor-mod-estimate}]\label{thm:mod-estimate}
    Use notation as in \cref{set:main}. Then,
    \[\MAS_B\le1+2\zeta_B(2)\zeta_B(10).\]
\end{thm}
We postpone a proof of \cref{thm:mod-estimate} until \cref{sect:main-results}. In light of \cref{eqn:CSelT-card}, the key to deducing \cref{thm:mod-estimate} from \cref{cor:mod-nt-estimate} is showing that 0\% of elliptic curves have a nonzero 2-torsion point. We will verify this in \cref{sect:count-EC-2tors}, and then afterwards prove \cref{thm:mod-estimate} (See \cref{cor:cor-mod-estimate}).

\subsection{Relating $\MAS_B(d)$ and $\AS_B(d)$ in Characteristic $2$}\label{sect:MAS-AS-char-2}
The just-proven \cref{cor:mod-nt-estimate} bounds the average number of non-trivial $2$-Selmer elements for an elliptic curve $E/K$ only when each $2$-Selmer element is given a non-standard weighting. We would like to show that the average computed using this non-standard weighting is at least as large as the average computed using the standard weighting (the $1/\#\Aut(E)$ weighting). In the present section, we verify this when working in characteristic 2 (see \cref{cor:AS-MAS-comp-char-2} for a precise statement). In \cref{sect:main-results} we will give a separate, but simpler, argument which verifies this in all other positive characteristics. 
\begin{set}
    We work within the context of \cref{set:main}. In addition, we assume that $\Char K=2$.
\end{set}
By \cref{lem:sel-aut}, for a given $(C,E,\rho,D)\in\CSel_2$ (representing some $[(C,D)]\in\Sel_2(E)$), its weight in $\MAS_B(d)$ (where, say, $d=\Ht(E)$) differs from its weight in $\AS_B(d)$ only if $E[2](K)\neq0$ or $\#\Aut(E)>2$. Furthermore, its weight in $\MAS_B(d)$ is smaller than its weight in $\AS_B(d)$ only if $E[2](K)\neq0$. Thus, in order to prove that the modified average is at least as large as the usual average, it will suffice to show that elliptic curves with non-trivial 2-torsion contribute $0$ to the average size of 2-Selmer. In the parlance of \cref{sect:count-hW}, this amounts to bounding the number of minimal hW curves (of bounded height) whose generic fibers have Jacobians with non-trivial 2-torsion.
\begin{defn}\label{defn:extra}
    We say that a hW curve $(H/B,D)\in\meH_{M,NT}(B)$ is \define{extra} if the elliptic curve $E=\Jac(H_K)$ has a non-trivial 2-torsion $K$-point.
\end{defn}
The main technical result of this section will show that the number of extra hW curves of height $d$ is $o(\#\meM_{1,1}^{=d}(K))$ (\cref{cor:extra-no-contribution}). Once this is known, we will be able to deduce that $\AS_B\le\MAS_B$ (\cref{cor:AS-MAS-comp-char-2}).

As in \cref{sect:count-hW}, we will partition (extra) hW curves according to their normalized (1,2,1)-datum. As in the aforementioned section, whenever we consider a normalized $(1,2,1)$-datum $\mbfD$, we additionally define the myriad of objects listed in \cref{notn:myriad}. By of the work of that section (in particular, by the application of the dominated convergence theorem in the proof of \cref{thm:mas-count-hard}), for understanding the asymptotic count of extra hW curves of height $d$ as $d\to\infty$, it will suffice to consider normalized (1,2,1)-data $\mbfD$ where $d(\mbfD)$ is much larger than $u(\mbfD)$; see \cref{lem:DCT-shortcut} for a more precise statement.

\subsubsection{Some Preliminary Simplifications}
Before getting started with bounding the number of extra hW curves, we need to set up some additional notation which will be used in the argument.
\begin{notn}
    Let $\meH_E(B)\into\meH_{M,NT}(B)$ denote the full subgroupoid consisting of extra (in the sense of \cref{defn:extra}) hW curves. For $d\ge0$, define $\meH_E^{=d}(B)$ and $\meH_E^{\le d}(B)$ as expected.
\end{notn}
\begin{rec}\label{rec:G8(P L e)}
    By \cref{lem:121-iso-bij}, to a triple $(P,\msL,\eps)$ where $P\in\hom^1(B,\PGL_2)$, $\msL\in\Pic(B)$, and $\eps\in\Ext^1(\inv\msL,\msV(P))$, one can associate a unique isomorphism class of $(1,2,1)$-data. Given such a triple $(P,\msL,\eps)$, we let $\msG_8=\msG_8(P,\msL,\eps)$ and $\msG_9=\msG_9(P,\msL,\eps)$ denote the (isomorphism classes of the) vector bundles alluded to in \cref{notn:G8-iso-inv}.
\end{rec}
\begin{defn}
    Let $(P,\msL,\eps)$ be as in \cref{rec:G8(P L e)}. We say that $(P,\msL,\eps)$ is \define{admissible} if both $\eps=0$ and $0\to\msG_8\to\msG_9\to\msG_9/\msG_8\to0$ splits.
\end{defn}
Note that $(P,\msL,\eps)$ will automatically be admissible if $\deg\msL\gg_P1$.
\begin{defn}
    Let $\mbfD=(\msE_1,\msE_2,\mu)$ be a (1,2,1)-datum. Suppose that there exists a short exact sequence
    \[0\too\msO_B\too\msE_1\too\msD\too0\]
    (note that, here, we do \important{not} assume that $\msE_1$ is unstable or that $\deg\msD<0$), and note that we can use this to extend the filtration $0=\msG_0\subset\msG_5\subset\msG_8\subset\msG_9$ (defined as in \cref{notn:myriad}) to a filtration $0=\msG_0\subset\msG_1\subset\msG_2\subset\dots\subset\msG_9$ whose graded pieces are exactly as indicated in \cref{eqn:filt-grad}. In this case, we say that $\mbfD$ is \define{admissible} if the corresponding triple $(P,\msL,\eps)$, as in \cref{lem:121-iso-bij}, is admissible, $\hom^1(B,\msG_{i+1}/\msG_i)=0$ for all $i\in\{0,1,\dots,8\}$, and $\deg(\msG_{i+1}/\msG_i)>0$ for all $i$. 
\end{defn}
Note that $\mbfD=(\msE_1,\msE_2,\mu)$ will automatically be admissible if $\deg\msL\gg_{\msE_1}1$. Furthermore, if $\mbfD$ is admissible, then $\hom^0(B,\msG_9)\cong\bigoplus_{i=0}^8\hom^0(B,\msG_{i+1}/\msG_i)$.
\begin{notn}
    Given an admissible $(P,\msL,\eps)$, we let $N_E(P,\msL,\eps)$ denote the number of of sections $s_1\in\hom^0(B,\msG_8)$ such that the (isomorphism class of the) hW curve cut out by $s=(s_1,1)\in\hom^0(B,\msG_8)\oplus\hom^0(B,\msO_B)\cong\hom^0(B,\msG_9)$ (the curve `$Z(s)$' in \cref{construct:F}) is extra. This number is independent of all choices involved.
\end{notn}
\begin{lemma}\label{lem:DCT-shortcut}
    Let $M=\hom^1(B,\PGL_2)$ equipped with the measure $m$ defined in \cref{notn:M}. Then,
    \[\limsup_{d\to\infty}\frac{\#\meH^{=d}_E(B)}{\#\meM_{1,1}^{=d}(K)}\le\frac{q^{3(1-g)}\zeta_B(10)}{\#\Pic^0(B)}\int_M\limsup_{d\to\infty}\sum_{\msL\in\Pic^d(B)}\frac{N_E(P,\msL,0)}{q^{13d+5(1-g)}}\dm.\]
\end{lemma}
\begin{proof}
    The main point here is to replace $\meH_{M,NT}(B)$ with $\meH_E(B)$ in the proof of \cref{thm:mas-count-hard}. In doing so, \cref{ineq:M-NT-init-count} becomes
    \begin{equation}\label{ineq:extra-init-bound}
        \#\meH^{=d}_E(B)\le\frac1{q-1}\int_M\chi_d(P)\sum_{\msL\in\Pic^d(B)}\;\sum_{\eps\in\hom^1(B,\msV(P)\otimes\msL)}\frac{N_E(P,\msL,\eps)}{q^{3d+3(1-g)}}\dm,
    \end{equation}
    where $\chi_d:M\to\bits$ is as in the proof of \cref{thm:mas-count-hard}.
    Because the right hand side of \cref{ineq:extra-init-bound} is smaller than the right hand side of \cref{ineq:M-NT-init-count}, one may still use the dominated convergence theorem to bound $\limsup_{d\to\infty}\left.{\#\meH^{=d}_E(B)}\right/{\#\meM_{1,1}^{=d}(K)}$. Doing so gives (compare the below with \cref{ineq:M-NT-frac-final})
    \begin{align*}
        \limsup_{d\to\infty}\frac{\#\meH^{=d}_E(B)}{\#\meM_{1,1}^{=d}(K)}
        &\le \frac{q^{3(1-g)}\zeta_B(10)}{\#\Pic^0(B)}\int_M\limsup_{d\to\infty}\sum_{\msL\in\Pic^d(B)}\;\sum_{\eps\in\hom^1(B,\msV(P)\otimes\msL)}\frac{N_E(P,\msL,\eps)}{q^{3d+3(1-g)}}\dm\\
        &=\frac{q^{3(1-g)}\zeta_B(10)}{\#\Pic^0(B)}\int_M\limsup_{d\to\infty}\sum_{\msL\in\Pic^d(B)}\frac{N_E(P,\msL,0)}{q^{3d+3(1-g)}}\dm,
    \end{align*}
    because Serre vanishing implies that $\hom^1(B,\msV(P)\otimes\msL)=0$ whenever $\deg\msL\gg_P1$.
\end{proof}
The utility of \cref{lem:DCT-shortcut} is that the limit is inside the integral, so we may consider the limit
\[\limsup_{d\to\infty}\sum_{\msL\in\Pic^d(B)}\frac{N_E(P,\msL,0)}{q^{13d+5(1-g)}}\]
separately for each fixed $P\in\hom^1(B,\PGL_2)$. This is the first simplification of our eventual count. The second simplification will be to replace the various vector bundles which show up (the $\msL,\msD,\msG_8,\dots$ of \cref{notn:myriad}) with sums of line bundles coming from divisors supported at a single point.
\begin{notn}
    Fix an arbitrary closed point $P\in B$. Also, given $x\in\R$, let $\ceil x$ denote the least integer greater than or equal to $x$.
\end{notn}
\begin{lemma}\label{lem:embed-in-nP}
    Let $\msM$ be a line bundle on $B$, and set $n:=\deg\msM$. Then, there exists an embedding $\msM\into\msO_B\p{\ceil{\frac{n+g}{\deg P}}P}$.
\end{lemma}
This is an easy consequence of Riemann-Roch applied to $\sHom\p{\msM,\msO_B\p{\ceil{\frac{n+g}{\deg P}}P}}$. To see how we will use this lemma, consider some admissible $(1,2,1)$-datum $\mbfD=(\msE_1,\msE_2,\mu)$. Hence, $\msG_9=\msG_9(\mbfD)$ supports a filtration $0=\msG_0\subset\msG_1\subset\msG_2\subset\dots\subset\msG_9$ whose graded pieces are exactly as indicated in \cref{eqn:filt-grad} and which satisfies $\hom^1(B,\msG_{i+1}/\msG_i)=0$ for all $i$. When $d\gg_{\msE_1}1$, \cref{lem:embed-in-nP} allows us to choose embeddings $\msG_{i+1}/\msG_i\into\msO_B(n_iP)\subset\ul K$, for appropriate $n_i\in\N$. This allows us to view a section $s\in\hom^0(B,\msG_8)$ as a tuple $(c_0,c_1,c_2,c_3,c_4,a_0,a_1,a_2)\in\bigoplus_{i=0}^7\hom^0(B,\msG_{i+1}/\msG_i)\into\bigoplus_{i=0}^7\hom^0(B,\msO_B(n_iP))\subset K^{\oplus8}$ in such a way that the corresponding hW curve $H_s/B$ cut out by $(s,1)\in\hom^0(B,\msG_8)\oplus\hom^0(B,\msO_B)\cong\hom^0(B,\msG_9)$ has generic fiber given by the equation
\[Y^2+(a_0X^2+a_1XZ+a_2Z^2)Y=c_0X^4+c_1X^3Z+c_2X^2Z^2+c_3XZ^3+c_4Z^4\]
inside $\P(1,2,1)_K$. Since the condition of being extra (\cref{defn:extra}) only depends on the generic fiber of $H_s$, in order to bound the number of extra hW curves of height $d$, it will suffice to bound the number of tuples $(c_0,c_1,c_2,c_3,c_4,a_0,a_1,a_2)\in\bigoplus_{i=0}^7\hom^0(B,\msO_B(n_iP))$ such that the curve cut out by the above equation is extra. In particular, this discussion proves \cref{lem:simplification-2}.
\begin{notn}\label{notn:NE-n}
    Given $n_0,\dots,n_7\in\Z$, we let $N_E(n_0,\dots,n_7)$ denote the number of tuples
    \[(c_0,c_1,c_2,c_3,c_4,a_0,a_1,a_2)\in\bigoplus_{i=0}^7\hom^0(B,\msO_B(n_iP))\subset K^{\oplus8}\]
    such that curve
    \[H_K:Y^2+(a_0X^2+a_1XZ+a_2Z^2)Y=c_0X^4+c_1X^3Z+c_2X^2Z^2+c_3XZ^3+c_4Z^4\]
    inside $\P(1,2,1)_K$ is smooth and satisfies $\Jac(H_K)[2](K)\neq0$.
\end{notn}
\begin{lemma}\label{lem:simplification-2}
    Let $\mbfD=(\msE_1,\msE_2,\mu)$ be an admissible (1,2,1)-datum, let $n=\deg\msE_1$, and let $(P,\msL,0)$ be the triple classifying $\mbfD$'s isomorphism class, as in \cref{lem:121-iso-bij}. Then, letting $d=\deg\msL$ as usual, 
    \[N_E(P,\msL,0)\le N_E(n_0,n_1,\dots,n_7),\]
    where
    \begin{equation}\label{eqn:ni-extra}
        n_i = \Twocases{\ceil{\frac{(i-2)n+2d+g}{\deg P}}}{0\le i\le 4}{\ceil{\frac{(i-6)n+d+g}{\deg P}}}{5\le i\le 7}
    \end{equation}
\end{lemma}

\subsubsection{Counting Extra hW Curves}
Let us recap the notation introduced in the previous section.
\begin{set}
    Fix an admissible $(1,2,1)$-datum $\mbfD=(\msE_1,\msE_2,\mu)$. 
    \begin{itemize}
        \item Let $\msL=\msL(\mbfD)$ as in \cref{notn:myriad}, let $d:=\deg\msL$, and let $n:=\deg\msE_1$.
        \item Because $\mbfD$ is admissible, we may and do fix a short exact sequence
        \[0\too\msO_B\too\msE_1\too\msD\too0,\]
        and we may and do extend the filtration $0=\msG_0\subset\msG_5\subset\msG_8\subset\msG_9$ (defined as in \cref{notn:myriad}) to a filtration $0=\msG_0\subset\msG_1\dots\subset\msG_9$ whose graded pieces are
        \begin{equation}\label{eqn:filt-grad-ii}
            \frac{\msG_{i+1}}{\msG_i}\cong\Threecases{\msD^{i-2}\otimes\msL^2}{0\le i\le4}{\msD^{i-6}\otimes\msL}{5\le i\le 7}{\msO_B}{i=8.}
        \end{equation}
        (as indicated in \cref{eqn:filt-grad}).
        \item Fix an arbitrary closed point $P\in B$.
        \item Let $n_0,\dots,n_7$ be as in \cref{eqn:ni-extra}. Choose embeddings $\msG_{i+1}/\msG_i\into\msO_B(n_iP)$ for all $i\in\{0,\dots,7\}$.
        \item We use the above embeddings (and a choice of splitting $\hom^0(B,\msG_9)\cong\bigoplus_{i=0}^8\hom^0(B,\msG_{i+1}/\msG_i)$, which exists by admissibility) to embed $\hom^0(B,\msG_8)\into\bigoplus_{i=0}^7\hom^0(B,\msO_B(n_iP))\subset K^{\oplus8}$, and so implicitly identify sections $s\in\hom^0(B,\msG_8)$ with tuples $(\alpha_0,\alpha_1,\dots,\alpha_7)\in\bigoplus_{i=0}^7\hom^0(B,\msO_B(n_iP))$.

        We write $(\alpha_0,\dots,\alpha_7)$ instead of the usual $(c_0,c_1,c_2,c_3,c_4,a_0,a_1,a_2)$ to denote the corresponding tuple in order to slightly simplify the notation in this section.
    \end{itemize}
\end{set}
\begin{notn}
    Given $s\in\hom^0(B,\msG_8)$, we let $H_s/K$ denote the generic fiber of the hW curve cut out by $s$; letting $s'=(s,1)\in\hom^0(B,\msG_8)\oplus\hom^0(B,\msO_B)\cong\hom^0(B,\msG_9)$, this is the curve `$Z(s')$' appearing in \cref{construct:F}. More concretely, this is the curve
    \begin{equation}
        H_s:Y^2 + (\alpha_5X^2+\alpha_6XZ+\alpha_7Z^2)Y=\alpha_0X^4+\alpha_1X^3Z+\alpha_2X^2Z^2+\alpha_3XZ^3+\alpha_4Z^4,
    \end{equation}
    where $\alpha_i\in\hom^0(B,\msO_B(n_iP))\subset K$.
\end{notn}
The first thing we need is a criterion of testing whether or not some section $s\in\hom^0(B,\msG_8)$ cuts out an extra hW curve.
\begin{lemma}\label{lem:2-tors-poly-weier}
    Let $S$ be an integral $\F_2$-scheme.  Let $\msM$ be a line bundle on $S$. Let $a_i\in\hom^0(B,\msM^i)$ for $i=1,2,3,4,6$.
    Let $W$ be the Weierstrass curve defined by 
    \[W:Y^2Z+a_1XYZ+a_3YZ^2=X^3+a_2X^2Z+a_4XZ^2+a_6Z^3\]
    in $\P \colonequals \P(\msO_B\oplus\inv[2]\msM\oplus\inv[3]\msM)$.
    The involution
    \[-1:[X,Y,Z]\mapstoo[X,Y+a_1X+a_3Z,Z]\]
    of $\P$ preserves $W$.
    If $W \to S$ has a non-identity section fixed by $-1$,
    then there exists $z\in\hom^0(B,\msM^5)$ such that
    \begin{equation}\label{eqn:2-tors-poly}
        z^2=a_1a_3^3+a_1^2a_2a_3^2+a_1^3a_3a_4+a_1^4a_6.
    \end{equation}
\end{lemma}

\begin{proof}
    Let $K$ be the function field of $S$. Fix an embedding $\msM\subset\ul K$ into the sheaf of rational functions on $S$. This induces embeddings $\msM^n\subset\ul K$ for all $n$, so we may treat the $a_i$ as elements of $K$. Let $\eta\in S$ denote the generic point. Since $\sigma$ is not the identity section, we may write $\sigma(\eta)=(x,y)\in\A^2(K)$. Thus, we have
    \[y^2+a_1xy+a_3y=x^3+a_2x^2+a_4x+a_6\tand y=y+a_1x+a_3\]
    for some $x,y,a_1,a_2,a_3,a_4,a_6\in K$. The second equation tells us that $0=a_1x+a_3$, so $y^2+a_1xy+a_3y=y^2$. Hence, multiplying the above displayed equation by $a_1^4$, we see that    
    \begin{equation}\label{eqn:blah}
        (a_1^2y)^2=a_1(a_1x)^3+a_1^2a_2(a_1x)^2+a_1^3a_4(a_1x)+a_1^4a_6=a_1a_3^3+a_1^2a_2a_3^2+a_1^3a_3a_4+a_1^4a_6.
    \end{equation}
    Set $z=a_1^2y$. 
    By \cref{eqn:blah}, $z^2 \in \hom^0(B,\msM^{10})$, so $z \in\hom^0(B,\msM^5)$.
\end{proof}
\begin{lemma}\label{lem:taut-jacobian}
    Write $\F_2[\ul T]\coloneqq\F_2[T_0,\dots,T_7]$ and $\A^8_{\F_2}=\spec\F_2[\ul T]$. Consider the tautological hW curve $H'/\A^8_{\F_2}$ given by
    \[H'\colon Y^2+(T_5X^2+T_6XZ+T_7Z^2)Y=T_0X^4+T_1X^3Z+T_2X^2Z^2+T_3XZ^3+T_4Z^4\]
    inside $\P(1,2,1)_{\A^8_{\F_2}}$. Let $U\subset\A^8_{\F_2}$ be the locus above which $H'$ is smooth, and let $H:=H'_U/U$. Then, its Jacobian $J:=\Pic^0_{H/U}$ has a Weierstrass model in $\P^2_U$.
\end{lemma}
\begin{proof}
    First note that $H$ is a smooth family of genus 1 curves, so $J$ is an elliptic scheme. Let $S\subset J$ denote its identity section, so $(J,S)$ is a Weierstrass curve over $U$. Since $\Pic(U)=0$, \cref{thm:weier-sum}\bp{3,4} shows us that $J$ has a Weierstrass model in $\P(\msE)$, for some rank $3$ vector bundle $\msE$ supporting a filtration whose graded pieces are $\msO_U,\msO_U,\msO_U$. At the same time, it is well known (see e.g. \cite[Section 2]{liu:hyp-fr}) that $U$ is the nonvanishing locus of a single discriminant polynomial $\Delta\in\F_2[\ul T]$, and so is affine. Hence, any short exact sequence of vector bundles over $U$ splits, so we conclude that $\msE\simeq\msO_U^{\oplus3}$ is trivial. Thus, $J$ has a Weierstrass model in $\P(\msE)\simeq\P(\msO_U^{\oplus3})=\P^2_U$, as claimed.
\end{proof}
\begin{lemma}\label{lem:2-tors-poly-hW}
    There exists a polynomial $F\in\F_2[\ul T]$ with the following property. Given $s\in\hom^0(B,\msG_8)$ such that $H_s$ is a smooth $K$-curve, its Jacobian $E:=\Jac H_s$ satisfies $E[2](K)\neq0$ only if the quantity
    \[F(\alpha_0,\dots,\alpha_7)\in K\]
    is a square. Furthermore, we may choose $F$ so that
    \begin{equation}\label{eqn:F-poly-triv-case}
        F(0,1,a_2,a_4,a_6,0,a_1,a_3)\in\p{a_1a_3^3+a_1^2a_2a_3^2+a_1^3a_3a_4+a_1^4a_6}\cdot(\F_2[a_1,a_2,a_3,a_4,a_6])^2
    \end{equation}
    (the polynomial appearing in the right hand side above is the right hand side of \cref{eqn:2-tors-poly}).
\end{lemma}
\begin{proof}
    Let $U,H,J$ be as in \cref{lem:taut-jacobian}. Let $\msO(U)$ denote the coordinate ring of $U$. By that lemma, we may write
    \begin{equation}\label{eqn:taut-J-weier}
        J\colon Y^2Z+a_1'XYZ+a_3'YZ^2=X^3+a_2'X^2Z+a_4'XZ^2+a_6'Z^3\tforsome a_i'\in\msO(U).
    \end{equation}
    Our desired polynomial $F$ will come from applying \cref{lem:2-tors-poly-weier} to \cref{eqn:taut-J-weier}. However, in order to guarantee that \cref{eqn:F-poly-triv-case} holds, we will need to choose the coefficients $a_i'$ somewhat carefully.

    To this end, write $\F_2[\ul A]:=\F_2[A_1,A_2,A_3,A_4,A_6]$, let $\A^5_{\F_2}:=\spec\F_2[\ul A]$, and consider the tautological Weierstrass curve $W'/\A^6_{\F_2}$ given by
    \[W'\colon Y^2Z+A_1XYZ+A_3YZ^2=X^3+A_2X^2Z+A_4XZ^2+A_6Z^3\]
    inside $\P^2_{\A^5_{\F_2}}$. As before, let $V\subset\A^5_{\F_2}$ be the locus above which $W'$ is smooth and let $W:=W'_V/V$, so $W\subset\P^2_V$ is an elliptic scheme over $V$. Letting $S\subset W$ denote the identity section, note that $(W,2S)$ is a hW curve with $W$ as its Jacobian and whose associated hyper-Weierstrass equation is $Y^2+(A_1X+A_3Z)YZ=(X^3+A_2X^2Z+A_4XZ^2+A_6Z^3)Z$. In other words, the map $(0,1,A_2,A_4,A_6,0,A_1,A_3):\A^5_{\F_2}\to\A^8_{\F_2}$ restricts to a map $\phi:V\into U$ satisfying $\pull\phi J\simeq W$. 
    
    Write $\msO(V)$ for the coordinate ring of $V$. Since $\pull\phi J\simeq W$, general theory (see e.g. \cite[The discussion around (1.2)]{deligne}) guarantees the existence of some linear change of coordinates on $\P^2_V$ which maps $\pull\phi J$ (embedded in $\P^2$ via \cref{eqn:taut-J-weier}) onto $W$. We claim that this change of coordinates lifts to $\P^2_U$, and so that we may arrange that the $a_i'\in\msO(U)$ satisfy $\pull\phi a_i'=A_i\in\F_2[\ul A]\subset\msO(V)$.

    To prove that we can change coordinates on $\P^2_U$ (instead of just $\P^2_V$) in the desired fashion, it suffices to show that the restriction maps $\pull\phi:\msO(U)\to\msO(V)$ and $\pull\phi:\units{\msO(U)}\to\units{\msO(V)}$ are both surjective. We note that both $U$ and $V$ are affine. Indeed, as remarked in the proof of \cref{lem:taut-jacobian}, $U\subset\A^8$ is the nonvanishing locus of a single discriminant polynomial $\Delta_H\in\F_2[\ul T]$, described e.g. in \cite[Section 2]{liu:hyp-fr}; similarly, $V\subset\A^5$ is the nonvanishing locus of the usual Weierstrass discriminant $\Delta_W\in\F_2[\ul A]$, described e.g. in \cite[Section III.1]{silverman}. Hence, $\msO(U)\to\msO(V)$ is surjective simply because $\phi:V\into U$ is a closed immersion. Furthermore, we can compute that $\units{\msO(U)}=\units{\F_2[\ul T][1/\Delta_H]}=\Delta_H^\Z$ and $\units{\msO(V)}=\Delta_W^\Z$, so $\pull\phi(\Delta_H)=\Delta_W^n\in\F_2[\ul A][1/\Delta_W]$ for some $n\in\Z$. We claim that $n=1$. This follows from computing the discriminant of the elliptic curve $Y^2Z+XYZ=X^3+A_6Z^3$ and that of its associated hW curve:
    \begin{align*}
        A_6^n 
        = \Delta_W(1,0,0,0,A_6)^n
        = \pull\phi(\Delta_H)(1,0,0,0,A_6)
        = \Delta_H(0,1,0,0,A_6,0,1,0)
        = -A_6(1+432A_6)
        = A_6
    \end{align*}
    (see \cite[Proposition A.1.1]{silverman} and \cite[Section 2]{liu:hyp-fr} for more details on computing these discriminants). Thus, $\pull\phi(\Delta_H)=\Delta_W$, proving surjectivity of $\units{\msO(U)}\to\units{\msO(V)}$.

    Therefore, as earlier remarked, we may choose that $a_i'\in\msO(U)$ so that they satisfy $\pull\phi a_i'=A_i\in\F_2[\ul A]$. To finish, we scale these by appropriate even powers of $\Delta_H$ so that $a_i'\in\F_2[\ul T]\subset\msO(U)=\F_2[\ul T][1/\Delta_H]$ and then apply \cref{lem:2-tors-poly-weier} to \cref{eqn:taut-J-weier}.
\end{proof}
\begin{rem}
    For $F$ as in \cref{lem:2-tors-poly-hW}, it follows from \cref{eqn:F-poly-triv-case} that, for any $\F_2$-algebra $R$ and any $\alpha\in R$,
    \begin{equation}\label{eqn:F-poly-special-case}
        F(0,1,0,0,\alpha,0,1,0)=\alpha\beta^2\tforsome\beta\in R.
    \end{equation}
    We will later use this remark to show that certain polynomials constructed from $F$ are not squares.
\end{rem}
Our task is now to bound the number of tuples $(\alpha_0,\alpha_1,\dots,\alpha_7)\in\bigoplus_{i=0}^7\hom^0(B,\msO_B(n_iP))$ for which $F(\alpha_0,\dots,\alpha_7)\in K$ is a square. For this, we will find it helpful to decompose each $\alpha_i$ into (approximately) a sum of a square and a nonsquare; doing so will ultimately allow us to reduce the problem of checking that $F$ evaluates to a square into the simpler problem of checking whether some auxiliary polynomial(s) vanish.
\begin{notn}
    Fix a choice of $t\in K$ which is regular away from $P$, but which has a pole of odd order $r\ge1$ at $P$. In particular, $t\not\in K^2$.
\end{notn}
\begin{lemma}\label{lem:square-basis}
    For every $m\ge0$ and every $s\in\hom^0(B,\msO_B(mP))$, it is possible to write $s=s_0^2+ts_1^2+e$ for some $e\in\hom^0(B,\msO_B((4g+6+r)P))$ and some $s_0,s_1\in K$ such that $s_0^2,ts_1^2\in\hom^0(B,\msO_B(mP))$.
\end{lemma}
\begin{proof}
    If $m\le4g+6+r$, take $s_0=0$, $s_1=0$, and $e=m$. We inductively show the claim holds for $m\ge4g+7+r$ as well. Fix such an $m$ as well as some $s\in\hom^0(B,\msO_B(mP))$. 
    \begin{itemize}
        \item Suppose that $m$ is even, say $m=2k$. Then, $k>2g+3$. Riemann-Roch guarantees the existence of some $s_0'\in\hom^0(B,\msO_B(kP))$ having a pole of order $k$ at $P$ and with any prescribed leading coefficient. Since the residue field $\kappa(P)$ satisfies $\kappa(P)=\kappa(P)^2$, we may choose $s_0'$ so that $(s_0')^2$ has a pole of order $m=2k$ at $P$ with the same leading coefficient as $s$ has. Thus, $s-(s_0')^2\in\hom^0(B,\msO_B((m-1)P))$.
        \item Suppose that $m$ is odd, say $m=r+2k$. Then, $k>2g+3$ and we argue as in the previous case. Riemann-Roch and perfectness of $\kappa(P)$ together allow us to find some $s_1'\in\hom^0(B,\msO_B(kP))$ such that $t(s_1')^2$ has a pole of order $m$ at $P$ with the same leading coefficient that $s$ has. Thus, $s-t(s_1')^2\in\hom^0(B,\msO_B((m-1)P)$.
    \end{itemize}
    In either case, the inductive hypothesis applied to $m-1$ allows us to write $s$ in the desired form.
\end{proof}
\begin{notn}
    Fix a polynomial $F\in\F_2[\ul T]$ as in \cref{lem:2-tors-poly-hW}, and set $E:=\hom^0(B,\msO_B((4g+6+r)P))$. For each $\ul e=(e_0,\dots,e_7)\in E^8$, let
    \[G_{\ul e}\colonequals F\p{U_0^2+tV_0^2+e_0,\dots,U_7^2+tV_7^2+e_7}\in K\sq{U_0^2,\dots,U_7^2,V_0^2,\dots,V_7^2}\equalscolon K\sq{\ul U^2,\ul V^2}.\]
    Note that $\{1,t\}$ is a basis for $K\sq{\ul U^2,\ul V^2}$ over $(K^2)[\ul U^2,\ul V^2]=(K[\ul U,\ul V])^2$, and write
    \[G_{\ul e}=G_{\ul e,0}^2+tG_{\ul e,1}^2\tforsome G_{\ul e,0},G_{\ul e,1}\in K[\ul U,\ul V].\]
\end{notn}
We will want to bound how often each $G_{\ul e}$ evaluates to a square in $K$ (i.e. how often each $G_{\ul e,1}$ evaluates to $0$). For this, we will appeal to the following version of the Schwartz-Zippel lemma.
\begin{lemma}[Schwartz-Zippel]\label{lem:schartz-zippel}
    Let $A$ be an integral domain, and let $Q\in A[x_1,\dots,x_n]$ be a nonzero polynomial. Let $d:=\deg Q$ denote the total degree of $Q$. For $1\le j\le n$, let $I_j$ be any finite, nonempty subset of $A$, and write $N_j=\#I_j$. Then, $Q$ has at most
    \[N_1N_2\dots N_n\cdot\sum_{j=1}^n\frac d{N_j}\]
    zeroes in the set $I_1\by\dots\by I_n$.
\end{lemma}
\begin{proof}
    This is (a slightly weaker version of) \cite[Lemma 1]{schwartz}.
\end{proof}
\begin{lemma}\label{lem:Gm-nonzero}
    For any $\ul e\in E^8$, the polynomial $G_{\ul e}$ is not a square in $K[\ul U, \ul V]$, i.e. $G_{\ul e,1}\neq0$.
\end{lemma}
\begin{proof}
    Fix any choice of $\ul e=(e_0,\dots,e_7)$. Since $K=K^2+tK^2$, we may write $e_i=\alpha_i^2+t\beta_i^2$ for all $i$. After doing so, we have
    \[G_{\ul e}=G_{\ul 0}\p{U_0 + \alpha_0,\dots, U_7+\alpha_7,V_0+\beta_0,\dots,V_7+\beta_7}.\]
    Thus, it suffices to show that $G_{\ul0}$ is not a square in $K[\ul U,\ul V]$. For this, we use \cref{eqn:F-poly-special-case} to compute that
    \[G_{\ul 0}\p{0,1,0,0,0,0,1,0,0,0,0,0,1,0,0,0}=F(0,1,0,0,t,0,1,0)=t\beta^2\]
    for some $\beta\in K$. In particular, it is not a square, so neither is the polynomial $G_{\ul0}$.
\end{proof}
\begin{prop}\label{prop:SZ-appli}
    Recall $N_E(\dots)$ from \cref{notn:NE-n}. When $n_0,\dots,n_7$ are chosen as in \cref{eqn:ni-extra}, $N_E(n_0,\dots,n_7)\le O(q^{12.5d})$ as $d\to\infty$. Here, the implicit constant arising from this argument depends on the choice of the the number $n\in\Z$ in \cref{eqn:ni-extra}, the closed point $P\in B$, and the polynomial $F$ as well as the field $K$.
\end{prop}
\begin{proof}
    By \cref{lem:square-basis}, given $(\alpha_i)_{i=0}^7\in\bigoplus_{i=0}^7\hom^0(B,\msO_B(n_iP))$ we may write $\alpha_i=\alpha_{i,0}^2+t\alpha_{i,1}^2+e_i$ for some
    \[\alpha_{i,0},\alpha_{i,1}\in\hom^0\p{B,\msO_B\p{\ceil{\frac{n_i}2}P}}\tand e_i\in\hom^0(B,\msO_B((4g+6+r)P)).\]
    Hence, by \cref{lem:2-tors-poly-hW}, the tuple $(\alpha_i)_{i=0}^7\in\bigoplus_{i=0}^7\hom^0(B,\msO_B(n_iP))$ will cut out a smooth hW curve (over $K$) whose Jacobian has non-trivial 2-torsion only if $G_{\ul e}(\alpha_{0,0},\dots,\alpha_{7,0},\alpha_{0,1},\dots,\alpha_{7,1})\in K$ is a square, where $\ul e\coloneqq(e_0,\dots,e_7)$, i.e. only if $G_{\ul e,1}(\alpha_{0,0},\dots,\alpha_{7,0},\alpha_{0,1},\dots,\alpha_{7,1})=0$. To ease notation, set $m_i\coloneqq\ceil{n_i/2}$. When $d$ is large, the number of possibilities for $(\alpha_{0,0},\dots,\alpha_{7,1})$ is\footnote{For the inequality labelled \cref{eqn:ni-extra} appearing here, we used the inequality 
    \[2y\ceil{\frac{\ceil{x/y}}2}\le2y\p{\frac{x/y+1}2+1}=x+3y\]
    for real numbers $x,y\ge1$.}
    \[\prod_{i=0}^7\#\hom^0\p{B,\msO_B\p{m_iP}}^2=q^{2\sum m_i\deg P+16(1-g)}\overset{\cref{eqn:ni-extra}}\le q^{13d+8g+24\deg P+16(1-g)}=q^{13d+24\deg P+16-8g}\]
    By \cref{lem:Gm-nonzero,lem:schartz-zippel}, for fixed $\ul e$, the proportion of such tuples causing $G_{\ul e,1}$ to vanish is at most
    \[\sum_{i=0}^7\frac{2\deg G_{\ul e,1}}{\#\hom^0(B,\msO_B(m_iP))}=\sum_{i=0}^7O\pfrac1{q^{n_i/2}}=O\pfrac1{q^{d/2}}.\]
    Since there are a constant, finite number of $\ul e$'s, we conclude that $N_E(n_0,\dots,n_7)$ is at most $O(q^{12.5d})$, as claimed.
\end{proof}
\begin{cor}\label{cor:extra-no-contribution}
    \[\lim_{d\to\infty}\frac{\#\meH_E^{=d}(B)}{\#\meM_{1,1}^{=d}(K)}=0\]
\end{cor}
\begin{proof}
    Combine \cref{lem:DCT-shortcut,lem:simplification-2} with \cref{prop:SZ-appli}.
\end{proof}

\subsubsection*{\bf Consequences for $\AS_B(d)$} In the remainder of this section, we use \cref{cor:extra-no-contribution} to compare $\MAS_B(d)$ and $\AS_B(d)$. Roughly, \cref{cor:extra-no-contribution} essentially shows that elliptic curves with non-trivial 2-torsion contribute nothing to the average size of 2-Selmer, so \cref{lem:sel-aut} will show that $\AS_B\le\MAS_B$, i.e. that the modified weighting can only lead to a larger average. More accurately, \cref{cor:extra-no-contribution} bounds the contribution of \important{non-trivial} $2$-Selmer elements attached to elliptic curves with non-trivial 2-torsion. We still also need to bound the contribution of trivial 2-Selmer elements attached to such curves, i.e. we need to bound the number of such curves. After doing so, we will show, in \cref{cor:AS-MAS-comp-char-2}, that $\AS_B\le\MAS_B$.
\begin{thm}\label{thm:E[2]-bound-char-2}
    Assume $\Char K=2$. The weighted number of elliptic curves $E/K$ of height $d$ with $E[2](K)\neq0$ is $O\p{q^{9d}}$ as $d\to\infty$.
\end{thm}
\begin{proof}
    Consider an elliptic curve $E/K$ of height $d$ sufficiently large for $E$'s minimal Weierstrass model to be cut out by a Weierstrass equation (in the sense of \cref{defn:weier-eqn}) for which $E[2](K)\neq0$. Letting $\msL\in\Pic^d(B)$ denote $E$'s Hodge bundle, \cref{lem:2-tors-poly-weier} thus tells us that any minimal Weierstrass equation
    \[Y^2Z + a_1XYZ + a_3YZ^2 = X^3 + a_2X^2Z + a_4XZ^2 + a_6Z^3\]
    for $E$ must satisfy 
    \[z^2=a_1a_3^3+a_1^2a_2a_3^2+a_1^3a_3a_4+a_1^4a_6\]
    for some $z\in\hom^0(B,\msL^5)$. Note that we must have $a_1\neq0$ above since $E[2](K)\neq0$. Indeed, if $a_1=0$, then the existence of a point fixed by negation would force $a_3=0$; however, in this case, $E$, the generic fiber of this equation, would be singular, a contradiction. Because $a_1\neq0$, we see that $a_6$ is determined by the choices of $z,a_1,a_2,a_3,a_4$. Hence, the total number of Weierstrass equations cutting out curves with Hodge bundle $\cong\msL$ and which support a non-trivial $2$-torsion point is at most
    \[\#\hom^0(\msL^5)\cdot\prod_{i=1}^4\#\hom^0(\msL^4)=q^{15d+5(1-g)}.\]
    Finally, arguing as in \cref{cor:WE-from-UW}, we conclude that the count of elliptic curves $E/K$, weighted by $1/\#\Aut(E)$, of height $d$ with $E[2](K)\neq0$ is at most
    \[\frac{\#\Pic^0(B)\cdot q^{15d+5(1-g)}}{(q-1)q^{6d+3(1-g)}}=O\p{q^{9d}}. \qedhere\]
\end{proof}

Now, we will find it useful to define the \define{intermediate average size of $2$-Selmer}:
\begin{equation}\label{eqn:IAS-def}
    \IAS_B(d)\colonequals\frac{N(d)}{\#\meM_{1,1}^{\le d}(K)}\twhere N(d)\coloneqq\sum_{\substack{E/K\\\Ht(E)\le d\\E[2](K)=0}}\frac{\#\Sel_2(E)}{\#\Aut(E)}.
\end{equation}
(observe that this only incorporates 2-Selmer elements of elliptic curves with trivial 2-torsion subgroups).
\begin{prop}\label{prop:IAS-AS-comp-char2}
    Use notation as in \cref{set:main}, and assume that $\Char K=2$. Then,
    \[\lim_{d\to\infty}\IAS_B(d)=\lim_{d\to\infty}\AS_B(d).\]
\end{prop}
\begin{proof}
    First remark that
    \[\AS_B(d)-\IAS_B(d)=\frac{E(d)}{\#\meM_{1,1}^{\le d}(K)}\twhere E(d):=\sum_{\substack{E/K\\\Ht(E)\le d\\E[2](K)\neq0}}\frac{\#\Sel_2(E)}{\#\Aut(E)}.\]
    Next, recall the definition of the 2-Selmer groupoid (\cref{defn:Selmer-groupoid}). By construction, there is a bijection between isomorphism classes of objects of $\CSel_2$ and pairs $(E,\alpha)$, where $E$ is an isomorphism class of elliptic curves and $\alpha\in\Sel_2(E)$ (See \cref{rem:selmer-interpretation}). Thus,
    \[E(d)=\sum_{\substack{[(C,E,\rho,D)]\in\abs{\CSel_2}\\\Ht(E)\le d\\E[2](K)\neq0}}\frac1{\#\Aut(E)}.\]
    At the same time, an easy consequence of \cref{lem:sel-aut} is that, for any $(C,E,\rho,D)\in\CSel_2$, one has $4\#\Aut(E)\ge\#\Aut(C,E,\rho,D)$. Combined with the above displayed equality, this shows that
    \[E(d)\le4\sum_{\substack{[(C,E,\rho,D)]\in\abs{\CSel_2}\\\Ht(E)\le d\\E[2](K)\neq0}}\frac1{\#\Aut(C,E,\rho,D)},\]
    i.e. $E(d)$ is bounded above by 4 times the cardinality of the subgroupoid $\CSel_{2,E}^{\le d}\into\CSel_2^{\le d}$ consisting of 2-Selmer elements attached to elliptic curves $E/K$ with $E[2](K)\neq0$. It is clear from \cref{construct:hW-Sel-func} that the functor $F$ in \cref{prop:sel-conn} sends an hW curve to an object in $\CSel_{2,E}$ if and only if that curve is extra in the sense of \cref{defn:extra}. Thus, \cref{prop:sel-conn} shows that the (groupoid) cardinality of the `non-trivial' objects in $\CSel_{2,E}^{\le d}$ is bounded above by $\#\meH_E^{\le d}$. As such (the first summands counts the `trivial' objects in $\CSel^{\le d}_{2,E}$),
    \[E(d)\le 4\sum_{\substack{E/K\\\Ht(E)\le d\\E[2](K)\neq0}}\frac1{\#\Aut_{\CSel_2}(E,E,\rho_E,2O)}+4\#\meH_E^{\le d}.\]
    Combining \cref{thm:E[2]-bound-char-2} with the observation that $\#\Aut_{\CSel_2}(E,E,\rho_E,2O)=\#E[2](K)\cdot\#\Aut(E)$ shows that the first summand above is $O(q^{9d})$. Similarly, \cref{cor:extra-no-contribution} shows that the second summand is $o\p{\#\meM_{1,1}^{\le d}(K)}$. Finally, \cref{thm:EC-d-asymp} allows us to conclude that $E(d)=o\p{\#\meM_{1,1}^{\le d}(K)}$, from which the claim follows.
\end{proof}
\begin{lemma}\label{lem:IAS-MAS-comp}
    Use notation as in \cref{set:main} (with no assumption on $\Char K$). Then,
    \[\limsup_{d\to\infty}\IAS_B(d)\le\limsup_{d\to\infty}\MAS_B(d).\]
\end{lemma}
\begin{proof}
    For this, one simply observes that the numerator $N(d)$ of $\IAS_B(d)$ can be expressed as a sum over isomorphism classes of objects of $\CSel_2$ and that \cref{lem:sel-aut} shows that $\#\Aut_{\CSel_2}(C,E,\rho,D)\le\#\Aut(E)$ if $E[2](K)=0$. Put together, these say that
    \[N(d)=\sum_{\substack{E/K\\\Ht(E)\le d\\E[2](K)=0}}\frac{\#\Sel_2(E)}{\#\Aut(E)}=\sum_{\substack{[(C,E,\rho,D)]\in\abs{\CSel_2}\\\Ht(E)\le d\\E[2](K)=0}}\frac1{\#\Aut(E)}\le\sum_{\substack{[(C,E,\rho,D)]\in\abs{\CSel_2}\\\Ht(E)\le d\\E[2](K)=0}}\frac1{\#\Aut_{\CSel_2}(C,E,\rho,D)}\le\#\CSel^{\le d},\]
    from which the claim follows.
\end{proof}
\begin{cor}\label{cor:AS-MAS-comp-char-2}
    Use notation as in \cref{set:main}, and assume that $\Char K=2$. Then,
    \[\AS_B\le\MAS_B.\]
\end{cor}
\begin{proof}
    Combine \cref{lem:IAS-MAS-comp} with \cref{prop:IAS-AS-comp-char2}.
\end{proof}

\section{\bf Proof of the Main Result}\label{sect:main-results}
In this section, we prove \cref{thma:main} (see \cref{thm:main}). We first extend the work of \cref{sect:MAS-AS-char-2} by relating $\MAS_B(d)$ and $\AS_B(d)$ in characteristics away from $2$. The argument used in \cref{sect:MAS-AS-char-2} heavily used that $\Char K=2$ and so does not apply in other characteristics; however, we will see that separate, simpler arguments suffice when $\Char K\neq2$. In brief, we will bound the weighted number of elliptic curves $E/K$ with $E[2](K)\neq0$ and separately obtain a strong enough on bound the sizes of their 2-Selmer groups in order to conclude that such curves do not contribute to the average size of 2-Selmer. Completing the first of these tasks (i.e. bounding the number of $E/K$ with $E[2](K)\neq0$) will already suffice to complete the proof of \cref{thm:mod-estimate} which states that
\[\MAS_B\le1+2\zeta_B(2)\zeta_B(10),\]
with $B,\zeta_B$ as in \cref{set:main} and $\MAS_B(d)$ defined in \cref{eqn:MAS-def}. Afterwards, completing the second task as well (i.e. bounding $\#\Sel_2(E)$ for appropriate $E$) will allow us to show that \[\AS_B\le\MAS_B,\] 
with $\AS_B(d)$ defined in \cref{eqn:AS-def}, even when $\Char K\neq2$. This will complete the proof of \cref{thma:main}.

\subsection{Some Bounds for Elliptic Curves with non-trivial 2-torsion, when $\Char K\neq2$}\label{sect:count-EC-2tors}
Work throughout in the context of \cref{set:main}. We will first bound the number of elliptic curves $E/K$ with $E[2](K)\neq0$. When $\Char K=2$, this was done already in \cref{thm:E[2]-bound-char-2}, so we focus here on the case $\Char K\neq2$. Our bound will be based on the existence of Weierstrass equations, so we begin by recalling the following.
\begin{rec}[\cref{prop:high-height=>weier-eqn}]
    Any Weierstrass curve $W/B$ of height $>N(g):=\max\{-1,2g-2\}$ is cut out by some Weierstrass equation
    \[Y^2Z + a_1XYZ + a_3YZ^2 = X^3 + a_2X^2Z + a_4XZ^2 + a_6Z^3\]
    in $\P:=\P\p{\msO_B\oplus\inv[2]\msL\oplus\inv[3]\msL}$, where $\msL$ is $W$'s Hodge bundle and $a_i\in\hom^0(B,\msL^i)$.
\end{rec}
\begin{assump}
    Assume $\Char K\neq2$.
\end{assump}

\subsubsection{The Number of Such Curves}
Let $E/K$ be an elliptic curve with Hodge bundle $\msL$ of height $d:=\deg\msL>N(g)$. Let $(W\xto\pi B,S)$ be its minimal Weierstrass model, so $W$ is given by some Weierstrass equation
\[Y^2Z + a_1XYZ + a_3YZ^2 = X^3 + a_2X^2Z + a_4XZ^2 + a_6Z^3\twith a_i\in\Gamma(B,\msL^i)\]
(inside of $\P:=\P(\msO_B\oplus\inv[2]\msL\oplus\inv[3]\msL)$). Note that negation on $E$ extends to the morphism
$$-1:\sq{X,Y,Z}\mapstoo\sq{X,-Y-a_1X-a_3Z,Z}$$
on $W\subset\P$. Suppose that $E$ has a non-trivial 2-torsion point $P\in E[2](K)$. By the valuative criterion of properness, $P$ extends to a section $\sigma:B\to W$.
Using the universal property of $\P$, the map $B\xto\sigma W\into\P$ corresponds to some line bundle $\msM\in\Pic(B)$ along with a surjection
\[\msO_B\oplus\inv[2]\msL\oplus\inv[3]\msL\onto\msM.\]
We first observe that, in fact, $\msM$ must be trivial. Indeed, it follows from \cite[Proposition VII.3.1(a)]{silverman} that, because $\Char K\neq2$, the image $\sigma(B)\subset W$ is disjoint from the zero section $S\subset W$, i.e. $P$ does not reduce to the identity at any place. Thus, $\sigma$ misses the subscheme $\{Z=0\}\subset W$, so the surjection $\msO_B\oplus\inv[2]\msL\oplus\inv[3]\msL\onto\msM$ defining $\sigma$ restricts to a map $\msO_B\to\msM$ which is non-vanishing in every fiber. Since $\msO_B,\msM$ are line bundles, this must in fact be an isomorphism.

The upshot is that we may view the section $\sigma$ as the triple $[\sigma_X,\sigma_Y,1]$ where $\sigma_X\in\Gamma(B,\msL^2)=\Hom(\inv[2]\msL,\msO_B)$ and $\sigma_Y\in\Gamma(B,\msL^3)=\Hom(\inv[3]\msL,\msO_B)$. Since $\sigma$ lands in $W\subset\P$, these are required to satisfy
\[\sigma^2_Y+a_1\sigma_X\sigma_Y+a_3\sigma_Y=\sigma_X^3+a_2\sigma_X^2+a_4\sigma_X+a_6.\]
Furthermore, since $P$ is 2-torsion, i.e. since $P=-P$, they must also satisfy
\begin{equation}\label{eqn:a3-determined}
    \sigma_Y = -\sigma_Y-a_1\sigma_X-a_3 \tandso a_3=-2\sigma_Y-a_1\sigma_X.
\end{equation}
Combining the previous two equations, we get that
\begin{align}\label{eqn:a6-determined}
    -\sigma_Y^2=\sigma_X^3+a_2\sigma_X^2+a_4\sigma_X+a_6 \tandso a_6 = -\sigma_Y^2-\sigma_X^3-a_2\sigma_X^2-a_4\sigma_X.
\end{align}
\begin{thm}\label{thm:E[2]-bound-char-not-2}
    Assume $\Char K\neq2$. The weighted number of elliptic curves $E/K$ of height $d$ with $E[2](K)\neq0$ is $O\p{q^{6d}}$ as $d\to\infty$.
\end{thm}
\begin{proof}
    Consider a pair $(E,P)$ of an elliptic curve $E/K$ of height $d>N(g)$ along with a choice of non-identity point $P\in E[2](K)$. The above discussion shows that $(E,P)$ arises from some tuple
    \[\p{\msL,a_1,a_2,a_3,a_4,a_6,\sigma_X,\sigma_Y}\]
    with $\msL\in\Pic^d(B)$, $a_i\in\hom^0(B,\msL^i)$, $\sigma_X\in\hom^0(B,\msL^2)$, and $\sigma_Y\in\hom^0(B,\msL^3)$. Furthermore, \cref{eqn:a3-determined} shows that $a_3$ is completely determined once $a_1,\sigma_X,\sigma_Y$ are chosen. Similarly, \cref{eqn:a6-determined} shows that $a_6$ is determined once $a_2,a_4,\sigma_X,\sigma_Y$ are chosen. Thus, the entire tuple is determined once one chooses $\msL$ followed by choosing $a_1,a_2,a_4,\sigma_X,\sigma_Y$. Therefore, the total number of possible tuples is bounded above by
    \[\#\Pic^d(B)\cdot\#\hom^0(\msL)\cdot\#\hom^0(\msL^2)\cdot\#\hom^0(\msL^4)\cdot\#\hom^0(\msL^2)\cdot\#\hom^0(\msL^3)=\#\Pic^0(B)\cdot q^{12d+5(1-g)},\]
    with equality by Riemann-Roch since $d>N(g)$.
    Finally, arguing as in \cref{cor:WE-from-UW}, we conclude that the count of pairs $(E,P)$, weighted by $1/\#\Aut(E)$, of height $d$ is at most
    \[\frac{\#\Pic^0(B)\cdot q^{12d+5(1-g)}}{(q-1)q^{6d+3(1-g)}}=O\p{q^{6d}}.\qedhere\]
\end{proof}

\subsubsection{The Size of 2-Selmer for Such Curves}
We will also need a bound on the 2-Selmer groups of such elliptic curves.
\begin{lemma}\label{lem:num-5}
    Let
    \[\begin{tikzcd}
        A_1\ar[r, "f_1"]\ar[d, "\alpha"]&B_1\ar[r]\ar[d, "\beta"]&C_1\ar[r]\ar[d, "\gamma"]&D_1\ar[r, "g_1"]\ar[d, "\delta"]&E_1\ar[d, "\eps"]\\
        A_2\ar[r, "f_2"]&B_2\ar[r]&C_2\ar[r]&D_2\ar[r, "g_2"]&E_2
    \end{tikzcd}\]
    be a homomorphism of exact sequences of abelian groups. Then,
    \[\#\ker\gamma\le\#\ker\beta\cdot\#(\ker\delta\cap\ker g_1)\cdot\#\coker\p{\im(f_1)\xto\beta\im(f_2)}\le\#\ker\beta\cdot\#\ker\delta\cdot\#\im f_2.\]
\end{lemma}
\begin{proof}
    Consider the homomorphisms
    \[\homses{\coker f_1}{}{C_1}{}{\ker g_1}{\bar\beta}\gamma\delta{\coker f_2}{}{C_2}{}{\ker g_2}\tand\homses{\im f_1}{}{B_1}{}{\coker f_1}{}\beta{\bar\beta}{\im f_2}{}{B_2}{}{\coker f_2}\]
    of short exact sequences. Applying the snake lemma to both of them immediately shows that
    \[\#\ker\gamma\le\#\ker\bar\beta\cdot\#(\ker\delta\cap\ker g_1)\tand\#\ker\bar\beta\le\#\ker\beta\cdot\#\coker\p{\im(f_1)\to\im(f_2)}.\qedhere\]
\end{proof}
\begin{lemma}\label{lem:ell-sch-dual-ses}
    Let $S$ be an arbitrary scheme, and let $\msE/S$ be an elliptic scheme. Let $\alpha\into\msE$ be a finite locally free $S$-group scheme of order $n$, and let $\dual\alpha:=\ul\Hom(\alpha,\G_m)$ be its Cartier dual. Then, there is a short exact sequence
    \[0\too\alpha\too\msE[n]\too\dual\alpha\too0\]
    of abelian sheaves on $\fppf S$.
\end{lemma}
\begin{proof}
    Consider the quotient map $q:\msE\onto\msE/\alpha=:\msE'$ 
    as well as its dual $\dual q:\msE'\to\msE$. Since $\dual qq=[n]:\msE\to\msE$, we get a short exact sequence of kernels
    \[0\too\ker q\too\msE[n]\too\ker\dual q\too0.\]
    Now, $\ker q=\alpha$ by construction, and so $\ker\dual q\simeq\dual\alpha$ by \cite[Corollary 1.3(ii)]{oda-dR}.
\end{proof}
\begin{prop}\label{prop:2-tors-Selmer-bound}
    Let $S\subset B$ be the set of places of bad reduction for $E$. Assume that $E[2](K)\neq0$. Then, 
    \[\dim_{\F_2}\Sel_2(E)\le3\#S+2\dim_{\F_2}\Pic^0(B)[2]+2\le3\#S+4g+2.\]
\end{prop}
\begin{proof}
    Let $U=B\sm S$ be the locus of good reduction for $E$, and let $\msE/U$ be $E$'s N\'eron model. Note that $[2]:\msE\to\msE$ is a flat (even \'etale) cover, so we can form the following commutative diagram with exact rows:
    \[\begin{tikzcd}
        0\ar[r]&\frac{\msE(U)}{2\msE(U)}\ar[d]\ar[r]&\hom^1(U,\msE[2])\ar[r]\ar[d]&\hom^1(U,\msE)[2]\ar[r]\ar[d]&0\\
        0\ar[r]&\prod_{v\in S}\frac{E(K_v)}{2E(K_v)}\ar[r, "\prod_v\delta_v"]&\prod_{v\in S}\hom^1(K_v,E[2])\ar[r]&\prod_{v\in S}\hom^1(K_v,E)[2]\ar[r]&0
        .
    \end{tikzcd}\]
    Now, it is not hard to show that $\hom^1(U,\msE[2])\subset\hom^1(K,E[2])$ consists exactly of cohomology classes which are everywhere unramified over $U$, and so produce a injection
    \[\Sel_2(E)\into\b{c\in\hom^1(U,\msE[2]):c_v\in\im\delta_v\tforall v\in S}=:G.\]
    Hence, it suffices to bound $\dim_{\F_2}G$. For this, we observe that it sits in a short exact sequence
    \[0\too\underbrace{\ker\p{\hom^1(U,\msE[2])\too\prod_{v\in S}\hom^1(K_v,E[2])}}_A\too G\too\prod_{v\in S}\frac{E(K_v)}{2E(K_v)}\too0.\]
    We separately bound the sizes of $A$ (defined in the above displayed sequence) and $\prod_{v\in S}E(K_v)/2E(K_v)$.
    \begin{itemize}
        \item For $A$, we first remark that, by \cref{lem:ell-sch-dual-ses}, we have a short exact sequence $0\to\ul{\zmod2}_U\to\msE[2]\to\mu_{2,U}\to0$. Comparing this with the analogous sequences over $K_v$ for $v\in S$, taking cohomology, and observing that $\ul{\zmod2}_U\simeq\mu_{2,U}$, we obtain
        \[\begin{tikzcd}
            \zmod 2\ar[d]\ar[r]&\hom^1(U,\zmod 2)\ar[d, "\beta"]\ar[r]&\hom^1(U,\msE[2])\ar[d, "\gamma"]\ar[r]&\hom^1(U,\zmod 2)\ar[d, "\delta"]\ar[r]&\hom^2(U,\zmod 2)\ar[d]\\
            (\zmod 2)^{\#S}\ar[r]&\prod_{v\in S}\hom^1(K_v,\zmod 2)\ar[r]&\prod_{v\in S}\hom^1(K_v,E[2])\ar[r]&\prod_{v\in S}\hom^1(K_v,\zmod 2)\ar[r]&\prod_{v\in S}\hom^2(K_v,\zmod2).
        \end{tikzcd}\]
        We now apply \cref{lem:num-5} to conclude that
        \begin{equation}\label{ineq:dim-A}
            \dim_{\F_2}A=\dim_{\F_2}\ker\gamma\le\dim_{\F_2}\ker\beta+\dim_{\F_2}\ker\delta+\#S=2\dim_{\F_2}\ker\beta+\#S,
        \end{equation}
        so we are reduced to bounding the size of
        \[B:=\ker\p{\hom^1(U,\zmod2)\xtoo\beta\prod_{v\in S}\hom^1(K_v,\zmod2)}.\]
        Note that
        \[\hom^1(U,\zmod 2)\simeq\ctsHom(G_{K,U},\zmod2)\tand\hom^1(K_v,E[2])\simeq\ctsHom(G_{K_v},\zmod2),\]
        where $K^s$ (resp. $K_v^s$) is the maximal separable extension of $K$ (resp. $K_v$), $G_K$ (resp. $G_{K_v}$) is the absolute Galois group of $K$ (resp. $K_v$), and $G_{K,U}=\Gal(K_U/K)$, where $K_U$ is the maximal extension of $K$ unramified above $U$. Thus any element of $B$ is represented by an everywhere unramified continuous homomorphism $G_K\to\zmod 2$, so $B\subset\ctsHom(\Pic(B),\zmod2)$
        by class field theory. As $\Pic(B)\cong\Pic^0(B)\by\Z$, this says that $B\subset\Hom(\Z,\zmod 2)\by\Hom(\Pic^0(B),\zmod 2)$. The first factor here is $\cong\zmod 2$, while the second factor has dimension $\dim_{\F_2}\Pic^0(B)[2]$. Recalling \cref{ineq:dim-A}, we conclude
        \[\dim_{\F_2}A\le2+2\dim_{\F_2}\Pic^0(B)[2]+\#S.\]
        \item For $\prod_{v\in S}E(K_v)/2E(K_v)$, we simply use the fact that, for each $v$, $E(K_v)$ is a profinite group with a finite index pro-$p$ subgroup (recall $p=\Char K\neq2$), and so $\#E(K_v)/2E(K_v)=\#E(K_v)[2]\le4$. Thus, $\dim_{\F_2}\prod_{v\in S}E(K_v)/2E(K_v)\le2\#S$.
    \end{itemize}
    The claim follows from combining these two bullet points.
\end{proof}
What remains is to convert this bound into one expressed in terms of the height of $E$ instead of its number of places of bad reduction.
\begin{lemma}\label{lem:div-supp-smol-deg}
    Let $D\subset B$ be an effective divisor. Fix $x\in\R$ such that every point in the support of $D$ has degree $<x$. Then,
    \[\#\supp D\le\frac{2g+2}{q-1}q^{x+1}.\]
\end{lemma}
\begin{proof}
    One can deduce from the Hasse-Weil bound that $\#B(\F_{q^r})\le(2g+2)q^r$ for any $r\ge1$. Hence,
    \[\#\supp D\le\sum_{1\le r<x}\#B(\F_{q^r})\le(2g+2)\sum_{r=1}^{\floor x}q^r\le\frac{2g+2}{q-1}q^{x+1}.\qedhere\]
\end{proof}
\begin{prop}\label{prop:div-bounds}
    Let $D\subset B$ be an effective divisor of degree $d\ge2$. Then,
    \begin{alignat}{4}
        \#\supp D &\le\frac{2d\log q}{\log d}+\frac{(2g+2)q}{q-1}\sqrt d&&=O\pfrac d{\log d}\label{ineq:supp-bound}
        .
    \end{alignat}
\end{prop}
\begin{proof}
    Write $D=\sum_pn_p[p]$, so $d=\deg D=\sum_pn_p\deg p$. Consider the function $f(x):=\frac12\frac{\log x}{\log q}=\log_q(\sqrt x)$, and split $D$ as $D=D_1+D_2$, where
    \[D_1=\sum_{\substack{p\\\deg p<f(d)}}n_p[p]\tand D_2=\sum_{\substack{p\\\deg p\ge f(d)}}n_p[p].\]
    By \cref{lem:div-supp-smol-deg}, we have
    \[\#\supp D_1\le\frac{2g+2}{q-1}q^{f(d)+1}=\frac{(2g+2)q}{q-1}\sqrt d.\]
    Furthermore,
    \[d=\deg D\ge\deg D_2\ge f(d)\sum_{\deg p\ge f(d)}n_p=f(d)\cdot\#\supp D_2\tandso \#\supp D_2\le\frac d{f(d)}=\frac{2d\log q}{\log d}.\]
    The claim follows, as $\#\supp D\le\#\supp D_2+\#\supp D_1$.
\end{proof}
\begin{prop}\label{prop:weak-selmer-bound}
    Use notation as in \cref{set:main}. Assume that $\Char K\neq2$. Let $E/K$ be an elliptic curve with conductor $N\in\Div(B)$, and set $n:=\deg N$. If $E[2](K)\neq0$, then
    \[\dim_{\F_2}\Sel_2(E)\le3\sq{\frac{2n\log q}{\log n}+\frac{(2g+2)q}{q-1}\sqrt n}+4g+2.\]
    In particular, if one restricts attention to elliptic curves $E/K$ with $E[2](K)\neq0$, then
    \[\dim_{\F_2}\Sel_2(E)=O\pfrac n{\log n}\le O\pfrac{\Ht(E)}{\log\Ht(E)}\]
    as $n\to\infty$.
\end{prop}
\begin{proof}
    The first part follows simply from combining \cref{prop:2-tors-Selmer-bound} with \cref{ineq:supp-bound} along with the observation that the set of bad places for $E$ is precisely $\supp N$. The asymptotic claim of the theorem statement is clear once one notes that $n\le\deg\Delta=12\Ht(E)$, with the inequality following e.g. from Ogg's formula \cite[Theorem 2]{ogg}, and the equality holding by \cref{rem:height-disc}.
\end{proof}

\subsection{Bounding the Average Size of 2-Selmer}\label{sect:final}
We are now in a position to prove \cref{thma:main}. We begin by completing the proof of \cref{thm:mod-estimate}, which we will restate below for the reader's convenience.
\begin{rec}
    Let $K$ be the function field of a smooth curve $B/\F_q$. Recall that $\meM_{1,1}(K)$ denotes the groupoid of elliptic curves over $K$, and that $\meM^{\le d}_{1,1}(K)$ denotes its full subgroupoid consisting elliptic curves of height $\le d$. Furthermore, recall the functions
    \[\AS_B(d):=\frac{\displaystyle\sum_{\substack{E/K\\\Ht(E)\le d}}\frac{\#\Sel_2(E)}{\#\Aut(E)}}{\#\meM^{\le d}_{1,1}(K)}\tand\MAS_B(d):=\frac{\#\CSel_2^{\le d}}{\#\meM_{1,1}^{\le d}(K)}\]
    defined in \cref{eqn:AS-def} and \cref{eqn:MAS-def}, respectively.
\end{rec}
\begin{prop}\label{prop:count-triv-sel}
      The groupoid $\CSel_{2,T}$ of trivial 2-Selmer elements (\cref{notn:Sel2NT}) satisfies 
      \[\lim_{d\to\infty}\frac{\#\CSel_{2,T}^{\le d}}{\#\meM_{1,1}^{\le d}(K)}=1.\]
\end{prop}
\begin{proof}
    Observe
    \begin{align*}
        \#\CSel_{2,T}^{\le d}
        &= \sum_{\substack{E/K\\\Ht(E)\le d}}\frac1{\#E[2](K)\cdot\#\Aut(E)} &&\t{by \cref{eqn:CSelT-card}}\\
        &= \sum_{\substack{E/K\\\Ht(E)\le d\\E[2](K)\neq0}}\frac1{\#E[2](K)}\cdot\frac1{\#\Aut(E)}+\sum_{\substack{E/K\\\Ht(E)\le d\\E[2](K)=0}}\frac1{\#\Aut(E)}\\
        &= \sum_{\substack{E/K\\\Ht(E)\le d}}\frac1{\#\Aut(E)}-\sum_{\substack{E/K\\\Ht(E)\le d\\E[2](K)\neq0}}\p{1-\frac1{\#E[2](K)}}\frac1{\#\Aut(E)}\\
        &\ge\sum_{\substack{E/K\\\Ht(E)\le d}}\frac1{\#\Aut(E)}-\sum_{\substack{E/K\\\Ht(E)\le d\\E[2](K)\neq0}}\frac1{\#\Aut(E)} &&\t{since }1-\frac1{\#E[2](K)}\le1
        .
    \end{align*}
    It is clear from \cref{eqn:CSelT-card} that $\#\Sel_{2,T}^{\le d}\le\#\meM_{1,1}^{\le d}(K)$. Combined with the above, we have
    \begin{equation}\label{ineq:N(d)}
        \#\meM_{1,1}^{\le d}(K)-\sum_{\substack{E/K\\\Ht(E)\le d\\E[2](K)\neq0}}\frac1{\#\Aut(E)}\le\#\CSel_{2,T}^{\le d}\le\#\meM_{1,1}^{\le d}(K).
    \end{equation}
    The claim now follows from dividing \cref{ineq:N(d)} by $\#\meM_{1,1}^{\le d}(K)$ and comparing the asymptotics obtained in \cref{thm:EC-d-asymp} and \cref{thma:E[2]}.
\end{proof}
\begin{cor}[= \cref{thm:mod-estimate}]\label{cor:cor-mod-estimate}
    Fix notation as in \cref{set:main}. Then,
    \[\MAS_B\le1+2\zeta_B(2)\zeta_B(10).\]
\end{cor}
\begin{proof}
     $\#\CSel_2^{\le d}=\#\CSel_{2,T}^{\le d}+\#\CSel_{2,NT}^{\le d}$, so combine \cref{prop:count-triv-sel} with \cref{cor:sel-conn,cor:mod-nt-estimate}.
\end{proof}
Continue to work within the context of \cref{set:main}. To prove \cref{thma:main}, it now suffices to prove the inequality $\AS_B\le\MAS_B$.
Recall that \cref{cor:AS-MAS-comp-char-2} showed that the above holds when $\Char K=2$. Furthermore, recalling the quantity
\[\IAS_B(d):=\frac{N(d)}{\#\meM_{1,1}^{\le d}(K)}\twhere N(d):=\sum_{\substack{E/K\\\Ht(E)\le d\\E[2](K)=0}}\frac{\#\Sel_2(E)}{\#\Aut(E)}\]
from \cref{eqn:IAS-def}, \cref{lem:IAS-MAS-comp} shows that $\limsup_{d\to\infty}\IAS_B(d)\le\limsup_{d\to\infty}\MAS_B(d)$ in every characteristic. Thus, it will suffice to compare $\IAS_B(d)$ and $\AS_B(d)$ when $\Char K\neq2$.
\begin{prop}\label{prop:comp-IAS-AS}
    Assume that $\Char K\neq2$. Then,
    \[\lim_{d\to\infty}\AS_B(d)=\lim_{d\to\infty}\IAS_B(d).\]
\end{prop}
\begin{proof}
    We first remark that
    \[\AS_B(d)-\IAS_B(d)=\frac{E(d)}{\#\meM_{1,1}^{\le d}(K)}\twhere E(d):=\sum_{\substack{E/K\\\Ht(E)\le d\\E[2](K)\neq0}}\frac{\#\Sel_2(E)}{\#\Aut(E)}.\]
    By combining \cref{thma:E[2]} with \cref{prop:weak-selmer-bound}, we see that
    \[E(d)=O\p{q^{9d}}\cdot O\p{2^{d/\log d}}=O\p{q^{9d+d/\log d}}\]
    as $d\to\infty$. Since, by \cref{thm:EC-d-asymp}, $\#\meM_{1,1}^{\le d}(K)\sim Cq^{10d}$ for some positive constant $C$, we conclude that $\lim_{d\to\infty}E(d)/\#\meM^{\le d}_{1,1}(K)=0$, from which the claim follows.
\end{proof}
\begin{thm}[= \cref{thma:main}]\label{thm:main}
    Fix notation as in \cref{set:main}. Then, $\AS_B\le1+2\zeta_B(2)\zeta_B(10)$.
\end{thm}
\begin{proof}
    Combine \cref{cor:cor-mod-estimate} with \cref{cor:AS-MAS-comp-char-2} if $\Char K=2$ or with \cref{lem:IAS-MAS-comp} and \cref{prop:comp-IAS-AS} if $\Char K\neq2$.
\end{proof}

\appendix
\appendixpage
\addappheadtotoc

\section{\bf Applications of Cohomology and Base Change}
\numberwithin{thm}{section}

We will need to apply the theorem of cohomology and base change in several places throughout this paper. In order to limit how much we repeat ourselves, we collect some standard consequences in this appendix.
\begin{thm}[Cohomology and Base Change]\label{thm:coh-bc}
    Let $f:X\to B$ be a proper, finitely presented morphism of schemes, and let $\msF$ be a finitely presented sheaf on $X$ which is flat over $B$. Suppose that for a point $b\in B$ and an integer $i$, the comparison map
    $$\phi^i_b:R^i\push f\msF\otimes\kappa(b)\too\hom^i(X_b,\msF_b)$$
    is surjective. Then, all of the following hold.
    \begin{enumerate}
        \setcounter{enumi}{-1}
        \item $\phi^i_b$ is an isomorphism.
        \item there is an open neighborhood $V\subset B$ of $b$ s.t. for any morphism $B'\xto gV$ of schemes, the comparison map
        $$\phi^i_{B'}:\pull gR^i\push f\msF\isoo R^i\push f'(\pull{g'}\msF)$$
        is an isomorphism. Above, $f',g'$ are the morphisms in the Cartesian square
        $$\commsquare{X'}{g'}X{f'}f{B'}g{B.}$$
        In particular, if $\phi^i_b$ is surjective for all $b\in B$, then formation of $R^i\push f\msF$ commutes with arbitrary base change.
        \item $\phi_b^{i-1}$ is surjective if and only if $R^i\push f\msF$ is a vector bundle in an open neighborhood of $b$.
        
        In particular, $\phi_b^{i-1}$ is surjective for all $b\in B$ if and only if $R^i\push f\msF$ is a vector bundle on $B$.
    \end{enumerate}
\end{thm}
\begin{proof}
    See \cite[Theorem 25.1.6]{ravi-notes} and \cite[Theorem A.7.5]{alper}.
\end{proof}
\begin{lemma}\label{lem:glob-gen}
    Let $f:X\to B$ be a morphism of schemes. Let $\msL$ be a line bundle on $X$ such that $\push f\msL$ is a vector bundle on $B$ whose formation commutes with arbitrary base change. Suppose that, for each $b\in B$, the fibral line bundle $\msL_b:=\msL\vert_{X_b}$ on $X_b$ is globally generated. Then, the natural map
    $$\pull f\push f\msL\too\msL$$
    is surjective.
\end{lemma}
\begin{proof}
    This argument comes from the proof of \cite[Proposition A.7.10]{alper}. Surjectivity can be checked on stalks. Applying Nakyama to the cokernels of the maps on stalks, we see that surjectivity can even be checked on the fibers of the line bundles. Thus, it also suffices to check that $\p{\pull f\push f\msL}\vert_{X_b}\too\msL\vert_{X_b}=\msL_b$ is surjective for each $b\in B$. Note that the left hand side is the pullback of $\push f\msL$ along the composition $X_b\into X\xto fB$, which is equivalently the composition $X_b\xto{f_b}\spec\kappa(b)\xinto bB$, so we are asking for surjectivity of the induced map
    $$\hom^0(X_b,\msL_b)\otimes\msO_{X_b}=\pull f_b\p{\wt{\hom^0(X_b,\msL_b)}}\simeq\pull f_b\p{\push f\msL\otimes\kappa(b)}\too\msL_b,$$
    where the second isomorphism holds since the formation of $\push f\msL$ commutes with base change along $\spec\kappa(b)\xinto bB$. The above map is surjective since $\msL_b$ is globally generated by assumption, so we win.
\end{proof}
\begin{lemma}\label{lem:fibral-gs+ds}
    Let $\pi:\meC\to B$ be a $B$-curve (see \cref{sect:conventions} for our definition of `curve'). Furthermore, assume that, for all $b\in B$, one has $\hom^0(\meC_b,\msO_{\meC_b})=\kappa(b)$ and $\omega_{\meC_b}\simeq\msO_{\meC_b}$. Then, $\push\pi\msO_\meC=\msO_B$ holds after arbitrary base change, and $\omega_{X/B}=\pull\pi\msL$ for a unique $\msL\in\Pic(B)$. In fact, $\msL\simeq\push\pi\omega_{\meC/B}$, whose formation will also commute with arbitrary base change.
\end{lemma}
\begin{proof}
    We wish to apply cohomology and base change, \cref{thm:coh-bc}. We will first apply it to $\msF=\msO_\meC$ (with $i=0$). The comparison map
    $$\phi^0_b:\push\pi\msO_\meC\otimes\kappa(b)\too\hom^0(\meC_b,\msO_{\meC_b})=\kappa(b)$$
    is nonzero (e.g. since it's a ring map, so $1\mapsto1$) and so surjective (for all $b\in B$). Therefore, by \cref{thm:coh-bc}, it is an isomorphism and $\push\pi\msO_\meC$ is a line bundle whose formation commutes with arbitrary base change. Now, the natural map $\msO_B\to\push\pi\msO_\meC$ is an isomorphism on fibers since it fits into the below commutative diagram (recall $\phi_b^0$ is itself an isomorphism)
    $$\compdiag{\kappa(b)}{}{\push\pi\msO_\meC\otimes\kappa(b)}{\phi^0_b}{\kappa(b).}{\id}$$
    Thus, $\msO_B\iso\push\pi\msO_\meC$ as desired.
    
    Now, since $h^2(\meC_b,\msO_{\meC_b})=0$ for all $b\in B$, \cref{thm:coh-bc} with $i=2$ applied to $\msF=\msO_\meC$ shows that $R^2\push f\msO_\meC=0$ and so (by part \bp3 of that theorem) $\phi_b^1$ is surjective for all $b\in B$. Since we saw above that also $\phi_b^0$ is surjective for all $b\in B$, another application of \cref{thm:coh-bc}, this time with $i=1$, to $\msF=\msO_\meC$ shows that $R^1\push\pi\msO_\meC$ is a vector bundle on $B$ of rank
    $$h^1(\meC_b,\msO_{\meC_b})=h^0(\meC_b,\omega_{\meC_b})=h^0(\meC_b,\msO_{\meC_b})=1$$
    whose formation commutes with arbitrary base change. By duality, we then conclude that $\msL:=\push\pi\omega_{\meC/B}\simeq\pdual{R^1\push\pi\msO_\meC}$ is a line bundle whose formation commutes with arbitrary base change as well. We claim that $\pull\pi\msL\simeq\omega_{\meC/B}$. This is because \cref{lem:glob-gen} gives a surjection $\pull\pi\msL\onto\omega_{\meC/B}$ and a surjective map between equal rank vector bundles is necessarily an isomorphism. Finally, uniqueness of this choice of $\msL$ follows from the projection formula, which guarantees that, if $\omega_{\meC/B}\simeq\pull\pi\msM$, then $\push\pi\omega_{\meC/B}\simeq\push\pi\msO_\meC\otimes\msM\simeq\msM$.
\end{proof}

\section{\bf Basic Geometry of Weighted Projective Space}
At a few points, we would like to use Theorem 1.4.1 and Theorem 3.3.4 from Dolgachev's paper \cite{wpv} on weighted projective varieties. However, he has a running assumption that for results about $\P(a_0,\dots,a_r)$ over a field $k$, he always assumes $\Char k\nmid a_i$ for all $i$. In this paper, we need to deal with $\P(1,2,1)$ in characteristic $2$. For completeness, here we prove special cases of Dolgachev's results which suffice for our purposes.

\begin{lemma}\label{lem:triv-locus-closed}
    Let $f:X\to Y$ be a flat, proper morphism of noetherian schemes with integral geometric fibers. For a line bundle $\msL$ on $X$, the locus
    \[\b{y\in Y:\msL_y\simeq\msO_{X_y}}\subset Y\]
    is closed.
\end{lemma}
\begin{proof}
    Since the fibers of $f$ are geometrically integral and proper, $\msL_y\simeq\msO_{X_y}$ if and only if both $h^0(X_y,\msL_y)$ and $h^0(X_y,\inv\msL_y)$ are nonzero. Given this, the claim follows from semicontinuity \cite[Theorem 12.8]{hart}. 
\end{proof}

To be clear, everything below appears already in \cite{wpv}, except they technically include a mild characteristic restriction there.
\begin{lemma}\label{lem:ds-p121}
    Let $\P(a_1,\dots,a_r)$ with $\gcd(a_i)=1$, viewed as a scheme over any field $k$. Then, its dualizing sheaf is $\msO(-a_0-\dots-a_r)$.
\end{lemma}
\begin{proof}
    Let $\P:=\P(a_1,\dots,a_r)_\Z$ be the corresponding weighted projective space over $\spec\Z$, and let $\msL:=\omega_\P\otimes\msO_\P(a_0+\dots+a_r)$. It suffices to show that $\msL$ has trivial fibers over all of $\spec\Z$. \cite[Theorem 3.3.4]{wpv} tells us that $\msL_p:=\msL\vert_{\P_{\F_p}}\simeq\msO_{\P_{\F_p}}$ for any $p\nmid(a_0\dots a_r)$, so \cref{lem:triv-locus-closed} tells us that $\b{p\in\spec\Z:\msL_p\t{ trivial}}$ is a closed set containing the dense set of $p$ not dividing any $a_i$ and so is all of $\spec\Z$.
\end{proof}
\begin{cor}\label{cor:ds-hs}
    Let $V\subset\P(a_0,\dots,a_r)$ (with $\gcd(a_i)=1$) be a degree $d$ hypersurface over any field $k$. Then, $\omega_V\simeq\msO_V(d-a_0-\dots-a_r):=\msO_{\P(a_0,\dots,a_r)}(d-a_0-\dots-a_r)\vert_V$.
\end{cor}
\begin{proof}
    This now follows directly from adjunction \cite[Corollary (19)]{kleiman-duality}.
\end{proof}
\begin{lemma}\label{lem:H^1(O(n))=0}
    Consider $\P(1,2,1)$ over an arbitrary field $k$. For any $n\in\Z$, we have
    $$\hom^1(\P(1,2,1),\msO(n))=0.$$
\end{lemma}
\begin{proof}
    Write $\P(1,2,1)=\Proj k[X,Y,Z]$ with $X,Z$ in degree 1 and $Y$ in degree 2. Note that $\P^1\into\P(1,2,1)$ as the subscheme $Y=0$, so we have an exact sequence
    \[0\too\msO_{\P(1,2,1)}(-2)\too\msO_{\P(1,2,1)}\too\msO_{\P^1}\too0.\]
    The line bundle $\msO(2)$ on $\P(1,2,1)$ is ample, so Serre vanishing tells us that $\hom^1(\P(1,2,1),\msO(n+2k))=0$ for some $k\gg1$. We induct backwards to get the same conclusion when $k=0$. Twisting our short exact sequence by $n+2k$ and taking cohomology gives the exact sequence
    \[\hom^0(\P(1,2,1),\msO(n+2k))\to\hom^0(\P^1,\msO(n+2k))\to\hom^1(\P(1,2,1),\msO(n+2(k-1)))\to\hom^1(\P(1,2,1),\msO(n+2k))=0.\]
    The leftmost map above is easily seen to be surjective, so exactness gives $\hom^1(\P(1,2,1),\msO(n+2(k-1))=0$. Downwards induction then let's us conclude that $\hom^1(\P(1,2,1),\msO(n))=0$ as desired.
\end{proof}

\bibliographystyle{alpha}
\bibliography{citations}

\end{document}